\documentclass[11pt]{amsart}

\usepackage[margin=1in,footskip=.5in]{geometry}

\usepackage{etoolbox}
\makeatletter
\let\ams@starttoc\@starttoc
\makeatother
\usepackage[parfill]{parskip}
\makeatletter
\let\@starttoc\ams@starttoc
\patchcmd{\@starttoc}{\makeatletter}{\makeatletter\parskip\z@}{}{}
\makeatother

\usepackage{graphicx}
\usepackage{longtable}
\newcommand{\heart}{\ensuremath\heartsuit}

\renewcommand{\Pr}{\text{Pr}}

\DeclareFontFamily{U}{min}{}

\DeclareFontShape{U}{min}{m}{n}{<-> udmj30}{}

\usepackage{amsthm,aliascnt}
\usepackage{url}
\usepackage{caption}
\newcommand{\mynewtheorem}[2]{
  \newaliascnt{#1}{theorem} 
  \newtheorem{#1}[#1]{#2}
  \aliascntresetthe{#1}
  \expandafter\def\csname #1autorefname\endcsname{#2}
}

\theoremstyle{definition}

\mynewtheorem{thm}{Theorem}
\mynewtheorem{proposition}{Proposition}
\mynewtheorem{p}{Proposition}
\mynewtheorem{lemma}{Lemma}
\mynewtheorem{lem}{Lemma}
\mynewtheorem{corollary}{Corollary}
\mynewtheorem{cor}{Corollary}
\mynewtheorem{conjecture}{Conjecture}

\theoremstyle{definition}
\mynewtheorem{exercise}{Exercise}
\mynewtheorem{definition}{Definition}
\mynewtheorem{de}{Definition}
\mynewtheorem{example}{Example}
\mynewtheorem{conj}{Conjecture}
\mynewtheorem{cons}{Construction}
\mynewtheorem{notation}{Notation}
\mynewtheorem{assu}{Assumption}

\mynewtheorem{Example}{Example}
\mynewtheorem{remark}{Remark}
\mynewtheorem{rk}{Remark}

\newcounter{question}

\newcommand{\mynewtheoremA}[2]{
  \newaliascnt{#1}{question} 
  \newtheorem{#1}[question]{#2} 
  \aliascntresetthe{#1}
  \expandafter\def\csname #1autorefname\endcsname{#2}
}

\mynewtheoremA{Question}{Question}

\newcounter{researchgoal} 
 
\newcommand{\mynewtheoremB}[2]{
  \newaliascnt{#1}{researchgoal} 
  \newtheorem{#1}[researchgoal]{#2} 
  \aliascntresetthe{#1}
  \expandafter\def\csname #1autorefname\endcsname{#2 Section}
}

\mynewtheoremB{ResearchGoal}{Research Goal}

\renewcommand*{\do}[1]{\expandafter\newcommand \csname #1#1\endcsname{\mathbb{#1}}}
\docsvlist{B,C,D,E,F,G,H,I,J,K,L,N,O,P,Q,R,T,U,V,W,X,Y,Z}
\renewcommand{\AA}{\mathbb{A}}

\renewcommand*{\do}[1]{\expandafter\newcommand \csname f#1\endcsname{\mathfrak{#1}}}
\docsvlist{
a, A, b,B, c, D, e, E, f,F,g,h, j, k, C, l, K, S, N,  Y,  m, M 
n,o,p,q,r,s,t,u,v,w, X, 
gl, sl, T, so, su,z,Z, sp}

\renewcommand*{\do}[1]{\expandafter\newcommand \csname c#1\endcsname{\mathcal{#1}}}
\docsvlist{A,B,C,D,E,F,G,H,I,J,K,L,M,N,O,P,Q,R, RP, DP, S,T,U,V,W,X,Y,Z, GW}

\renewcommand*{\do}[1]{\expandafter\newcommand \csname
b#1\endcsname{\mathbf{#1}}}
\docsvlist{A,B,C,D,E,F,G,H,I,J,K,L,M,
N,O,P,Q,R, RP, DP, S,T,U,V,W,X,Y,Z, GW}

\newcommand{\DeclareMyOperator}[1]{\expandafter\DeclareMathOperator\csname #1\endcsname{#1}}
\newcommand{\DeclareMathOperators}{\forcsvlist{\DeclareMyOperator}}

\DeclareMathOperators{Isom, Aut,ch, KU, map,  End, Sym, Hom, Map, Res, Rep, Spec, CAT, Spa, Spd, Gr, PGL, MV, act, Shv, Sch} 

\renewcommand*{\do}[1]{\expandafter\newcommand \csname #1\endcsname{\text{#1}}}
\docsvlist{
Grp, LieGrp, poly, Surg, ds,  Moduli,Ge, DM, DTM, MTM,  std, PsR,  ess, Ramj,Ger,  sAb, lis, CS,  Tens, alm, cart, Term, Tm,  ccart, haus, CRng,Int, Auto, coeff,  ext, Drinf, Cor, LCat, Graph, Cdga,Nis, Bool, Stone,   art,  Gerb, col, Conv, Cur,  Cpt, light, relint, Reg, Hk, MM,Hdg,  Borel, RKE, GLike, JM, LP, WDP, aRep,Conf, KL, Kir, Av, CSS, Poset, Lax,  Gro,lax,  LLC, HShv, JH, coev, RMod, MW, perf, coeq, Fml, Rig, cHk, cWhit, pfp,  PHS, MHS, MHT, Quasi, AStk, CStk, CAni, CAStk, CSet, RTop, ASpec, desc, ARng, Berk, Rigid, Huber, Pair, afnd, Adic, Hck, sheafy, Dia, Flag, Disk, Isog, Dem, Tate, Afnd, Cpl, Dr, Rg, Ffgs,pf, dRg, dAf, dStk,   coAlg,  Dih, Liq, Prim, PSt, Groth, MT, LCH, Stacks, Anima, String, ccn, DAlg, CDGA, Tmf, tmf,RH,  GN, TM,  MString, Coinv, rs, Lin, Adm, Dieud, height, FGL, Stack, Dieu, Laf,  Arith, action, Nil, HP, TC, HN, Poly,basic,FIso, Nrd, cycl, Kt, asym,AVar, KeV, nm, Topoi, sqz, phase, stone, Cd, Corr, glob, indep, EC, Eis, cd, Toric,Sect, qi, dg,  cstr, Euc,  Semi,  bdry, Summable, Cont, Emb, Imm, Poiss, Proj,graph,  prop, torus, SH, Red,  disj,  Ran, dSch, Var,  SL, SO, DifDat, Rt, free, sch,  Ch, Spc,Alg, Mf, op, Top,cg,  Whit, PSh, quasi, Iso,RLP,  hol,RTF, Out,  SRTF, Pl, Mea,Sht,sph, CGWH, IC, Quillen, aHk, sHk, aShv, Funt,   Rie, GL, Sp, Mul, Geom, odia, SpSch,  Mono, CoCone, Com, Dist, LG, Par, cocompl, Weil, Higgs, LKE, PrSt, Cov, FWeil, KotTri,Virt,  dRing, red,Teich, TG, DMod,rlx,kir,  LS,Betti, fset, td, FV, Affinoid, AdicSpc, gas, Orb, fgt, Jet, Jets, adic,Div, elem, RHom, ctf, ob,  Nilp, Drag, Grass, Conn, crys,KTF, cp,  Jac,restr, Lisse, Fct, Vect,FAlg, Br, lisse, RO, Ana, sPShv, sShv, Obs, curve, DO,   trf, NE,  Artin, cond, Spl, qproj, Fun,Hilb, BCol, Path, Pos, coind,  Vel, Fl,  Time, Sol, EL, Adj, lc, Cat,Sober,  add, Mod,Pre, gplke, acus, cus, cocart, Algbrd,  Tors, qsplit, Mor, Eq, FT, arith, Band, Dynk, uf, TQFT,  sInd, Topos, Op,Irr,wild, PreStk,  supcus, tor, val, open, oblv, DK, CW,  Perv, Str, Bi, Cht, co, aut,  Crys, eff, eq, laur, GSp, Dec, coim,  mult, Para, Gerbe, Sel, photon,   gen, BG, coh,  AM,MG, mot, Dtb, Test, Msr, fix, rat,  AG, dense,Syl, spec, unit, counit, adm, scus, sqint, unip, cen, ndeg, Mass, JL,  Isoc,  Fin,Assoc, dAff, cdga,  geom, scn, irr,  Hil, irred, LC, ven, Comm, qspl, edst,  Car, ord, Bor, FMP, uni, RCl,  spl,  ChAff, holim, hocolim, SysLoc, Core, Supp,colim,  tdlc,AV,  prim,  aniso,STF,TF, naive, rad, BS, Spm, Pftd, enh, Dmod, VB, noet, integral,codim,  can,  cCAlg, GLke,  CDiv, CFact,   lcm, Ord, inc, Cone, Krull, depth, Ann,  HT, cris, trd, QLisse, coarse,  nrd, iso, Area, ST, tr, Sys,  Lie, isog,   typ, coMod, Hopf, Set, idem, fin, id, dia, Indx, der, Her,Sat, DGCat,   HSD, h,tf, Free, Tori, Cox, Grph, Hecke, str, Mp, RS, match, Autom, Four, Rad, dsr, St,vol,WD, ur, IH,  Cart,SET, Haus, nInd, GU, TO, Spin, PO, BC,  coCart,Tw, st, verd, kar, SpDM, rex, grp, lex, Un, BiFun,cofib, fib, Fm,TMF, BP,  CAlg,FPoly, sat, FFM, PreSp, seq, gpcpl, scts, cjg, Cts, MS, LT,  Ani, Mon, Th, BO, MO, odd, Index, exact, closed, Td, Pro, cof,  Pn, TwAr, Wald, Cr, He, Joyal, LMod, Triv, CT, GrAb,LocSys, pt, Bun, cusp, Pic,Loc, diff, filt, inj, sgn, GW, cInd, Yu,triv, Sh,  hyp, Frob, CM, univ,cores,  res, FG, Mfld, Stab, Span, sm,  QCoh, Zar, refl, type, ind, rec, cst, inv, ap, Inn, GLC, inl, inr, fd, Uni, car,  isconstr, Bord, framed, even, coinv, Sph,  SupFormMfd, ev, PShv, cts, reg, dR, Man, proj, Kan, Qui, AJ,  nr, Kl,  Nm, CNL, ab, PSL, pro, tame,  Latt,cyc, Spf,  Prop, funExt, isProp, isBiinv, Sq, whisker,  cha, Hmlgy, Simple, nGrp, ValSub, TotAbGrp, cls, GO, PS, Null, Tot, Spv, hgt, snRng, Frac, TopRng, Rng, mer, hor, level, KM,  snGrp,  cpt,  HS, trp, frm, Pet,  ori, spin, CGHaus, Inert, CHaus, ED, ms, aff, affine,
cn, Ind, ex,Ab, Met, Hyp, lag, NA, Ar, SymSpc,  IndCoh, LieAlg, dgLie, Ho, MC, Ext, Perf, Solid, sing, Mat, Cond, Ring, Efim, Mot, Nuc, CH, im, dlog, rig, coker, SCR, HC, Fil, gr, HKR, Pf, WhtStr, StrWht, Sm, pr, IW, CoMod, THH, nil, loc, FL,  Cent,WEIL,  FSet, FFMod, cl, Der, Aff, fpqc, Forms, fgp,  SU, Tr, SmProj, an, QC, Coh, Gpd, MU, MSO, non,  LCA, alg, An, Ell, BU,cone, norm, qc,extr,AW,  Ad, fact, ren,  pants, KO, MF, Bl,Li, dilog, pre, geo,  Fr,RM, Fix,  cpl, cplt, Cst,chern,Cplx,  DefP, fppf, Rel, supp, aug, Exc,Gal, temp, disc, Disc, trg,  ad, stk, Quot, forget, SCRMod, CMon, Stk,  sMod, csMod, CwF, Cph, sheaf, CwA, Ty, LCC, wcof, wfib, tors, unim, isom,sym, Diff, gw, gq,gs, bord,  Cob, BK, rel, met, Art, sep, Sing, gp, CoInd, length, elliptic, obs, EZ, Tor, narc,Wh, Conj, Cjg, nc, Hol, decomp, Irrep, ULA, Iw, Field, ins, 
isIso, isequiv, isContr, isEquiv,Equiv, Sub, Type, Cl, wCat, Tribe, DG, Prof, sh, AlgSp, AlgStk, Oblv, mon}

\renewcommand*{\do}[1]{\expandafter\newcommand \csname #1\endcsname{\mathrm{#1}}}
\docsvlist{tot, is-contr}

\usepackage[OT2,T1]{fontenc}
\DeclareSymbolFont{cyrletters}{OT2}{wncyr}{m}{n}
\DeclareMathSymbol{\Sha}{\mathalpha}{cyrletters}{"58}

\DeclareFontFamily{T1}{cbgreek}{}
\DeclareFontShape{T1}{cbgreek}{m}{n}{<-6>  grmn0500 <6-7> grmn0600 <7-8> grmn0700 <8-9> grmn0800 <9-10> grmn0900 <10-12> grmn1000 <12-17> grmn1200 <17-> grmn1728}{}
\DeclareSymbolFont{quadratics}{T1}{cbgreek}{m}{n}
\DeclareMathSymbol{\qoppa}{\mathord}{quadratics}{19}
\DeclareMathSymbol{\Qoppa}{\mathord}{quadratics}{21}

\newcommand{\bs}{\backslash} 
\newcommand{\noi}{\noindent}

\newcommand{\hra}{\hookrightarrow}

\DeclareMathSymbol{\mh}{\mathord}{operators}{`\-}

\newcommand{\smbr}[1]{  (  #1   ) }

\usepackage{etoolbox}
\makeatletter
\patchcmd\@thm
  {\let\thm@indent\indent}{\let\thm@indent\noindent}  {}{}
\makeatother

\usepackage{amsfonts, amsmath, amssymb, stmaryrd, color, enumerate, relsize, mathrsfs,array}
\usepackage{booktabs}
\newcommand{\Thick}{\mathrm{Thick}}
\newcommand{\last}{\operatorname{last}}

\usepackage{ragged2e} 

\numberwithin{equation}{section}

\usepackage{float}

\usepackage[backend=biber, backref=true,
style=alphabetic, doi=false
]{biblatex}

\DeclareNameAlias{author}{family-given}

\usepackage[dvipsnames]{xcolor}

\usepackage{setspace}

\usepackage{hyperref}
\hypersetup{colorlinks,linkcolor={blue},citecolor={blue},urlcolor={red}} 

\usepackage{dynkin-diagrams}

\usepackage{tikz-cd}

\tikzcdset{
  cells={font=\everymath\expandafter{\the\everymath\displaystyle}},
}

\usepackage{fullpage}

\usepackage[mathcal]{euscript}
\usepackage{verbatim}

\usepackage[nameinlink]{cleveref} 
\crefname{thm}{Theorem}{Theorems}
\crefname{lem}{Lemma}{Lemmas}
\crefname{prop}{Proposition}{Propositions}
\crefname{rem}{Remark}{Remarks}
\crefname{defe}{Definition}{Definitions}
\crefname{exam}{Example}{Examples}
\crefname{notation}{Notation}{Notations}

\theoremstyle{remark}

\let\oldmarginpar\marginpar
\renewcommand\marginpar[1]{\-\oldmarginpar[\raggedleft\footnotesize #1]{\raggedright\footnotesize #1}}

\expandafter\let\expandafter\oldproof\csname\string\proof\endcsname
\let\oldendproof\endproof
\renewenvironment{proof}[1][\proofname]{  \oldproof[\slshape #1.]}{\oldendproof}

\definecolor{darkred}{rgb}{0.5,0,0}  
\usepackage{enumitem}

\let\marginpar\oldmarginpar

\title{On a tamely ramified local relative Langlands conjecture via categorical representations}
\date{\today}

\begin{document}
\author{Milton Lin, Toan Pham, Jize Yu}
\address{Department of Mathematics, Northwestern University, Evanston, IL, U.S.A., and NSF-Simons National Institute for Theory and Mathematics in Biology, Chicago, IL, U.S.A.}
\address{Department of Mathematics, Johns Hopkins University, Baltimore, MD, U.S.A.}
\address{Department of Mathematics, Georgia Southern University, GA, U.S.A.}
\email{milton.lin@northwestern.edu}
\email{tpham45@jhu.edu}
\email{jizeyu@georgiasouthern.edu}

\begin{abstract}
    Let $G$ be a complex reductive group. For a smooth affine
    spherical $G$-variety
    $X$ such that the stabilizer of a point in the $B$-orbit under the $B$-action be connected, assume  
    that the unramified relative
    local Langlands 
    conjecture
    of \cite[Conjecture 7.5.1]{BZSV} for $X$ holds.  We establish
    a spectral description
    of a 
    full subcategory of 
    Iwahori equivariant coherent
    D-modules on $LX$
    in terms of the 
    relative Langlands 
    dual of $X$, 
    confirming the tamely ramified local relative Langlands conjecture proposed
    by Devalapurkar
    \cite[Conjecture 3.4.14.]{devalapurkar2024kutheoreticspectraldecompositionsspheres}. 
\end{abstract}
    
\maketitle

\tableofcontents 

\section{Introduction}

Sakellaridis-Venkatesh \cite{SV17} proposed a framework of the relative Langlands program, in which the local harmonic analysis of spherical $G$-varieties is studied using Langlands parameters of the Langlands dual group $\check{G}$, 
later extended in 
Ben--Zvi--Sakellaridis-Venkatesh
\cite{BZSV} in terms 
of dualities between 
graded Hamiltonian $G$-spaces
$M$ and $\check{G}$-spaces 
$\check{M}$, known as
relative Langlands duality. 

\subsection{The unramified local conjecture.}  The local harmonic analysis story admits 
a categorification for its unramified part \cite[Conjecture 7.5.1]{BZSV} in terms of 
this duality. Let $G$ be a complex reductive group and $X$ be a spherical $G$-variety see \autoref{def:polarized_hyperspherical}. 
Let $L^+G$, resp.\ $L^+X$, denote the positive loop {group}, resp.\ positive loop {space},
and let $LG$, resp.\ $LX$ denote the loop group resp.\ loop space. Denote by $\cD_c(\bullet)$  the stable $\infty$-category of \textit{coherent} $\cD$-modules, see \autoref{defn:coherent_D_mod}. The shearing functor $\sh^{1/2}$ is defined in \autoref{sec:shear}. 
\begin{conjecture}[\autoref{conj:relative_Langlands_unramified}]\label{Conj:intro}
 There is an equivalence of categories 
$$
\cD_c(L^+G\backslash LX)^{\Sat}\simeq \Perf(\sh^{1/2}(\check{M})/\check{G}),
$$
where the left hand side is the full subcategory of $\cD_c(L^+G\backslash LX)$ generated by the action of the spherical Hecke category $\cD_c(L^+G \backslash \Gr) $ on the unramified basic object $\IC_0$. 
This equivalence satisfies the \textit{pointing condition}:  $\IC_0$ is sent to $\cO_{\sh^{1/2}(\check{M})}$, and is compatible with the action of
derived geometric Satake equivalence   \cite{bezrukavnikov2008equivariant}.
\end{conjecture}

\subsection{Main result.} \label{subsec:applicable_cases} Beyond the unramified case, an Iwahori-level  local conjecture
was proposed by Devalapurkar
\cite[Conjecture~3.4.14]{devalapurkar2024kutheoreticspectraldecompositionsspheres}, following earlier formulations of \cite[Conjecture~1.4.1]{BFGT21},
 \cite[Conjecture~1.4.1]{BFT22},
and 
\cite[Conjecture~1.7.1]{TY23}\footnote{An Iwahori-level conjecture is also formulated  in
\cite[Conjecture~1.1.4]{travkin2026iwahorigaiottoconjecturecotangent},
which appeared after the first version of this paper. A vairant version of this conjecture is also proved in \textit{loc. cit}.}. Our  main theorem establishes   \cite[Conjecture 3.4.14]{devalapurkar2024kutheoreticspectraldecompositionsspheres} for a class of spherical varieties assuming the unramified local conjecture. 

Choose a Borel pair $T\subset B\subset G$ and let $I\subset L^+G$
be the Iwahori subgroup.
Let $\check T\subset\check B\subset\check G$ be the dual Borel pair
and let $\check{\mathfrak n}$ be the Lie algebra of the unipotent
radical of $\check B$. Let $\mathfrak g$ denote the Lie algebra of $G$ and write
$\check{\mathfrak n}^{\perp}\subset\check{\mathfrak g}^*$ for its
annihilator. Then
$\tilde{\check{\mathfrak g}}^*(2)
:=\check G\times^{\check B}\check{\mathfrak n}^{\perp}(2)$ is 
the Grothendieck--Springer resolution with weight two on the fibres over
$\check G/\check B$. Write $\check M$ for the  dual Hamiltonian space of $M=T^*(2)X$. Let $\cD_{{c}}(I\backslash LX)^\Sat$ be
   the full subcategory of $\cD_c(I\backslash LX)$ generated by the convolution of the affine Hecke category $\cD_c(I \backslash \Fl)$  on $r^!\mathrm{IC}_0$, where $r:I\backslash LX\longrightarrow L^+G\backslash LX$ is the natural projection.

\begin{thm} \label{thm:main_theorem} [\autoref{thm:iwahori_proof}]
Let $X$ be a spherical variety as in \autoref{def:polarized_hyperspherical} and assume that \autoref{Conj:intro} holds. Then there is an equivalence
    \begin{equation}\label{equation:main_thm}
        \LL^{\rm Sat}: \cD_{{c}}(I\backslash LX)^\Sat\simeq \mathrm{Perf}(\mathrm{sh}^{1/2}(\tilde{\check{\mathfrak{g}}}^*(2)\times^{\mathbf R}_{\check{\mathfrak{g}}^*(2)}\check{M})/\check{G}),
    \end{equation}
     satisfying
     \begin{enumerate}
         \item (Pointing condition.) There is an isomorphism \begin{equation}\label{eq:iwahori-pointing}
  \LL^{\rm Sat}(r^!\IC_0)
  \simeq
  \cO_{\sh^{1/2}  ( 
    \tilde{\check{\mathfrak g}}^*(2)
    \times_{\check{\mathfrak g}^*(2)}^{\mathbf R}
    \check M
    )},
\end{equation}
where $r:I\backslash LX\longrightarrow L^+G\backslash LX$ is the natural projection.
\item (Affine Hecke module structure.) The equivalence  $\mathbb{L}^\Sat$ is equivariant with respect to the action of 
    \begin{equation}\label{equation:algebra}
        \mathrm{End}_{\cD_c(L^+G\backslash LG/L^+G)}(\cD_c(I\backslash LG/L^+G))\simeq \mathrm{End}_{\Perf(\check{\mathfrak{g}}^*[2]/\check{G})}(\Perf(\tilde{\check{\mathfrak{g}}}^\ast[2]/\check{G})).
    \end{equation}
    In a forthcoming work of Bando--Gleason--Louren\c{c}o--Yu, we identify \eqref{equation:algebra} with the Bezrukavnikov equivalence \cite{bezrukavnikov2016two}.
     \end{enumerate}
\end{thm}

We give lists of cases where the main theorem directly applies and  those where it applies conditionally.

The main theorem holds unconditionally in the following cases. Let $H$ be a connected reductive group with the Langlands dual group $\check H$.
\begin{itemize}
\item \textit{The group case:}
$(G,X,\check M)
  = (
    H\times H,
    \operatorname{diag}(H)\backslash G,
    T^*(2)\check H
    )$. \autoref{Conj:intro} is the derived geometric Satake
\cite[Theorem 5]{bezrukavnikov2008equivariant}.

\item \textit{The trivial period:} $(G,X,\check M)
  =
   (H,\mathrm{pt},\check H\times\Sigma  )$,
where $\Sigma\subset\check{\mathfrak h}^*$ is a Kostant slice.
\autoref{Conj:intro} follows from Kostant--Whittaker reduction
\cite[Theorems 4--5]{bezrukavnikov2008equivariant},
see also \cite[\S7.6.3]{BZSV}.

\item \textit{Symplectic periods:} $(G,X,\check M)
  =
   (
    \GL_{2n},
    \Sp_{2n}\backslash\GL_{2n},
    \GL_{2n}\times^{\GL_n}\mathfrak{gl}_n
    )$ for $n\geq 1$.
\autoref{Conj:intro} is provided by
\cite[Theorem 1.3]{chen2022quaternionic}.

\item \textit{Linear periods:} $(G,X,\check M)
  =
   (
    \GL_{p+q},
    (\GL_p\times\GL_q)\backslash\GL_{p+q},
    \check M
    )$ for $1\leq p\leq q$.
Here $\check M$ is the affine graded Hamiltonian dual
constructed in
\cite[\S3.1]{chen2026relativelanglandsdualitylinear}.
\autoref{Conj:intro} is given by
\cite[Theorems 1.2.1 and 4.2.1(1)]
{chen2026relativelanglandsdualitylinear}.

\item \textit{Orthogonal Gan--Gross--Prasad.} $(G,X,\check M)
  =
   (
    \SO_{N-1}\times\SO_N,
    \SO_{N-1}^{\mathrm{diag}}\backslash G,
    (V_0\otimes V_1)(1)
    )$  for $N\geq 3$,
where $V_0$ is orthogonal and $V_1$ is symplectic.
\autoref{Conj:intro} is proved in
\cite[Theorem 2.2.1(a) and \S2.7]
{BFT_orthosymplectic_satake}.

\item \textit{Twisted loop symmetric spaces:} $   (G,LG/(LG)^\sigma,T^*(2)Y  )$.
Here $G$ is semisimple and simply connected, $Y=\check G/\check G^\theta$ is a quasi-split
symmetric variety and $LG/(LG)^\sigma$ is the associated
twisted loop space of \cite[\S1.1]{Chen_relative_Satake}.
\autoref{Conj:intro} follows from
\cite[Theorem 4.8(1)]{Chen_relative_Satake}.

\item \textit{Rankin--Selberg periods:}
$$ (G,X,\check M)= (\GL_{n-1}\times\GL_n,
       \GL_{n-1}^{\mathrm{diag}}\backslash G, (
        \operatorname{Hom}(\CC^{n-1},\CC^n)
        \oplus
        \operatorname{Hom}(\CC^n,\CC^{n-1})
         )(1)
      ) $$ for $n\geq2$.
\autoref{Conj:intro} follows from
\cites[Theorem 4.1.1]{BravermanFinkelbergGinzburgTravkin_2021}
      [Proposition 4.0.1]{TravkinYang_untwisted_Gaiotto}.
\end{itemize}

\autoref{thm:iwahori_proof} applies to the following cases conditionally. All varieties below satisfy the geometric conditions in
\autoref{def:polarized_hyperspherical}. We provide reference for the conditional \autoref{Conj:intro} in each case.

\begin{itemize}

\item \textit{Triple-product periods:}
$(G,X,\check M)
  =
   (
    \operatorname{PGL}_2^{\times3},
    \operatorname{PGL}_2^{\mathrm{diag}}\backslash G,
    ((\CC^2)^{\otimes3})(1)
    )$.
\autoref{Conj:intro} follows from
\cite[Theorem 1.2 and Remark 3.8]{Devalapurkar_triple_product}
under the hypothesis of
\cite[Hypothesis 3.5.2]
{devalapurkar2024kutheoreticspectraldecompositionsspheres}.

\item \textit{Rectangular matrix spaces:}
$(G,X,\check M)
  =
   (
    \GL_n\times\GL_m,
    \operatorname{Hom}(\CC^n,\CC^m),
    \GL_m\times S
    )$ for $1\leq n<m$.
Here $X$ carries the natural left--right action, and
$S=f+\mathfrak z_{\mathfrak{gl}_m}(e)$, where
$(e,h,f)$ is a principal $\mathfrak{sl}_2$-triple in
$\mathfrak{gl}_{m-n}\subset\mathfrak{gl}_m$.
\autoref{Conj:intro} follows from
\cite[Theorem 2 and \S\S2.2--2.4]{Fu_derived_Weil}
provided that Lysenko's functor specified in that theorem
is symmetric monoidal.

\item \textit{Rank-one homogeneous spaces:}
$(G,X)=$ 
$ (
  \operatorname{PGL}_{n+1},
  \GL_n\backslash G
  )$,
$ (
  \SO_{2n+1},
  \SO_{2n}\backslash G
  )$,
$ (
  \Sp_{2n},
  (\Sp_2\times\Sp_{2n-2})\backslash G
  )$ for $n\geq2$,
$ (
  F_4,
  \operatorname{Spin}_9\backslash G
  )$,
$ (
  G_2,
  \SL_3\backslash G
  )$ and
$ (
  \SO_7,
  G_2\backslash G
  )$.
The duals $\check M$ are described in
\cite[\S\S3.4 and 4]
{devalapurkar2024kutheoreticspectraldecompositionsspheres}.
\autoref{Conj:intro} follows under a hypothesis from
\cite[Theorem 3.4.16]
{devalapurkar2024kutheoreticspectraldecompositionsspheres},
see also
\cite[Remark 4.1.2]
{devalapurkar2024kutheoreticspectraldecompositionsspheres}.

\item \textit{Triality:}
$(G,X,\check M)
  =
   (
    \operatorname{PSO}_8,
    G_2\backslash G,
    \check M
    )$.
Here $\check M$ is the Whittaker-induced Hamiltonian
$\operatorname{Spin}_8$-variety described in
\cite[\S\S1 and 4]{Devalapurkar_triple_product}.
\autoref{Conj:intro} is proved conditionally by
\cite[Theorem 1.5]{Devalapurkar_triple_product}.

\item \textit{Lorentzian and octonionic symmetric spaces:}
$(G,X,\check M)
  =
   (
    \operatorname{PSO}_{2n},
    \SO_{2n-1}\backslash \operatorname{PSO}_{2n},
    \check M
    )$ for $n\geq3$ or
$ (
  E_6^{\mathrm{ad}},
  F_4\backslash E_6^{\mathrm{ad}},
  \check M
  )$.
Here $\check M$ is described in \cite[\S4]{BZSV}.
\autoref{Conj:intro} is proved conditionally in
\cite[Theorems 1.2 and 7.2]{Chen_OBrien_Lorentzian}.

\end{itemize}

\subsection{Unramified Satake category versus the unramified category} A natural question is that when canonical embedding $\cD_c(I\backslash LX)^\Sat\subseteq \cD_c(I\backslash LX)$ is an equivalence. As a first step toward answering it, we prove the following unramified level.   
\begin{thm}
(\autoref{thm:spherical_sat_full})
Let $X=G/G^{\theta}$ be a symmetric variety satisfying 
\autoref{assu:symmetric_condition_satfull}. There is an equivalence of categories
$$
\cD_c(L^+G\backslash LX)^{\Sat}=\cD_c(L^+G\backslash LX). 
$$
 Examples of symmetric varieties $X$ satisfying 
\autoref{assu:symmetric_condition_satfull}
are listed in \autoref{exam:generation}.  

\end{thm}

 The above theorem concerns the unramified level automorphic category, but as discussed in the next section, it shows that the Iwahori-level Satake category $\cD_c(I\backslash LX)^\Sat$ is in fact spherically generated. 
 In the companion paper 
 \cite{LinPhamYu_generation} we study when the Iwahori and spherical Satake subcategories are the whole categories using trace methods. 
 In this paper we prove that the Iwahori Satake subcategory is the whole category when $X=\AA^n$ and $G=\GL_n$; see \autoref{prop:affcase_categorical_generation}. 

\subsection{Strategy of the proof of 
the main \autoref{thm:iwahori_proof}}
The principal idea of the proof is to establish a categorical ``level raising'' type of result \autoref{thm:F_cstr_Sat}. More precisely we  reconstruct the full subcategory $\cD_c(I\backslash LX)^{\rm Sph, Sat}$ of $\cD_c(I\backslash LX)$ from the spherical Satake subcategory $\cD_c(L^+G\backslash LX)^\Sat$.
We then identify $\cD_c(I\backslash LX)^{\rm Sph, Sat}$ with the Iwahori Satake subcategory $\cD_c(I\backslash LX)^\Sat$ which appears on the automorphic side of the Iwahori conjecture. 

Let $H$ be a placid group ind-scheme and let $J\subset H$ be an affine group subscheme whose prounipotent radical has finite codimension.
For $\cC\in\cD^*(H)\text{-}\mathrm{mod}$, its category of $J$-invariants $\cC^J$ is naturally a left module over $\cH_J=\cD(J\backslash H/J)$, and the functor of taking $J$-invariants admits the left adjoint $\cD(H/J)\otimes_{\cH_J}(\bullet)$, see \eqref{eq:adjunction}.
If in addition $H/J$ is ind-proper,  \cite[Theorem 3.1.5]{campbell2021affine} allows us to reconstruct a full subcategory of $\cC$ from its $J$-invariants.
More precisely, the convolution 
$$
\cD(H/J)\otimes_{\cH_J}\cC^J\rightarrow\cC
$$
is a fully faithful functor of $\cD^*(H)$-modules, see \autoref{thm:CD_counit}. In particular, if $J=L^+G$, $H=LG$ we call $\cD(H/J)\otimes_{\cH_J}\cC^J$ the \textit{spherically generated part} of $\cC$.

Let $\Gr$ resp. $\Fl$ denote the affine Grassmannian resp. affine flag variety of $G$.   Adapting Campbell--Dhillon's theorem to our setting yields a fully faithful functor 
$$  F:\cD(I\backslash\Gr)\otimes_{\cD(L^+G\backslash\Gr)}\cD(L^+G\backslash LX)\rightarrow D(I\backslash LX),
$$
see \eqref{equation:I-equiv-functor}. The functor $F$ is defined on presentable categories; however, to prove \autoref{thm:main_theorem}, we need to pass to coherent $\cD$-modules.
Note that coherent equivariant objects need
\textit{not} be compact in the ordinary category, thus it does not follow
by taking compact objects in Campbell--Dhillon's theorem. In
\autoref{section:construction_of_the_functor} we use the
renormalized categories to construct the comparison functor
$$
\iota:\cD_c(I\backslash\Gr)\otimes_{\cD_c(L^+G\backslash\Gr)}\cD_c(L^+G\backslash LX) \rightarrow\cD(I\backslash\Gr)\otimes_{\cD(L^+G\backslash\Gr)}\cD(L^+G\backslash LX),
$$
see \eqref{functor:embedding_of_small_tensor}. 
 Here the source of $\iota$ is a relative tensor product of \textit{small} categories, which is technically difficult to work with. The full faithfulness of $\iota$ is proved by comparing
mapping complexes on generators in \autoref{thm:F_cstr_Sat}. This yields the equivalence
$$
F_c^\Sat:\cD_c(I\backslash\Gr)\otimes_{\cD_c(L^+G\backslash\Gr)}\cD_c(L^+G\backslash LX)^\Sat\simeq\cD_c(I\backslash LX)^{\rm  Sat},
$$
see \autoref{lem:satake-generators}.
We prove in \autoref{thm:spherical_sat_full} that for symmetric variety $X$  satisfying \autoref{assu:symmetric_condition_satfull}, there is an equivalence $\cD_c(L^+G\backslash LX)^\Sat=\cD_c(L^+G\backslash LX)$.
Hence $\cD_c(I\backslash LX)^\Sat\simeq \cD_c(I\backslash\Gr)\otimes_{\cD_c(L^+G\backslash\Gr)}\cD_c(L^+G\backslash LX)$ i.e. it is spherically generated.

To pass to the spectral side we use the integral transform of Ben--Zvi--Francis--Nadler \cite{ben2010integral}.
An integral transform is the categorical analogue of integration against a kernel:
a correspondence $Z \subset X\times_S Y$ equipped with a kernel $\cK$ encodes a
pull--push functor on the derived categories of bounded complexes of coherent sheaves on $X$ and $Y$. For perfect stacks $X$ and $Y$, \cite[Theorem 4.7]{ben2010integral}
identifies $\QCoh(X)\otimes_{\QCoh(S)}\QCoh(Y)$ with the
same category of kernels by exterior tensor product. Passing
to compact objects gives
\begin{equation}\label{eq:BZFN-kernel}
  \Perf(X\times_S^{\mathbf R}Y)
  \simeq
  \Perf(X)\otimes_{\Perf(S)}\Perf(Y).
\end{equation}

Via derived geometric Satake and its Iwahori-module
version
\cite{bezrukavnikov2008equivariant,arkhipov2004quantum,
gaitsgory-semiinfinite-intersection-coh} and \autoref{Conj:intro},  the source
of $F_c^\Sat$ is identified with a relative tensor product of spectral
categories. A sheared version of \eqref{eq:BZFN-kernel} (see \autoref{lem:spectral_relative_tensor}) completes the proof of the desired equivalence. The pointing condition and the affine Hecke module structure naturally follow from the above construction. 

\subsection{Organization.}
\Cref{section:conventions} collects the preliminaries on $\infty$-categories,  shearing, theory of $\cD$-modules and categorical actions.   Technical results are also established for later use. We recall the unramified relative
local Langlands conjecture in \Cref{section:satake_category}. We also define the various Satake
subcategories and prove that the spherical Satake subcategory equals the unramified coherent $\cD$-module category for symmetric
varieties satisfying \autoref{assu:symmetric_condition_satfull}. \Cref{section:construction_of_the_functor} is devoted to give a relative
tensor product description of
$\cD_c(I\backslash LX)^{\rm Sph, Sat}$.
In \Cref{section:satake_comparison} we identify this category with
the Iwahori Satake subcategory $\cD_c(I\backslash LX)^\Sat$.
We prove the main result  \Cref{thm:iwahori_proof} in \Cref{section:proof_of_the_main_theorem}.
Finally \Cref{section:affine_categorical_generation} treats the
standard action of $\GL_n$ on $\mathbb A^n$. We prove that the
Iwahori Satake subcategory is the whole category of coherent
equivariant $\cD$-modules and establish the corresponding generation
statement for functions over $\mathbb F_q((t))$.
\subsection{Notations}
\label{subsec:notations}

We use the following notation throughout the paper unless otherwise specified.
Write $F=\mathbb{C}((t))$ and $\cO=\mathbb{C}[[t]]$. Let $G$ be  a connected reductive group over  $\mathbb{C}$. Fix a Borel subgroup $B\subset G$
and a maximal torus $T\subset B$. We denote by $\mathbb X^*(T)$ and
$\mathbb X_*(T)$ the character and cocharacter lattices respectively.
Let $\Phi$ and $\check\Phi$ be the roots and coroots of $(G,T)$ and let
$\Phi^+$ and $\check\Phi^+$ be the positive roots and positive coroots
determined by $B$. We write $W:=N_G(T)/T$, 
$W_{\mathrm{ext}}:=\mathbb X_*(T)\rtimes W$ and 
$W_{\mathrm{aff}}:=\mathbb Z\check\Phi\rtimes W$
for the finite, extended affine and affine Weyl groups respectively.
Elements of $W_{\mathrm{ext}}$ are written $t^\lambda w$ with
$\lambda\in\mathbb X_*(T)$ and $w\in W$.

For a finite type scheme $X$ let $L^+X$ and $LX$ denote its
positive loop space and loop space respectively. We write $K=L^+G$ and
let $I\subset K$ be the inverse image of $B$ under evaluation at $t=0$.
Let $\Gr=LG/K$ and $\Fl=LG/I$ denote the affine Grassmannian and affine
flag variety of $G$ respectively.

Let $\check G$ denote the Langlands dual group of $G$.  Denote by 
$\check T\subset\check G$ the maximal torus dual to $T$ and
$\check B\supset\check T$ the Borel subgroup which chooses  $\check\Phi^+$ as positive roots.
Let
$\check{\mathfrak g}:=\mathrm{Lie}(\check G)$. Denote by $\check{\mathfrak b}$ and $\check{\mathfrak n}$  the Lie algebras
of $\check B$ and its unipotent radical respectively. Let
$\check{\mathfrak n}^{\perp}\subset\check{\mathfrak g}^*$ be the
annihilator of $\check{\mathfrak n}$ and $\tilde{\check{\mathfrak g}}^*
:=\check G\times^{\check B}\check{\mathfrak n}^{\perp}$ be the Grothendieck alteration.

For $M=T^*(2)X$ as in \autoref{def:polarized_hyperspherical}, let
$\check M$ denote the dual graded affine Hamiltonian $\check G$-variety
of \cite[\S 4]{BZSV} and write
$\mu_{\check M}:\check M\to\check{\mathfrak g}^*(2)$ for its moment map.

We write $\cD(\bullet)=\cD^!(\bullet)$ for the $!$-$\cD$-modules. For
a group ind-scheme $H$ we use the $*$-$\cD$-modules $\cD^*(H)$.
Write $\cD_c(\bullet)$ for
the small full subcategory of coherent $\cD$-modules in the sense of
\autoref{defn:coherent_D_mod}. Denote by
$\cD^\ren(\bullet):=\Ind(\cD_c(\bullet))$ the renormalized category
of $\cD$-modules, see \autoref{defn:renormalized_D_mod}.

For $J\in\{I,K\}$, we use notations $\cH_J:=\cD(J\backslash LG/J)$, 
$\cH_{c,J}:=\cD_c(J\backslash LG/J)$, and 
$\cH_J^\ren:=\cD^\ren(J\backslash LG/J)$.
\subsection{Acknowledgments}
This project originated during the workshop Advances in Representation Theory at Northeastern University in June 2025. We thank the organizers for their hospitality and support. We thank David Ben--Zvi for many helpful discussions related to \cite{BZSV} and Lingfei Yi for  explaining \cite{chenyi_2025singularitiesorbitclosuresloop} to us. We are  grateful to Justin Campbell and Xinwen Zhu for helpful discussions and comments and to Yiannis Sakellaridis for sustained interest and encouragement. We thank Sanath K. Devalapurkar, Taeuk Nam and David Yang for helpful discussions. We also thank Ruotao Yang and Roman Travkin for communicating their  ongoing work with us and  useful discussions.

Milton Lin was supported in part by grants from the NSF (DMS-2235451) and Simons Foundation (MPS-NITMB-00005320) to the NSF-Simons National Institute for Theory and Mathematics in Biology (NITMB).

\subsection{AI usage}
The first version of the paper in which the main theorem \autoref{thm:iwahori_proof} was proved was written by the 
authors without 
any AI assistance, except for reference checking. \\ 

\noi In our current version,  
we use Claude Opus 5.5 High and our own multiagent set up to generate ideas and proof strategies for \Cref{section:affine_categorical_generation}. The multiagent framework is a typical orchestrator-worker set up, \cite{s2024building}, with paired verifier-solver architecture, following Alethia, \cite{feng2026autonomousmathematicsresearch}. The run has resulted in further results and significant extensions of \cref{thm:spherical_sat_full}, which we do not include in this manuscript.  \\ 

\noi We use ChatGPT Astra 6 Ultra
to proofread, reorganize the manuscript, to rephrase statements, and to help with the literature review. ChatGPT Astra 6 Ultra found a mistake in the innitial proof of \autoref{lem:convolution_preserves_coherence}. It also helped us reorganize \autoref{section:conventions}: merging two subsections and reorder the subsections.
 The current paper is entirely written by the authors, and we take full responsibility for the final text, including its mathematical claims, proofs, and references.

\section{Conventions}\label{section:conventions}

Throughout this paper we will substantially use the theory of $\infty$-categories as developed in \cite{HTT,HA}. Let $\Ani$ denote the $\infty$-category of $\infty$-groupoids, or \textit{anima}. Given an $\infty$-category $\mathcal{C}$ and a pair of objects $c_1,c_2\in\mathcal{C}$, we let 
$$
\mathrm{Map}_{\mathcal{C}}(c_1,c_2)\in\Ani
$$
be the mapping
anima between them.  
Let $\widehat{\Cat}_\infty$ be the category of all
(not necessarily small) categories. We will frequently use the following subcategories of $\widehat{\Cat}_\infty$. Let us first describe the ``large'' setting. 

 Let $\smbr{\Pr^L,\otimes,\Ani}$ denote the subcategory of (not necessarily small) categories $\widehat{\Cat}_\infty$,  of presentable $\infty$-categories and colimit preserving functors. It inherits a symmetric monoidal structure, which we again abusively denote by $\otimes$, such that the inclusion is lax symmetric monoidal cf. \cite[Proposition 4.8.2.18]{HA}.

Let $\Pr^L_{\st} \hra \Pr^L$ be the subcategory of presentable stable categories with colimit preserving 
functors, which is \textit{lax} monoidal, cf.\ \cite[Proposition 4.8.2.18]{HA}, which we again abusively denote $\otimes$. 

Let $R$ be an $\EE_\infty$ ring spectrum then $\LMod_R \in \CAlg(\Pr^L)$, allowing us consider left module objects:  $$ \Lin\Cat_R:=\LMod_{\LMod_R}(\Pr^L) $$ 
is the $\infty$-category of $R$-linear categories. Following conventions of \cite{GR17vol1} (see also \cite{cohn2016differentialgradedcategoriesklinear}), we denote 
$$
\mathrm{DGCat}:=\Lin\Cat_k,
$$ whose unit object is 
$
\Vect:=\LMod_k.$ Let $\cC\in\mathrm{DGCat}$ and $c_1,c_2\in\cC$.
We denote by
$$
  \underline{\Map}_{\cC}(c_1,c_2)\in\Vect
$$
the mapping complex. Its underlying anima is
$\Map_{\cC}(c_1,c_2)$. 

\subsection{Compactness}
 We recall the notions of thick and localizing
generation and compare them under Ind-completion.
\begin{de}\label{Compact objects and Ind-completion}
Let $\cC$ be an $\infty$-category admitting small filtered colimits.
An object $c\in\cC$ \textit{compact} if the corepresentable
functor
$$
\Map_\cC(c,-):\cC\longrightarrow\Ani
$$
preserves small filtered colimits. We denote by
$\cC^\omega\subseteq\cC$ the full subcategory of compact objects.
\end{de}
For a small $\infty$-category $\cD$ its \textit{ind-completion}
$\Ind(\cD)$ admits small filtered colimits and is characterized
by the following universal property: for every $\infty$-category
$\cE$ admitting small filtered colimits precomposition with
$\cD\to\Ind(\cD)$ induces an equivalence
$$
\Fun^\omega(\Ind(\cD),\cE)
\xrightarrow{\simeq}
\Fun(\cD,\cE),
$$
where $\Fun^\omega$ denotes the full subcategory of functors
preserving small filtered colimits
\cite[Proposition 5.3.5.10]{HTT}.

\begin{de}\label{def:compactly_generated_category}
Let $\cC$ be a presentable $\infty$-category. Then $\cC$ is
\textit{compactly generated} if it satisfies the following
equivalent conditions:
\begin{enumerate}
\item Every object of $\cC$ is a colimit of a small filtered
diagram of compact objects.
\item The natural functor $\Ind(\cC^\omega)\longrightarrow\cC$
is an equivalence, where $\cC^\omega$ is essentially small.
\end{enumerate}
\end{de}
\begin{de}\label{def:stable_categories}
Let $\Pr^L_{\st,\omega}$ denote the category of compactly
generated stable presentable $\infty$-categories and
colimit-preserving functors that preserve compact objects.
\end{de}

\begin{definition}\label{def:small_relative_tensor}
${}$
\begin{enumerate}
\item Let $\Cat^{\perf}$ denote the category of small
idempotent-complete stable $\infty$-categories and exact functors.
The category $\Cat^{\perf}$ admits a symmetric monoidal
structure characterized by
$$
\cA\otimes\cB
\simeq
 (\Ind(\cA)\otimes\Ind(\cB))^\omega
,
$$
for any $\cA,\cB\in\Cat^{\perf}$, where the tensor product on the right is taken in
$\Pr^L_{\st}$. Ind-completion gives a symmetric monoidal
equivalence
$$
\Ind:\Cat^{\perf}\xrightarrow{\simeq}\Pr^L_{\st,\omega}
$$
with inverse $(\bullet)^\omega$.

\item Let $\cA$ be an associative algebra object in
$\Cat^{\perf}$,  $\cM$ be a right $\cA$-module and 
$\cN$ be a left $\cA$-module in $\Cat^{\perf}$. We define
$$
\cM\otimes_{\cA}\cN
:=
 (
\Ind(\cM)\otimes_{\Ind(\cA)}\Ind(\cN)
)^\omega.
$$

\item Let $R$ be an $\EE_\infty$-ring spectrum. Then
$\Perf(R)\in\CAlg(\Cat^{\perf})$, and we define
$$
\Cat^{\perf}_R
:=
\LMod_{\Perf(R)}(\Cat^{\perf}),
$$
the category of small $R$-linear idempotent-complete stable
$\infty$-categories
\footnote{The forgetful functor
$\Cat^{\perf}_R\to\Cat^{\perf}$ preserves finite limits and
colimits but is in general not symmetric monoidal.}.
\end{enumerate}
\end{definition}

In applications to the $k$-linear sheaf categories below we use
the $k$-linear versions of these constructions: tensor products
are taken in $\Cat^{\perf}_k$ in the small setting and in
$\mathrm{DGCat}$ in the presentable setting.

\subsection{Shearing}
\label{sec:shear}
In this section we recall the shearing functor from \cite{devalapurkar2024kutheoreticspectraldecompositionsspheres, BZSV}. Throughout the paper we adopt cohomological degrees,
i.e. for $M\in \Mod_k$
viewed as a chain complex whose 
$i$-th degree is denoted  $M^i$, define 
$(M[n])^i:= M^{i+n}$. 

Let $k$ be a field and let  $(\cC,\otimes,\mathbf{1})$ be a presentable stable symmetric monoidal
$k$-linear $\infty$-category whose tensor product preserves colimits
separately in each variable. Define
$\cC^{\gr}:=\Fun(\ZZ_{\ds},\cC)$, where $\ZZ_{\ds}$ is
$\ZZ$ regarded as a discrete symmetric monoidal category under
addition. By \cite[Example 2.2.6.17]{HA} $\cC^{\gr}$ carries
the Day-convolution symmetric monoidal structure
$$
  (M_\bullet\otimes N_\bullet)_w
  =\bigoplus_{u+v=w}M_u\otimes N_v.
$$

\begin{definition}\label{def:weight}
Let $V\in\cC$ and $n\in\ZZ$. We denote by $V(n)\in\cC^{\gr}$
the graded object supported in weight $n$ with component $V$.
\end{definition}

\begin{definition}\label{def:shearing}
We denote by \emph{$\frac12$-shearing} the $\EE_1$-monoidal
equivalence
\begin{align*}
    \sh^{1/2}:\cC^{\gr} \xrightarrow{\sim}\cC^{\gr},\quad
     (\sh^{1/2}M  )_w = M_w[w]
\end{align*}
where the monoidal structure is constructed in
\cite[Construction 2.1.1, Lemma 2.1.6]
{devalapurkar2024kutheoreticspectraldecompositionsspheres}.
We also define the \textit{ungraded} half-shearing functor
\begin{align*}
    \sh^{1/2}_{\mathrm{ungr}}
  :=\mathrm{oblv}_{\gr}\circ\sh^{1/2},\quad
  \sh^{1/2}_{\mathrm{ungr}}(M)=\bigoplus_{w\in\ZZ}M_w[w]
\end{align*}
where $\mathrm{oblv}_{\gr}((M_w)_w)=\bigoplus_wM_w$ 
see \cite[Notation 2.1.5]
{devalapurkar2024kutheoreticspectraldecompositionsspheres}.
\end{definition}

\begin{p}
\label{prop:shearing_is_symmetric_monoidal} Let $\phi_{M,N}:
  \sh^{1/2}(M)\otimes\sh^{1/2}(N)
  \simeq
  \sh^{1/2}(M\otimes N)$ be the isomorphism given by the monoidal structure of $\sh^{1/2}$ and $\tau_{M,N}:M\otimes N\simeq N\otimes M$ be the symmetry isomorphism in $\cC^{\gr}$. Then
\begin{enumerate}
    \item The restriction of $\sh^{1/2}$ to objects
    supported in even weights is symmetric monoidal.
\item If either $M$ or $N$ is supported in even
    weights, then
    $$
      \phi_{N,M}\circ
      \tau_{\sh^{1/2}(M),\sh^{1/2}(N)}
      \simeq
      \sh^{1/2}(\tau_{M,N})\circ\phi_{M,N}.
    $$
\end{enumerate}
\end{p}

\begin{proof}
The first statement immediately follows from the Koszul
sign rule c.f.\ \cite[Remark 2.1.9]
{devalapurkar2024kutheoreticspectraldecompositionsspheres}. The second assertion can be proved by restricting to the case that $M$ and $N$ are supported on single weights. The proof in this case is again a direct computation. 
\end{proof}

\begin{example}
\label{exam:shear_vector_space}
Assume that $k$ has characteristic zero and let $V$
be a finite-dimensional $k$-vector space.
Then  $V(n)$ has coordinate algebra
$\Sym^\bullet(V^\vee(-n))$, and by definition we have  an isomorphism of associative dg algebras $\sh^{1/2}_{\mathrm{ungr}}(\Sym^\bullet(V^\vee(-n)))\simeq\bigoplus_{j\geq 0}\Sym^j(V^\vee)[-nj]$.
For even $n$, this algebra is in fact commutative and identifies
with $\Sym^\bullet(V^\vee[-n])$.
For odd $n$ the latter commutative dg algebra instead has
underlying complex $\bigoplus_{j\geq 0}\bigwedge^j(V^\vee)[-nj]$ c.f.\ 
\cite[Remark 2.1.4]
{devalapurkar2024kutheoreticspectraldecompositionsspheres}.
\end{example}

\begin{definition}[Sheared spectral categories]
\label{def:sheared_spectral_categories}
    Let $H$ be an affine algebraic group, and  $Y$ be
a derived $H$-scheme  with a commuting $\GG_m$-action. Define the category
$$
\Perf(\sh^{1/2}(Y)/H)
$$ to be the full subcategory of compact
objects in the half-sheared category of quasi-coherent
sheaves on $Y/H$, see \cite[\S\S6.3--6.4]{BZSV}.  In particular, when $Y$ is an affine, there is an identification
$$
  \Perf(\sh^{1/2}(Y)/H)
  \simeq
  \mathrm{LMod}_{
    \sh^{1/2}_{\mathrm{ungr}}\mathcal O(Y)}
     (\QCoh(BH))^\omega.
$$
as in \cite[Example 6.3.7]{BZSV}.

\end{definition}

\begin{definition}[Sheared vector bundles]
\label{def:shifted_vector_bundles}
Let $n\in\ZZ$ and let $E$ be an $H$-equivariant
vector bundle on a classical $H$-scheme $Z$.
Define $E(n)$ to be the graded scheme having the ambient scheme the total space of $E$ with weight $n$ on the fibres
and trivial weight on $Z$. We use the notation
$$
  \Perf(E[n]/H)
  :=
  \Perf(\sh^{1/2}(E(n))/H).
$$
\end{definition}
Since $E(n)/H\to Z/H$ is affine and the grading
action on $Z/H$ is trivial, we have
$$
  \Perf(E[n]/H)
  \simeq
  (
    \mathrm{LMod}_{
      \bigoplus_{j\geq0}
        \Sym^j_{\mathcal O_Z}(E^\vee)[-nj]}
      (\QCoh(Z/H))
  )^\omega.
$$

\begin{definition}[Graded cotangent bundles]
\label{def:cotangent}
Let $j\in\ZZ$ and let $X$ be a smooth classical
scheme. We denote by $T^*(j)X$ the cotangent bundle
with weight $j$ on its fibres and trivial
$\GG_m$-action on $X$. 
\end{definition}

\begin{corollary}
\label{lem:shearing_even_base}
Let $R_0\to B_0$ and $R_0\to A_0$ be maps of
commutative algebras in $\cC^{\gr}$, and suppose that
$R_0$ and $B_0$ are supported in even weights.
Then $\sh^{1/2}_{\mathrm{ungr}}(R_0)$ and
$\sh^{1/2}_{\mathrm{ungr}}(B_0)$ are commutative,
$\sh^{1/2}_{\mathrm{ungr}}(A_0)$ is an associative
algebra in
$\mathrm{Mod}_{\sh^{1/2}_{\mathrm{ungr}}(R_0)}(\cC)$,
and there is a natural equivalence of associative algebras
$$
  \sh^{1/2}_{\mathrm{ungr}}
     (B_0\otimes^{\mathbf L}_{R_0}A_0)
  \simeq
  \sh^{1/2}_{\mathrm{ungr}}(B_0)
  \otimes^{\mathbf L}_{\sh^{1/2}_{\mathrm{ungr}}(R_0)}
  \sh^{1/2}_{\mathrm{ungr}}(A_0).
$$
This equivalence is compatible with continuous
symmetric monoidal functors applied weightwise.
In particular it is compatible with $f^*:\QCoh(Z/H)\longrightarrow\QCoh(Z'/H')$ where $f:Z'/H'\to Z/H$ is a morphism of quotient stacks  with
trivial weights.
\end{corollary}

\begin{proof}
Note that $\sh^{1/2}_{\mathrm{ungr}}(R_0)$ and $\sh^{1/2}_{\mathrm{ungr}}(B_0)$ are both commutative follows directly from \autoref{prop:shearing_is_symmetric_monoidal}. Again, by \autoref{prop:shearing_is_symmetric_monoidal}, $\sh^{1/2}_{\mathrm{ungr}}(A_0)$ is an associative algebra in
$\mathrm{Mod}_{\sh^{1/2}_{\mathrm{ungr}}(R_0)}(\cC)$.

Since $\sh^{1/2}_{\mathrm{ungr}}$ is monoidal and preserves colimits, we have
$$
  \sh^{1/2}_{\mathrm{ungr}} (B_0\otimes R_0^{\otimes q}\otimes A_0  )
  \simeq
  \sh^{1/2}_{\mathrm{ungr}}(B_0)\otimes \sh^{1/2}_{\mathrm{ungr}}(R_0)^{\otimes q}\otimes \sh^{1/2}_{\mathrm{ungr}}(A_0).
$$
Taking geometric realizations gives the claimed
equivalence of associative algebras.
Finally the above construction is natural under continuous symmetric
monoidal functors. We thus prove the final assertion.
\end{proof}

\subsection{Generation}
\label{compact_generation}
We pass between small and large categories. 

\subsubsection{Generators} In this section, we recall the notions of generation for small and presentable
stable categories. Their distinction is particularly relevant for the discussion of coherent
$\cD$-modules in \autoref{Section:sheaf theory}.
We also recall a criterion for checking the full faithfulness
of exact functors 
which will be used in the proof of \autoref{thm:F_cstr_Sat}.

\begin{definition}[Generations for small vs large categories]\label{def:generation}\label{def:generation_category}
Let $\cC$ be a stable $\infty$-category and $S$ a set of objects of $\cC$.
\begin{enumerate}
\item[(T)] (\textit{Thick span \cite[Definition 1.1.3.2]{HA}}.) 
Define
$\Thick_\cC(S)$ to be the smallest stable subcategory of $\cC$  that contains $S$ and is closed under retracts. Equivalently it is  the smallest full subcategory of $\cC$ that contains $0$ and $S$ and is closed under cofibres, shifts $X\mapsto X[\pm1]$ and retracts by \cite[Lemma 1.1.3.3]{HA}. 
\item[(L)] (\textit{Localizing span.}) For presentable $\cC$, define $\langle S\rangle_\cC$ to be the smallest full stable subcategory of $\cC$ that contains $S$ and is closed under small colimits computed in $\cC$. We say that $S$ \textit{generates} $\cC$ if $\langle S\rangle_\cC=\cC$\footnote{See \cite[Ch. 1, Proposition 5.4.5]{GR17vol1} for an equivalent characterization.}.
\end{enumerate}
\end{definition}

\begin{remark}
Even with the same set of generators, the two spans in \autoref{def:generation} can provide different categories. Let $R$ be an $\EE_1$-ring. By
\cite[Definition 7.2.4.1 and Proposition 7.2.4.2]{HA},
the thick and localizing spans of the regular left
$R$-module are given by
$$
  \Thick_{\LMod_R}(\{R\})
  =\LMod_R^{\perf}
  =(\LMod_R)^\omega,
  \quad
  \langle\{R\}\rangle_{\LMod_R}=\LMod_R,
$$
respectively.
In particular, let $R=k$ be a field. Then the thick span of $k$
consists of the complexes with finite-dimensional total
cohomology; while the localizing span also contains the infinite
direct sum $\bigoplus_{n\geq0}k$ concentrated in degree zero.
\end{remark}

\begin{lemma}\label{rem:thick_properties}
Let $\cC,\cC'$ be stable $\infty$-categories. For $F:\cC\to\cC'$ an exact functor and $S$ a set of objects of $\cC$.
\begin{enumerate}
\item[(1)] If $A\to B\to C$ is a cofibre sequence in $\cC$  with $A,C\in\Thick_\cC(S)$, then $B\in\Thick_\cC(S)$.
\item[(2)]  We have $F\big(\Thick_\cC(S)\big)\subseteq\Thick_{\cC'}\big(F(S)\big)$.
\item[(3)] Suppose that $\cC=\Thick_\cC(S)$. If the map
$$
  \Hom_{h\cC}(s,t[n])
  \longrightarrow
  \Hom_{h\cC'} (F(s),F(t)[n]  )
$$
is an equivalence for all $s,t\in S$ and $n\in\ZZ$, then
$F$ is fully faithful.
For $\cC$, $\cC'$ and $F$ to be $k$-linear, the equivalence above is
equivalent to requiring 
$$
  \underline{\Map}_{\cC}(s,t)
  \longrightarrow
  \underline{\Map}_{\cC'} (F(s),F(t)  )
$$
to be equivalences for all $s,t\in S$.

\item[(4)] If $F$ is fully faithful and $\cC$ is idempotent
complete, then the essential image of $F$ is thick in $\cC'$.
If in addition $\cC=\Thick_\cC(S)$, the essential image is
$\Thick_{\cC'}(F(S))$.
\end{enumerate}
\end{lemma}

\begin{proof}
(1) By \cite[Theorem 1.1.2.14]{HA} the cofibre sequence  gives a distinguished triangle $C[-1]\to A\to B\to C$.  Hence $B\simeq\cofib(C[-1]\to A)$ and belongs to $\Thick_\cC(S)$.

(2) Let $\cC_0\subseteq\cC$ be the full subcategory spanned by the objects $X$ with $F(X)\in\Thick_{\cC'}(F(S))$. Since $F$ preserves zero objects and finite colimits \cite[Proposition 1.1.4.1]{HA}, it preserves cofibres and shifts and $\cC_0$ is a stable subcategory. Since $F$ is exact, $\cC_0$ is closed under retracts. Hence $\Thick_\cC(S)\subseteq\cC_0$.

(3) Fix $s\in S$. By exactness the objects $d\in\cC$ such that
$$
  \Hom_{h\cC}(s,d[n])
  \longrightarrow
  \Hom_{h\cC'}(F(s),F(d)[n])
$$
is an isomorphism for every $n\in\ZZ$ form a thick subcategory
containing $S$. It is clear that subcategory equals $\cC$.
Now fix $d\in\cC$ and apply the same argument in the first
variable. Thus $F$ induces isomorphisms on all shifted Hom
groups and hence equivalences on mapping anima.

(4) Full faithfulness allows morphisms between objects
in the image to be lifted to $\cC$, so exactness makes the
image stable. Idempotent completeness of $\cC$ makes the
image closed under retracts. The final assertion follows from
(2) and the fact that the image contains $F(S)$.
\end{proof}

\begin{proposition}\label{lem:spans_compare}
Let $\cC\in\Cat^{\perf}$ and
$\widehat\cC:=\Ind(\cC)$. For $S$ a small set of
objects in $\cC$, the inclusion
$\Thick_\cC(S)\hookrightarrow\cC$ induces an equivalence
$$
  \Ind (\Thick_\cC(S)  )
  \xrightarrow{\simeq}
  \langle S\rangle_{\widehat\cC}.
$$
Identifying these categories with their essential images in
$\widehat\cC$, we have
$$
  \Thick_\cC(S)
  = (\langle S\rangle_{\widehat\cC}  )^\omega
  =\cC\cap\langle S\rangle_{\widehat\cC}.
$$
\end{proposition}

\begin{proof}
Let $\cC_0=\Thick_\cC(S)$. Since $\cC$ is small and
idempotent complete, so is $\cC_0$.
Every object of $\cC_0$ is compact in $\widehat\cC$.
The exact fully faithful inclusion $\cC_0\hookrightarrow\cC$
induces a fully faithful colimit-preserving functor
$$
  J:\Ind(\cC_0)\longrightarrow\widehat\cC;
$$
see \cite[Proposition 5.3.5.11(1)]{HTT} and
\autoref{def:small_relative_tensor}.
Its essential image is stable, closed under small colimits
and contains $S$, so it contains
$\langle S\rangle_{\widehat\cC}$.
Conversely a localizing subcategory is closed under retracts
since a retract is the sequential colimit of its associated
idempotent. Therefore $\langle S\rangle_{\widehat\cC}$ contains
$\cC_0$ and all filtered colimits of its objects. Hence it
contains the essential image of $J$.

The compact objects of $\Ind(\cC_0)$ are exactly $\cC_0$,
and $J$ sends them to compact objects of $\widehat\cC$.
On the other hand, an object of its image which is compact in
$\widehat\cC$ is compact in $\Ind(\cC_0)$ since $J$ is
fully faithful and preserves filtered colimits.
The equality $\widehat\cC^{\,\omega}=\cC$ gives the desired
equalities.
\end{proof}

\begin{corollary}\label{cor:thick_vs_generation}
For $\cC\in\Cat^{\perf}$ we write $S^{\perp_{\widehat\cC}}
  :=\{X\in\widehat\cC:
       \Hom_{h\widehat\cC}(s,X[n])=0
       \text{ for all }s\in S,\ n\in\ZZ\}$.       
       For any small set $S$  of objects of $\cC$, the followings are equivalent:
       \begin{enumerate}
           \item $\Thick_\cC(S)=\cC$,
           \item $\langle S\rangle_{\widehat\cC}=\Ind\cC$,
           \item $S^{\perp_{\widehat\cC}}=0$.
       \end{enumerate}
\end{corollary}
\begin{proof}
The equivalence between $(1)$ and $(2)$ follows from
\autoref{lem:spans_compare}. We prove $(2)$ is equivalent to $(3)$. For any $X\in S^{\perp_{\widehat\cC}}$, note that  the functor 
$\mathrm{Map}_{\widehat\cC}(\bullet,X[n])$ carries colimits to limits for any $n\in\ZZ$, we have $(2)$ implies $(3)$. Conversely, the colimit-preserving inclusion
$\iota:\cE\hookrightarrow\widehat\cC$ has a right adjoint
$R$ by \cite[Corollary 5.5.2.9]{HTT}.
For $X\in\widehat\cC$, let $Q$ be the cofibre of the counit $\iota RX\rightarrow X$.
The full faithfulness of $\iota$ yields an isomorphism
$$
  \Hom_{h\widehat\cC}(s,(\iota RX)[n])
  \xrightarrow{\sim}
  \Hom_{h\widehat\cC}(s,X[n]),
$$
for all $s\in S$ abnd $n\in \ZZ$.
Thus $Q\in S^{\perp_{\widehat\cC}}=0$ and we prove $(2)$.
\end{proof}

\subsection{Module categories}
We adopt the definitions from \cite{HA}. In this subsection
a monoidal category means a unital associative algebra object
in $\mathrm{DGCat}$. For a given monoidal category $\cA$ we have the corresponding notions of left and right module categories (cf.\ \cite[Sections 4.2, 4.3]{HA}). We write $\mathrm{LMod}_\cA$ (resp. $\mathrm{RMod}_\mathcal{A}$) the category of left (resp. right) $\mathcal{A}$-module objects in $\mathrm{DGCat}$. Their morphisms are colimit-preserving module functors.
We also use their natural $(\infty,2)$-categorical enhancements,
whose $2$-morphisms are module natural transformations.
For left $\cA$-module categories $\cM,\cN$ we write
$$
\Hom_{\LMod_\cA}(\cM,\cN)
:=
\Fun^L_\cA(\cM,\cN)
$$
for the category of colimit-preserving $\cA$-linear functors;
the corresponding convention applies to right modules. There is a canonical identification $$
\mathrm{LMod}_\cA\simeq \mathrm{RMod}_{\mathcal{A}^\mathrm{rev}},
$$
where $\mathcal{A}^{\mathrm{rev}}$ is the monoidal category $\mathcal{A}$ with reversed multiplication c.f.\ \cite[Chapter 1, 3.1.4]{GR17vol1}. 

For monoidal categories $\mathcal{A}$ and $\mathcal{B}$ we denote by $_{\mathcal{A}}\mathrm{BMod}_\mathcal{B}$ the category of ($\mathcal{A},\mathcal{B}$)-bimodules. There is a similar identification 
$$
_{\mathcal{A}}\mathrm{BMod}_\mathcal{B}\simeq \mathrm{LMod}_{\mathcal{A}\otimes \mathcal{B}^{\mathrm{rev}}}.
$$
Let $\mathcal{M}\in {_{\mathcal{A}}}\mathrm{BMod}_\mathcal{B}$,
there is a functor
\begin{equation}
    \mathcal{M}\otimes_\mathcal{B}(\bullet):\mathrm{LMod}_{\mathcal{B}}\rightarrow \mathrm{LMod}_{\mathcal{A}}
\end{equation}
given by Lurie's relative tensor product cf. \cite[Section 4.8.1]{HA}. Its underlying category is the geometric realization in
$\mathrm{DGCat}$ of the two-sided bar construction
$$
\cM\otimes_\cB\cN
\simeq
\left|
[q]\longmapsto
\cM\otimes\cB^{\otimes q}\otimes\cN
\right|.
$$
The map from degree zero is the insertion functor
$$
\mathrm{ins}:\cM\otimes\cN
\longrightarrow\cM\otimes_\cB\cN.
$$

\subsection{Compact objects in relative tensor products} We will use the following result to perform renormalization of sheaf categories.  Let $\cA$ be a stable monoidal category and $\cM$, resp. $\cN$, be a right resp. left $\cA$-module. 

\begin{proposition}(\cite[Chapter 1, Corollary 8.7.4]{GR17vol1})\label{prop:compact_objects_in_relative_tensor}
   Assume that $\cA$, $\cM$ and $\cN$ are compactly generated and the functors 
    $$
\cA\otimes \cA\rightarrow \cA,\ \cM\otimes \cA\rightarrow \cM,\ \cA\otimes \cN\rightarrow \cN
    $$
    preserve compact objects. Then the insertion functor 
    $$
    \mathrm{ins}:\cM\otimes \cN\rightarrow \cM\otimes_\cA\cN
    $$
    preserves compact objects and $\cM\otimes_\cA\cN$ is compactly generated by objects of the form $\mathrm{ins}(m\boxtimes n)$ for $m\in\cM^\omega$ and $n\in \cN^\omega$.
\end{proposition}
In this paper the following special case of the above proposition will be useful to us.

\begin{corollary}(\cite[Corollary 9.3.3]{GR17vol1})\label{coro:compact_objects_in_tensor_rigid}
Let $\cA$ be rigid stable monoidal category, and $\cM$, resp. $\cN$, is a right, resp. left, $\cA$-module. Then the insertion functor
 $$
    \mathrm{ins}:\cM\otimes \cN\rightarrow \cM\otimes_\cA\cN
    $$
    preserves compact objects. If $\cA$, $\cM$ and $\cN$ are compactly generated, 
$\cM\otimes_\cA\cN$ is compactly generated by objects of the form $\mathrm{ins}(m\boxtimes n)$ for $m\in\cM^\omega$ and $n\in \cN^\omega$. Moreover the right adjoint of $\rm ins$ is continuous. 

\end{corollary}

\subsection{Sheaf theory} \label{Section:sheaf theory}  Throughout the paper we work with $\cD$-modules. 
To make our exposition self-contained we briefly summarize   the machinery developed in \cite{beraldo2017loop, raskin_dmodules_infinite_var} and refer the reader to \textit{op. cit.} for more details.

Let $\mathrm{AffSch}$ be the $(1,1)$-category of classical affine scheme over $k$ and $\mathrm{AffSch}^{f.t.}$ its full subcategory of finite type affine schemes. Denote by $\mathrm{PreStk}
:=\mathrm{Fun}(\mathrm{AffSch}^{\rm op},\Ani)$ the category of classical prestacks. We recall the definition of $\cD$-modules:
\begin{definition}[{\cite[\S\S~2.3, \S\S~2.5]{raskin_dmodules_infinite_var}}]
We define two functors
\begin{align*}
  \cD^!&:\mathrm{PreStk}^{\rm op}\longrightarrow\mathrm{DGCat},\\ 
  \cD^*&:\mathrm{PreStk}\longrightarrow\mathrm{DGCat}
\end{align*}
such  that 

\begin{enumerate}
\item For $S\in \mathrm{AffSch}^{f.t.}$, $\cD^?(S):=\cD(S)$ for $?=!$ and $*$.
\item A morphism $f:S_1\to S_2$ in $\mathrm{AffSch}^{f.t.}$ induces
the $!$-pullback $f^!:\cD(S_2)\to\cD(S_1)$. We define $\cD^!$
by left Kan extension of the resulting functor
$(\mathrm{AffSch}^{f.t.})^{\rm op}\to\mathrm{DGCat}$ to
$\mathrm{AffSch}^{\rm op}$ followed by right Kan extension to
$\mathrm{PreStk}^{\rm op}$.

\item A morphism $f:S_1\to S_2$ in $\mathrm{AffSch}^{f.t.}$ induces
the de Rham pushforward $f_*:\cD(S_1)\to\cD(S_2)$. We define $\cD^*$
by right Kan extension of the resulting functor
$\mathrm{AffSch}^{f.t.}\to\mathrm{DGCat}$ to $\mathrm{AffSch}$
followed by left Kan extension to $\mathrm{PreStk}$.
\end{enumerate}
\end{definition}

Unless otherwise stated, we write $\cD(\cY):=\cD^!(\cY)$
for a prestack $\cY$. When a group ind-scheme $H$ acts on $\cY$,
we write $\cD(H):=\cD^*(H)$ for the convolution category of
the acting group. 

In this paper we will mostly work with categories of $\cD$-modules on the following infinite dimensional varieties:

\begin{definition}[Placidness]\label{placid_ind_scheme}
\begin{enumerate}
    \item A scheme or algebraic space $Y$ admits a
\textit{placid presentation} if $Y\simeq\varprojlim_{j\in J^{\rm op}}Y_j$,
where $J$ is filtered,  $Y_i$ is of finite type,
and the transition maps are smooth and affine.
We call $Y$ \textit{placid} if it admits an \'etale
covering by a scheme which admits such presentations.
\item An ind-scheme $Y$ is \textit{placid} if it admits a filtered colimit presentation by quasi-compact quasi-separated placid
schemes with finitely presented closed embeddings transition maps.
If a placid group scheme $H$ acts on $Y$, we call $Y$
\textit{$H$-placid} if such a presentation can be chosen
with every stage $H$-stable.
\item For placid schemes, we use the \'etale-local convention of
\cite[Appendix A.1]{chenyi_2025singularitiesorbitclosuresloop}.
In particular, it allows non-surjective transition maps\footnote{The definition
in \cite[\S 4.2]{raskin_dmodules_infinite_var} requires
a global presentation with smooth affine surjective
transition maps.}. 
\end{enumerate}

\end{definition}

\begin{example}[Loops and positive loops are placid]
\label{ex:loop_placid_pole_bounds}
Let $X$ be a smooth affine scheme of finite type over $k$.
For $n\geq 1$, the functor $L^nX(R):=X(R[t]/(t^n))$
on ordinary commutative $k$-algebras is represented by
an affine scheme of finite type. The canonical quotient
$L^{n+1}X\to L^nX$ is smooth and affine, so
$L^+X\simeq\varprojlim_n L^nX$
is a placid presentation. Choose a closed embedding $X\hookrightarrow V$
into a finite-dimensional vector space.
When $X$ is a $G$--space, we choose $V$
to be a $G$-representation and an $G$-equivariant embedding.
By \cite[Theorem 6.3]
{drinfeld_infinite_vectorbundle} and \cite[Theorem 35]{chenyi_2025singularitiesorbitclosuresloop}
each $X^{(m)}:=LX\cap t^{-m}L^+V$ is an affine placid scheme for $m\geq0$. Moreover we have $LX\simeq\varinjlim_{m\geq0}X^{(m)}$.
It is therefore a finitely presented closed
embedding. In the equivariant setting, this presentation makes $LX$
an $L^+G$-placid ind-scheme and an $I$-placid
ind-scheme.
\end{example}

\begin{Example}\label{ex:sheaf_on_placid_scheme}
Let $X\simeq\varprojlim_{i\in I^{\rm op}}X_i$
be a filtered placid presentation. Write
$p_{ji}:X_j\to X_i$ for the transition maps.
Then we have equivalences $\cD^*(X)\simeq\varinjlim_{i\in I}\cD(X_i)$ and $\cD^!(X)\simeq\varprojlim_{i\in I^{\rm op}}\cD(X_i)$, where the transition functors are $p_{ji}^*$ and
$(p_{ji}^!)^R$ respectively.
Both categories are compactly generated and are
canonically dual: $\cD^!(X)\simeq\cD^*(X)^\vee$.
\end{Example}

\begin{Example}\label{ex:sheaf_placid_ind}
Let $X$ be a placid ind-scheme with a filtered placid presentation
$X=\varinjlim_{i\in I}X_i$, where
each $X_i$ is a quasi-compact quasi-separated placid
scheme and the transition maps
$u_{ij}:X_i\hookrightarrow X_j$ are finitely presented
closed embeddings. By \cite[\S 6.5]{raskin_dmodules_infinite_var}, we have equivalences
$$
\cD^!(X)
\simeq
\varprojlim_{i\in I^{\rm op},\,u_{ij}^!}\cD^!(X_i)
\simeq
\varinjlim_{i\in I,\,(u_{ij}^!)^L}\cD^!(X_i)
$$
and
$$
\cD^*(X)
\simeq
\varinjlim_{i\in I,\,u_{ij,*}}\cD^*(X_i)
\simeq
\varprojlim_{i\in I^{\rm op},\,(u_{ij,*})^R}\cD^*(X_i).
$$
Here $(u_{ij}^!)^L$ denotes the left adjoint of
$u_{ij}^!$, and $(u_{ij,*})^R$ denotes the right
adjoint of $u_{ij,*}$. 
These
adjoints exist and are continuous because the transition morphisms $u_{ij}$ are finitely presented. By \cite[\S 6.15]{raskin_dmodules_infinite_var} a chosen dimension theory $\tau$ on $X$ identifies the $!$ and $*$ theory: $E_X^\tau:\cD^!(X)\xrightarrow{\sim}\cD^*(X)$. When we use $E_X^\tau$, we implicitly choose a dimension theory.
\end{Example}

For any two reasonable ind-schemes $Y$ and $Z$, external
product induces a canonical equivalence
$$
\cD^!(Y)\otimes\cD^!(Z)
\xrightarrow{\sim}\cD^!(Y\times Z)
$$
by \cite[Equation (6.5.2)]
{raskin_dmodules_infinite_var}.  If $Y$ admits a left action by a group scheme $H$, the action map $a:H\times Y\to Y$
induces the coaction
$$
a^!:\cD^!(Y)\longrightarrow
\cD^!(H)\otimes\cD^!(Y).
$$
For  $H=LG$, $L^+G$, or $I$,  
the duality
$\cD^!(H)\simeq(\cD^*(H))^\vee$  
identifies $a^!$
with the action of $\cD^*(H)$
$$
\begin{aligned}
  \cD^!(Y)\otimes\cD^*(H) &\longrightarrow \cD^!(Y).
\end{aligned}
$$
 We define the $\cD$-modules on the quotient stack as
\begin{equation}\label{eq:geometric_realization_d_mod}
\cD(H\backslash Y)
:=
\mathrm{Tot}_{[n]\in\Delta}
\cD^!(H^n\times Y),
\end{equation}
with transition functors given by $!$-pullback.

\begin{definition}[Iwahori and spherical orbits on the loop space]
\label{def:orbits_iwahori_spherical}
Let $X$ be a smooth affine spherical $G$-variety. We denote the
$I$-orbits in $LX$ by $\{\cO_v\}_{v\in W_{X,\ext}}$ and the
$L^+G$-orbits in $LX$ by $\{LX_\lambda\}_{\lambda\in\Lambda_X^+}$
\footnote{In the group case $X=H$ with $G=H\times H$ acting by
left and right multiplication, $W_{X,\ext}$ identifies with the
extended affine Weyl group of $H$ and $\Lambda_X^+$ with the set
of dominant coweights of $H$.}.
Write $\overline{\cO}_v$ and $\overline{LX}_\lambda$ for the
corresponding reduced orbit closures in $LX$. We define the
closure relations on $W_{X,\ext}$ and $\Lambda_X^+$ by $v\leq w$ resp. $\lambda\leq\mu$ if and only if $\cO_v\subseteq\overline{\cO}_w$ resp. $LX_\lambda\subseteq\overline{LX}_\mu$. We use the following notations for morphisms
\begin{align*}
  i_v\colon \cO_v
  &\xrightarrow{j_v} \overline{\cO}_v
   \xrightarrow{\overline{i}_v} LX,\\
  i_\lambda\colon LX_\lambda
  &\xrightarrow{j_\lambda} \overline{LX}_\lambda
   \xrightarrow{\overline{i}_\lambda} LX,
\end{align*}
where $j_v$ and $j_\lambda$ are the open embeddings and $\overline{i}_v$ and
$\overline{i}_\lambda$ are the closed embeddings.
\end{definition}
\begin{lem}\label{lem:orbit_closure_order}
Let $X$ be an affine $G$-variety of finite type
and let $H=L^+G$ or $H=I$.
The closure relation on the set of $H(k)$-orbits
in $LX(k)$ is a partial order.
\end{lem}
\begin{proof}
It suffices to check the antisymmetry condition. Assume that  $x,y\in LX(k)$ have the same orbit closure. Embed $X$
equivariantly into a finite-dimensional $G$-representation $V$ and
choose $r\geq 0$ such that $x,y\in t^{-r}L^+V$.
For $N\geq 1$, let $x_N$ and $y_N$ be the reductions of $t^rx$
and $t^ry$ modulo $t^N$, and let $H_N$ be the image of $H$ in
$L^NG$. 
Note that $H_N$-orbits of $x_N$ and $y_N$ have the same closure since reduction modulo $t^N$ is equivaraint. 
These finite-type orbits are open and dense in their common closure, and thus coincide.

Let $T_N=\{h\in H_N(k)\mid hx_N=y_N\}$. For fixed $N$ the images
of $T_M$ in $H_N(k)$ for $M\geq N$ form a descending sequence
of nonempty closed cosets, which stabilize since $H_N(k)$ is Noetherian. Thus we find an $h\in H(k)$ with $hx=y$, and complete the proof.
\end{proof}

The following result  provides the orbit geometry
needed in the symmetric case.

\begin{thm}({\cite[Theorem 36, Lemma 38, Proposition 39]
{chenyi_2025singularitiesorbitclosuresloop}})\label{thm:chen_Ti}
Assume that $X=G/G^\theta$ is a symmetric variety
and that $G^\theta$ is connected.
Let $J=L^+G$ or $J=I$. 
\begin{enumerate}
    \item The reduced closures of $J$-orbits in $LX$
are placid schemes.
The inclusions of orbits into their closures
are finitely presented open embeddings,
and inclusions between comparable orbit
closures are finitely presented closed
embeddings. On each connected component of $LX$ the
closure order on $J$-orbits is filtered.
\item Let $\{Z_a\}$ be the family of finite disjoint
unions of reduced $J$-orbit closures with at
most one closure chosen in each connected
component.
Ordered by inclusion, this family is filtered,
its transition maps are finitely presented
closed embeddings and
$$
(LX)_{\mathrm{red}}\simeq\varinjlim_a Z_a.
$$
Each $Z_a$ is quasi-compact and quasi-separated.
\end{enumerate}
\end{thm}

A central object of study in our paper is the full subcategory of coherent $\cD$-modules.
\begin{definition}[Coherent $\cD$-modules]
\label{defn:coherent_D_mod}
\begin{enumerate}
    \item Let $\cY$ be an algebraic stack of finite type.
An object $\cF\in\cD(\cY)$ is \textit{coherent} if
$p^!\cF$ is bounded with coherent $\cD$-module 
cohomology for one (and thus for all) smooth
surjective morphism $p:S\to\cY$ where $S$ a finite-type
scheme. We denote the resulting full subcategory
by $\cD_c(\cY)$.
\item For a placid ind-scheme $Y$ considered in this paper
we define $\cD_c(Y):=\cD(Y)^\omega$.
\item Let $H$ be an affine group scheme whose prounipotent radical has finite codimension. For quotient stack $H\backslash Y$,
where $Y$ is a placid ind-scheme with an $H$-action, define
$$
\cD_c(H\backslash Y)
:=
\{
\cF\in\cD(H\backslash Y)
\ \vert\
p^!(\cF)\in\cD_c(Y)
\},
$$
where  $p:Y\to H\backslash Y$ is the natural projection.
\end{enumerate}
\end{definition}
\begin{remark}
For a finite-type scheme $Y$,  coherence
is equivalent to compactness i.e. $\cD_c(Y)=\cD(Y)^\omega$.   Thus for a finite type smooth algebraic group $H$ acting on such $Y$,  $\cF\in \cD_c(H\backslash Y)$ if and only if $p^!\cF$ is compact in $\cD(Y)$. In general, compactness is strictly stronger than coherence for stacks, c.f. \cite[\S1.1.11]{chen2023langlands}.
\end{remark}

\begin{lem}[Coherent $\cD$-modules on placid presentations]
\label{lem:coherent_finite_stage}\label{lem:coherence_via_stages}
Let $A$ be a filtered poset.
\begin{enumerate}
\item Let $J$ be the trivial group,
$L^+G$ or $I$. Let $Y=\varinjlim_{a\in A}Y_a$  as in
\autoref{ex:sheaf_placid_ind} with a $J$-action such that each $Y_a$ stable under $J$.
The $*$-pushforward along closed embeddings $Y_a\hookrightarrow Y$ induces an equivalence
$$
  \varinjlim_{a\in A}\cD_c(J\backslash Y_a)
  \xrightarrow{\sim}\cD_c(J\backslash Y).
$$

\item Let $H=\varprojlim_{j\in A^{\op}}H_j$ be a placid affine
group scheme, where $H_j$  is finite-type quotients and the kernel of each connecting morphism has pro-unipotent kernels. Let $Y=\varprojlim_{j\in A^{\op}}Y_j$
be a placid presentation of an $H$-scheme such that $H$ acts on
$Y_j$ through $H_j$ and the transition maps
$p_{kj}:Y_k\to Y_j$ are $H$-equivariant with fully faithful
$!$-pullback. Write $q_j:H\backslash Y\to H_j\backslash Y_j$
for the projection. Then $q_j^!$ is fully faithful and $\cF$ is coherent if and only if $q_j^!\cF$ is. Moreover, these pullbacks induce an equivalence
$$
  \varinjlim_{j\in A}\cD_c(H_j\backslash Y_j)
  \xrightarrow{\sim}\cD_c(H\backslash Y),
$$
where the transition functors are $!$-pullbacks along
$H_k\backslash Y_k\to H_j\backslash Y_j$.
\end{enumerate}
\end{lem}

\begin{proof}
(1) The assertion follows directly from \cite[\S\S4.2.7 and 4.3.4]{chen2023langlands}.
(2) Write $p_j:Y\to Y_j$.
By \autoref{ex:sheaf_on_placid_scheme} and
\cite[\S3.1.5]{beraldo2017loop}, we have
$\cD(Y)\simeq
  \varinjlim_{j\in A,\,p_{kj}^!}\cD(Y_j)$. By assumption, $p_{k,j}^!$ is fully faithful and and preserve compact objects. Write $r_{j,n}:H^n\times Y\longrightarrow H_j^n\times Y_j$ for the natural projection. The external product formula and the assumption show that
$$
  r_{j,n}^!:\cD(H_j^n\times Y_j)
  \longrightarrow\cD(H^n\times Y)
$$
is fully faithful for every $n$.
Taking totalizations proves full faithfulness
of $q_j^!$.
Since $p_j^!$ is continuous and fully faithful,
it also reflects compactness. Let $\alpha: Y\rightarrow H\backslash Y$ and $\alpha_j:Y_j\rightarrow H_j\backslash Y_j$ be natural projections. The isomorphism
$\alpha^!q_j^!
  \simeq p_j^!\alpha_j^!$ proves the asserted coherence criterion.
Let $\cF\in\cD_c(H\backslash Y)$.
Choose $j$ such that its underlying compact object
comes from $Y_j$. Let
$\pi_n:H^n\times Y\longrightarrow Y$ and $\pi_{j,n}:H_j^n\times Y_j\longrightarrow Y_j$
denote the projections onto the second factors. Let $\cF^{(n)}\in\cD(H^n\times Y)$ be the
degree-$n$ component in \eqref{eq:geometric_realization_d_mod}.  Then $ \cF^{(n)}
  \simeq \pi_n^!p_j^!\cM_j
  \simeq r_{j,n}^!\pi_{j,n}^!\cM_j$. Degreewise full faithfulness yields an object
$\cF_j\in\cD(H_j\backslash Y_j)$ with
$q_j^!\cF_j\simeq\cF$.
Since $q_j^!$ reflects coherence, $\cF_j$ is thus
coherent. We prove the essential surjectivity of the desired equivalence.
The full faithfulness follows from that of $q_j^!$.

\end{proof}
\begin{lemma}\label{lem:convolution_preserves_coherence}
Let $J,J',J''\in\{I,K\}$. Convolution restricts to the functors between coherent $\cD$-modules
\begin{align*}
    \cD_c(J'\bs LG/J)\times\cD_c(J\bs LG/J'')&\rightarrow\cD_c(J'\bs LG/J''),\\
\cD_c(J'\bs LG/J)\times\cD_c(J\bs LX)&\rightarrow\cD_c(J'\bs LX).
\end{align*}
\end{lemma}
\begin{proof}
We establish the second restricted functor and the first one follows from a similar argument. Let $p:LG/J\rightarrow J'\backslash LG/J$ and $q:LX\rightarrow J\backslash LX$ be natural projections. Let $\cF\in \cD_c(J'\backslash LG/J)$ and $\cG\in \cD_c(J\backslash LX)$. Then $p^!\cF$ and $q^!\cG$ are both compact. By \autoref{lem:coherent_finite_stage}, the object
$p^!\cF$ is supported on a finite union of spherical Schubert varieties $Z\subset LG/J$. Write
$i_Z:Z\hookrightarrow LG/J$ and $Z^{(\infty)}\subset LG$ be the inverse image of $Z$. Choose
$\cF_Z\in\cD_c(Z)$ such that $(i_Z)_*\cF_Z\simeq p^!\cF$. Note that there is an isomorphism $ Z^{(\infty)}\times^J LX
    \simeq  Z\times LX$ sending $[g,x]$ to $(gJ,gx)$. The convotlution morphism $m_Z:Z^{(\infty)}\times^J LX\longrightarrow LX$ is identified with the second projection $ Z\times LX\rightarrow LX$. Thus, $m$ is proper and $m_*$ has a continuous right adjoint $m^!$. It follows that $m_*$ preserves coherence.  Over an open subscheme $U\subset Z$ on
which $Z^{(\infty)}\rightarrow Z$ is trivialized, the above isomorphism identifies
$\cF_Z\,\widetilde\boxtimes\,\cG$ with the image of the compact object
$\cF_Z|_U\boxtimes q^!\cG$. Then the twisted product $\cF_Z\widetilde\boxtimes\cG
  \in\cD(Z^{(\infty)}\times^J Y)$ is coherent.  Thus the statement follows from $q'^!(\cF\ast\cG)\simeq m_*
(p^!\cF\,\widetilde\boxtimes\,\cG)$, where $q':LX\rightarrow J'\backslash LX$ is the canonical projection.
    \end{proof}

Let $\cY$ be one of the followings:
\begin{enumerate}
    \item a finite-type algebraic stack with affine
    stabilizers;
    \item a placid ind-scheme $Y$;
    \item a  quotient prestack $J\backslash Y$,
    where $Y$ is a placid ind-scheme equipped with an
    action of $J=L^+G$ or $J=I$.
\end{enumerate}
\begin{definition}[Renormalization]
\label{defn:renormalized_D_mod}
We define the \textit{renormalized category} by
$$
\cD^{\ren}(\cY):=\Ind (\cD_c(\cY)).
$$
The inclusion $\cD_c(\cY)\hookrightarrow\cD(\cY)$
extends uniquely to a continuous functor
$$
\theta_{\cY}:
\cD^{\ren}(\cY)\longrightarrow\cD(\cY),
$$
called the \textit{unrenormalization functor}.
\end{definition}

\subsection{Categorical Actions}\label{section:categorical_actions}
\label{sec:categorical_action}

We briefly recall the theory of categorical actions that we need in this paper.
For a placid group ind-scheme $H$, multiplication
endows $\cD^*(H)$ with a monoidal
structure give by convolution, see
\cite[\S 2.3.1]{campbell2021affine}.  For $\cC\in \cD^*(H)\text{-}\mathrm{mod}$, define the category of $H$-invariants
\begin{equation}\label{eq:categorical_invariants}
\cC^H:=
\mathrm{Fun}^{L}_{\cD^*(H)}
(\Vect,\cC)\simeq\operatorname{Tot}_{[n]\in\Delta}
  \mathrm{Fun}^{L}
  (\cD^*(\cH)^{\otimes n},\cC),
\end{equation}
where $\Vect$ carries the trivial action, 
c.f.\ \cite[\S 4.1]{beraldo2017loop}.

 Let $J\subset H$ be an affine group subscheme whose
prounipotent radical has finite codimension and
assume that $H/J$ is an ind-proper ind-scheme.
Then
$\cD^*(H/J)\simeq\cD^!(H/J)=:\cD(H/J)$.
Define $\cH_J:=\cD(J\backslash H/J)\simeq\cD(H/J)^J$.
Convolution gives $\cD(H/J)$ commuting left
$\cD^*(H)$-action and right $\cH_J$-action.
Since
$\cD(H/J)\simeq
\cD^*(H)\otimes_{\cD^*(J)}\Vect$,
induction--restriction gives
$$
\cC^J\simeq
\mathrm{Fun}^{L}_{\cD^*(H)}
(\cD(H/J),\cC),
$$
which is naturally  a left $\cH_J$-module. By \cite[\S 3.1]{campbell2021affine} we obtain the adjunction
\begin{equation}\label{eq:adjunction}
\cD(H/J)\otimes_{\cH_J}(\bullet):
\cH_J\text{-}\mathrm{mod}
\rightleftarrows
\cD^*(H)\text{-}\mathrm{mod}:(\bullet)^J.
\end{equation}

  \begin{thm}\cite[Proposition 3.1.3,\ Theorem 3.1.5]{campbell2021affine}\label{thm:CD_counit} If $H/J$ is ind-proper,
\begin{enumerate}
    \item the left adjoint in the adjunction (\ref{eq:adjunction}) is fully faithful.
    \item the counit of the adjunction (\ref{eq:adjunction})
    \begin{equation}
        \cD(H/J)\otimes_{\cH_J} \cC^J\rightarrow \cC
    \end{equation}
is fully faithful and admits a continuous $\cD(H)$-linear right adjoint.
\end{enumerate}
\end{thm}

\begin{lem} \label{lem:invrariance_lemma}
    Let $\cC,\cD\in\cD^*(L^+G)\textnormal{-mod}$, and $F:\cC\rightarrow \cD$ a fully faithful functor of $\cD^*(L^+G)$-modules. For any $P\subset L^+G$ a parahoric subgroup, the associated functor of taking $P$-invariants
    $$
F^P:\cC^P\rightarrow \cD^P
    $$
    is also fully faithful and admits a right adjoint.
\end{lem}

\begin{proof}
Recalling \eqref{eq:categorical_invariants}, the functor $F^P$ is induced by postcomposition with $F$
in each (cosimplicial) degree.
Since colimits of continuous functors are computed
pointwise, these degreewise functors preserve colimits. Postcomposition with $F$ is fully faithful in every cosimplicial degree, which shows that taking the limit preserves full faithfulness.
By \cite[Propositions 5.5.3.8 and 5.5.3.13]{HTT},
the invariant categories are presentable and $F^P$
preserves colimits. Hence $F^P$ admits a right adjoint
by the adjoint functor theorem for presentable categories \cite[Corollary 5.5.2.9(1)]{HTT}.
\end{proof}

\section{The Satake category and unramified relative local Langlands}
\label{section:satake_category}
In this section we define the Satake subcategory $\cD_c(L^+G\backslash LX)^{\Sat}$ of $\cD_c(L^+G\backslash LX)$
which is the automorphic category in
the unramified relative local Langlands \autoref{conj:relative_Langlands_unramified}.
Our main result in this section is \autoref{thm:spherical_sat_full}, 
which proves the equality 
$\cD_c(L^+G\backslash LX)^{\Sat}=\cD_c(L^+G\backslash LX)$
for a general class of symmetric varieties $X$ under 
\autoref{assu:symmetric_condition_satfull}. 

\subsection{The unramified relative local Langlands conjecture}\label{section:unramified_equivalence}
We recall the preliminaries and state the unramified
relative local Langlands
\autoref{conj:relative_Langlands_unramified}
following \cite[Conjecture 7.5.1]{BZSV}.
Our study focuses on  pairs $(G,M)$ and their
duals $(\check G,\check M)$ in the sense of
\cite[\S5.2]{BZSV}, where $G$ is a connected reductive group over $\CC$
and $M=T^*(2)X$ is a polarized hyperspherical
$G$-variety as defined below.

\begin{definition}[Spherical varieties]
\label{def:polarized_hyperspherical}
A normal $G$-variety $X$ over $\mathbb{C}$ is \textit{spherical} if $B$ has a Zariski open dense orbit.
We further assume the following two conditions:
\begin{enumerate}
    \item $X$ is smooth and affine.
    \item the stabilizer of the $B$-action on a point in the open
    $B$-orbit is connected.\footnote{This condition
    implies that $X$ has no spherical roots of type $N$,
    see \cite[Remark 4.2.1]{BZSV}.}
\end{enumerate}

Under these conditions, $M=T^*(2)X$ is a \textit{polarized 
hyperspherical $G$-variety}, with the weight-2 $\GG_m$ action on fibre  in the sense of 
\cite[\S 3.5.1.]{BZSV} (see \cite[Proposition 3.7.4]{BZSV}
for the proof). In particular $M$ is a \textit{graded
Hamiltonian $G$-variety}. 
\end{definition}

For a polarized hyperspherical variety $M=T^*(2)X$,
the construction of \cite[\S4]{BZSV} gives a graded
Hamiltonian $\check G$-variety $\check M$ assuming \cite[Conjecture 4.3.16]{BZSV}.
The variety $\check M$ is smooth and affine by
\cite[\S3.4.8 and Lemma 3.4.9]{BZSV} and is expected
to be hyperspherical as in
\cite[Expectation 5.2.1]{BZSV}.

\begin{definition}[Basic objects]
\label{def:basic_object}
Let $X$ be a smooth affine spherical $G$-variety
as in \autoref{def:polarized_hyperspherical}.
Let $i_0: L^+X\hookrightarrow LX$
be the closed embedding. We define 
$\IC_0:= i_{0*} \omega_{L^+X}\in 
\cD_c(L^+G\backslash LX)$, referred to 
as the \textit{basic object}. 
\end{definition}

\begin{definition}[Satake subcategories]
\label{def:satake_subcategory}
Let $X$ be as in \autoref{def:polarized_hyperspherical} and let 
$r:I \backslash LX\to L^+G\backslash LX$
be the natural projection. 
The spherical and Iwahori level Satake subcategories are the thick spans 
of the convolutions with $\IC_0$:
$$
\cD_c(K\bs LX)^{\Sat}:=\Thick_{\cD_c(K\bs LX)}\{\cF\ast\IC_0\ \vert\ \cF\in\cD_c(K\bs\Gr)\},
$$
and
$$
\cD_c(I\bs LX)^{\Sat}:=\Thick_{\cD_c(I\bs LX)}\{\cF\ast r^!\IC_0\ \vert\ \cF\in\cD_c(I\bs\Fl)\},
$$
respectively. We also introduce a closely related category
$$
\cD_c(I\backslash LX)^{\rm Sph, Sat}:=\Thick_{\cD_c(I\bs LX)}\{\cF\ast \IC_0\vert \ \cF\in \cD_c(I\backslash \Gr)\}.
$$
The superscript $\rm Sph$ refers to the spherically generated part $\cD_c(I\backslash\Gr)$. This notation is justified in  \autoref{section:construction_of_the_functor}.
We will later on show that the last two categories are equivalent (see \autoref{lem:satake-generators}). 
\end{definition}

We recall the unramified relative local conjecture  \cite[Conjecture 7.5.1]{BZSV}.
\begin{conjecture}[Unramified relative local Langlands]
\label{conj:relative_Langlands_unramified}
Let $M=T^*(2)(X)$ be a polarized 
hyperspherical $G$-variety with the corresponding 
hyperspherical dual $(\check{G},\check{M})$. There 
 is an equivalence
of categories 
$$
\cD_{{c}}(L^+G\backslash LX)^{\Sat} \simeq \Perf(\sh^{1/2}(\check{M})/\check{G})
$$
such that 
\begin{enumerate}
    \item the equivalence is
    compatible with the action of 
derived geometric Satake equivalence 
\cite{bezrukavnikov2008equivariant}
$$
\cD_{{c}}(L^+G\backslash LG/L^+G)\simeq 
\Perf(\check{\mathfrak{g}}^*[2]/\check{G}). 
$$
\item Pointing condition is satisfied, i.e. 
the unramified basic object $\IC_0$ is sent to $\cO_{\sh^{1/2}(\check{M})}$.
\end{enumerate}
\end{conjecture}

For a detailed discussion of known cases and related
results see \cite[\S7.6, especially \S7.6.5]{BZSV}. See \autoref{subsec:applicable_cases} for 
examples of \autoref{conj:relative_Langlands_unramified} that are relevant to our main theorem.

\subsection{Spherical orbits in loop spaces of symmetric varieties}
\label{subsec:symmetric_spherical_orbits}

We first recall the geometry  $L^+G$-orbits
of $LX$ when $X$ is symmetric following \cite{chenyi_2025singularitiesorbitclosuresloop}. 

\begin{definition}[Root datum for symmetric varieties]
\label{def:symmetric_var}
Let $\theta:G\to G$ be an involution and let
$X=G/G^\theta$ with its natural left $G$-action.
Choose a maximal $\theta$-split torus $S$, so that
$\theta(s)=s^{-1}$ for all $s\in S$, and a minimal
$\theta$-split parabolic $P$ such that $P\cap\theta(P)=Z_G(S)$.
Choose the Borel pair $T\subset B$ of
\autoref{subsec:notations} so that
$S\subset T\subset B\subset P$.
Then $T$ is $\theta$-stable
\cite[\S 2.6(iv), \S 2.8(ii)]{richardson1982orbits}.
Write $L:=Z_G(S)$, the common Levi subgroup of
$P$ and $\theta(P)$. We keep the notation
$\Phi,\Phi^+,\check\Phi,\check\Phi^+$ of
\autoref{subsec:notations} and write
$$
  \mathbb X_*(T)^+
  :=
  \{\lambda\in\mathbb X_*(T):
    \langle\alpha,\lambda\rangle\geq0
    \text{ for all }\alpha\in\Phi^+\}.
$$
The inclusion $S\subset T$ identifies
$\mathbb X_*(S)$ with a sublattice of $\mathbb X_*(T)$.
Let
$$
  \mathbb X_*(S)^+
  :=\mathbb X_*(S)\cap\mathbb X_*(T)^+.
$$
For $\lambda,\mu\in\mathbb X_*(T)$ write $\mu\leq\lambda$
if $\lambda-\mu$ is a nonnegative integral combination
of positive coroots in $\check\Phi^+$.
We equip $\mathbb X_*(S)$ with the restriction of this order.
\end{definition}

\begin{example}
\begin{enumerate}
    \item For $G=\SL_n$ and $\theta(g)=(g^t)^{-1}$ we have
$X\cong\SL_n/\SO_n$.
Then $S=T$ is the torus of diagonal matrices of $\SL_n$.
We can choose $P=B$, the upper triangular Borel subgroup,
as $P\cap\theta(P)=Z_G(S)=T$. Hence
$$
\mathbb X_*(S)^+=\mathbb X_*(T)^+
=\{(\lambda_1,\ldots,\lambda_n)\in\mathbb Z^n:
\lambda_1\ge\cdots\ge\lambda_n,\
\textstyle\sum_{i=1}^n\lambda_i=0\}.
$$
\item For $G=\SL_{2n}$ and $\theta(g)=J(g^t)^{-1}J^{-1}$,
where $J=\begin{pmatrix}0&I_n\\-I_n&0\end{pmatrix}$,
we have $G^\theta=\Sp_{2n}$.
Let $T$ be the maximal torus of diagonal matrices in $\SL_{2n}$.
A maximal $\theta$-split torus is
$$
S=\{\operatorname{diag}(t_1,\ldots,t_n,t_1,\ldots,t_n):
\textstyle\prod_{i=1}^n t_i=1\}.
$$
Let $P$ be the stabilizer of the partial flag
$0\subset\langle e_1,e_{n+1}\rangle
\subset\langle e_1,e_{n+1},e_2,e_{n+2}\rangle
\subset\cdots\subset k^{2n}$ 
and let $B$ be the Borel subgroup corresponding to the
complete flag defined by the ordered basis
$c=(e_1,e_{n+1},e_2,e_{n+2},\ldots,e_n,e_{2n})$.
The root datum is the usual root datum of $\SL_{2n}$
with the ordering determined by $c$. Then
$$
\mathbb X_*(S)^+
=\{(\lambda_1,\ldots,\lambda_n,\lambda_1,\ldots,\lambda_n)
\in\mathbb Z^{2n}:
\lambda_1\ge\cdots\ge\lambda_n,\
\textstyle\sum_{i=1}^n\lambda_i=0\}.
$$
\end{enumerate}
\end{example}

\begin{proposition}[{\cite[Proposition 17]{chenyi_2025singularitiesorbitclosuresloop} }] 
\label{prop:symmetric_var_spherical_orbit}
\begin{enumerate}
    \item The map $\tau: G\to G$ sending $\tau(g)$ to $g\theta(g)^{-1}$ induces 
an isomorphism $\tau: X=G/G^{\theta} \to 
(G^{\text{inv}\circ \theta})^{\circ}
\subset G$, where the target is the identity component of
$G^{\text{inv}\circ \theta}$ on which $\theta$ acts by inverse. 
\item Under $\tau$, the $G$-action on $X$
translate to the $G$-action on
$G^{\text{inv}\circ \theta}$ given by 
$g\cdot_{\theta} x= gx\theta(g)^{-1}$. The $L^+G$-orbits
of $LX$ are represented by $\mathbb X_*(S)^+$, whose representatives
can be chosen to be $x_{\mu}\in LX(k)=X(F)$ so 
$\tau(x_{\mu})=t^{\mu}\in S(F) \subset 
(G^{\inv\circ \theta})^{\circ}(F)$ 
for $\mu\in \mathbb X_*(S)^+$.
\end{enumerate}
\end{proposition}

\begin{definition}[Relative cocharacters]
\label{def:symmetric_relative_cocharacters}
Let $X$ be a symmetric $G$-variety. We have an identification
of the maximal $\theta$-split torus $S$ (as in \autoref{prop:symmetric_var_spherical_orbit}) with 
the universal Cartan of $X$ via 
$$
A_X\cong T/T^\theta\xrightarrow{\sim}S,
\quad t\longmapsto t\theta(t)^{-1}.
$$
Under this identification the quotient map $T\to A_X$
induces the map
$$
\overline{(-)}:\mathbb X_*(T)\longrightarrow\mathbb X_*(S),
\quad \lambda\longmapsto \overline{\lambda}:= \lambda-\theta(\lambda).
$$
\end{definition}
The map $\overline{(-)}$ sends $\mathbb X_*(T)^+$ to $\mathbb X_*(S)^+$.
Indeed $\theta$ fixes the roots $\Phi_L$ of
$L=Z_G(S)$ and sends $\Phi^+\setminus\Phi_L$ to
negative roots. 
Consequently for $\lambda\in\mathbb X_*(T)^+$,
$$
\langle\alpha,\overline\lambda\rangle
=
\langle\alpha,\lambda\rangle
-\langle\theta(\alpha),\lambda\rangle\geq0,
$$
for all $\alpha\in\Phi^+$. The map $\overline{(-)}$ also preserves the coweight
order. For a positive coroot $\alpha^\vee$ its image
under $1-\theta$ is zero if $\alpha\in\Phi_L$
and otherwise is the sum of the positive coroots
$\alpha^\vee$ and $(-\theta(\alpha))^\vee$.
Thus $1-\theta$ preserves nonnegative integral
combinations of positive coroots, and
$\lambda'\leq\lambda$ implies
$\overline{\lambda'}\leq\overline\lambda$.

\begin{lemma}\label{lem:symmetric_stabilizers}
Let $\mu\in\mathbb X_*(S)^+$ and $Z_{G^\theta}(\mu):=G^\theta\cap Z_G(\mu)$.
\begin{enumerate}
\item $\pi_0(K_\mu)\cong\pi_0\big(Z_{G^\theta}(\mu)\big)$, where $K_\mu=\Stab_{L^+G}(x_\mu)$. 
\item $Z_{G^\theta}(\mu)=Z_{G^\theta}(\mu)^\circ\cdot Z_{G^\theta}(S)$.
\item $K_\mu$ is connected for every $\mu\in\mathbb X_*(S)^+$ if and only if $Z_{G^\theta}(S)$ is connected.
\end{enumerate}
\end{lemma}

\begin{proof}
(1) By \autoref{prop:symmetric_var_spherical_orbit} 
$g\in L^+G$ fixes $x_\mu$ if and only if $\theta(g)=t^{-\mu}gt^{\mu}$.
Let $J_\mu:=L^+G\cap t^\mu L^+G\,t^{-\mu}$, which comes with 
an involution $\theta_{\mu}= \text{Ad}(t^{\mu})\circ \theta$.
Then $K_{\mu}$ is the $\theta_{\mu}$-fixed points of $J_{\mu}$.
The image of $J_{\mu}$ under the 
evaluation map $\text{ev}: L^+G\to G$ is the parabolic $P_{\mu}$ corresponding to $\mu$ \cite[p. 26]{zhu2016introduction}, i.e. 
the group generated by $T$ and root subgroups $U_{\alpha}$ for those 
root $\alpha$ satisfying $\langle \alpha, \mu\rangle \le 0$.
By \cite[Theorem 13.33]{milneAlgGrp} $P_{\mu}\cong U_{\mu} 
\rtimes Z_G(\mu)$ where $U_{\mu}$ is the unipotent radical of
$P_{\mu}$. 
As a result we find $J_{\mu}= J_{\mu}^+ \rtimes 
Z_G(\mu)$ where $J_{\mu}^+=\ker (J_{\mu}\to P_{\mu}\to Z_G(\mu))$
is pro-unipotent and is stable under $\theta_{\mu}$.
Therefore as $\theta=\theta_{\mu}$ on $Z_G(\mu)$, we find 
$K_{\mu}= J_{\mu}^{\theta_{\mu}}= (J_{\mu}^+)^{\theta_{\mu}} 
\rtimes Z_{G^{\theta}}(\mu)$. Note that $(J^+_{\mu})^{\theta_{\mu}}$
is connected \footnote{in characteristic $0$ unipotent group
is always connected; in characteristic $p\ne 2$ fixed-points of 
an involution $\theta$ (i.e. order $2$) of unipotent group $U$ is always connected},
hence $\pi_0(K_{\mu})=\pi_0(Z_{G^{\theta}}(\mu))$. 

(2) Let $M:=Z_G(\mu)$. This is a $\theta$-stable connected reductive
group containing $S$. Moreover $S$ is a maximal
$\theta$-split torus of $M$ since it is already maximal
among the $\theta$-split tori of $G$.
By \cite[2.9]{richardson1982orbits}
$M^\theta=(M^\theta)^\circ\cdot(S\cap M^\theta)$.
Since
$S\cap M^\theta
\subseteq Z_{G^\theta}(S)
\subseteq M^\theta$,
we have
$Z_{G^\theta}(\mu)
=Z_{G^\theta}(\mu)^\circ\cdot Z_{G^\theta}(S)$.

(3) Suppose that $Z_{G^\theta}(S)$ is connected.
Since it is contained in $Z_{G^\theta}(\mu)$, it lies
in $Z_{G^\theta}(\mu)^\circ$. Thus (2) shows that
$Z_{G^\theta}(\mu)$ is connected for every
$\mu\in\mathbb X_*(S)^+$, and so is $K_\mu$ by (1).

Conversely suppose that $K_\mu$ is connected for every
$\mu\in\mathbb X_*(S)^+$. Choose a regular dominant
cocharacter $\lambda\in\mathbb X_*(T)^+$, then by the 
same argument as in \autoref{def:symmetric_relative_cocharacters}
we know $\langle \alpha, \overline{\lambda}\rangle \ge 
\langle \alpha, \lambda\rangle >0$ for $\alpha\in \Phi^+\setminus 
\Phi_L$, so the only roots vanishing on $\mu=\overline{\lambda}$
are precisely those roots from $L=Z_G(S)$. 
The description of centralizers by root subgroups therefore gives
$Z_G(\mu)=Z_G(S)$.
Thus $Z_{G^\theta}(\mu)=Z_{G^\theta}(S)$.
By (1) this group is connected since $K_\mu$
is connected by assumption.
\end{proof}

\begin{assu}[Geometric conditions for Sat equals to Full]
\label{assu:symmetric_condition_satfull}
Let $X$ be a symmetric $G$-variety. 
We assume that
\begin{enumerate}
  \item $Z_{G^{\theta}}(S)$ is connected. 
  \item The map $\mathbb X_*(T)^+\to \mathbb X_*(S)^+$
  sending $\lambda$ to $\overline{\lambda}$ is surjective. 
\end{enumerate}
\end{assu}

These conditions are related to the relative Satake
results of
\cite[Theorem 87 and its proof]
{chenyi_2025singularitiesorbitclosuresloop}.
That theorem concerns the perverse category under an
even-codimension hypothesis. Our result below concerns
thick generation of the coherent derived category.
Function-theoretically this is an analogue to asking 
when the basic function $1_{X(\CC[[t]])}$ generates all 
$G(\mathbb{C}[[t]])$-invariant compactly supported functions 
on $X(\mathbb{C}((t)))$, which was addressed in \cite[Theorem 8.0.2]{sakellaridis2013spherical} with similar (and more) geometric conditions for spherical variety $X$. In our upcoming work 
we hope to address such results at the categorical level for 
general spherical varieties. 

\begin{remark}
Condition (1) of
\autoref{assu:symmetric_condition_satfull} and \autoref{lem:symmetric_stabilizers} imply that
$G^\theta$ and all stabilizers of $L^+G$-orbits in
$LX$ are connected.
Thus the constant local system is the only irreducible
$L^+G$-equivariant local system on each orbit.
Condition (2) will allow us to recover the corresponding
IC-objects as shifted direct summands of the
convolutions $\Sat(V_\lambda)*\IC_0$.
These observations will be used to prove $\cD_c(L^+G\backslash LX)^{\Sat}
=
\cD_c(L^+G\backslash LX)$.
\end{remark}

\begin{example}
\label{exam:generation}
The above assumption applies 
to $\Spin_{2n}/\Spin_{2n-1}$ $(n\ge2)$, $\GL_{p+q}/(\GL_p\times\GL_q)$, $\SL_4/\Sp_4$ and $\GL_2/\Sp_2$. 
The condition $\overline{\mathbb X_*(T)^+}=\mathbb X_*(S)^+$ fails for $\SL_{2n}/\Sp_{2n}$ $(n\ge3)$.
\end{example}

\subsection{Standard and IC generators}
\label{subsec:standard_IC_generators}

\begin{cons}[Generators for coherent equivariant $\cD$-modules on an orbit]
\label{cons:generator_sheaves_on_orbit}
Let $H$ be a connected placid affine group scheme
over $\CC$ admitting a placid presentation
$H\cong\varprojlim_j H_j$ by finite-type quotients,
where the kernels of $H\to H_j$ are pro-unipotent.
Let $Y$ be a locally equidimensional $H$-placid
ind-scheme equipped with an $H$-invariant dimension theory $\delta$
\cite[\S6.10]{raskin_dmodules_infinite_var}, as in
\cite[Example 108]{chenyi_2025singularitiesorbitclosuresloop}.
Let $\cO$ be an $H$-orbit in $Y$ whose inclusion
$\cO\hookrightarrow Y$ is a finitely presented
locally closed immersion.
We equip $\cO$ with the dimension theory induced by
$\delta$ as in
\cite[\S6.13]{raskin_dmodules_infinite_var}.
Since $\cO$ is a connected placid scheme, this dimension theory
differs from the canonical dimension theory of $\cO$ by a constant,
which we denote by $\delta(\cO)\in\ZZ$.

For $y\in\cO$ let
$H_y:=\Stab_H(y)$ and
$H_{y,j}:=\operatorname{im}(H_y\to H_j)$.
By \cite[Example 108]
{chenyi_2025singularitiesorbitclosuresloop}
we have a placid presentation
$$
  \cO\cong H/H_y
  \cong\varprojlim_j H_j/H_{y,j},
$$
whose transition maps are affine, smooth and unipotent. In
particular their $!$-pullbacks are fully faithful.
The kernels of $H_y\to H_{y,j}$ are pro-unipotent.
Hence $\pi_0(H_y)\cong\pi_0(H_{y,j})$, the transition functors
$\cD(BH_{y,j})\to\cD(BH_{y,k})$ are equivalences, and
\autoref{lem:coherence_via_stages}(2) gives equivalences
$$
  \cD_c(BH_{y,j})
  \simeq\cD_c (H_j\backslash(H_j/H_{y,j}) )
  \simeq\cD_c(H\backslash\cO)
$$
induced by $!$-pullback.

For a finite-dimensional representation $\xi$ of $\pi_0(H_y)$, let
$\mathcal L^{\omega}_{\xi,\cO}\in\cD_c(H\backslash\cO)$ be the object
corresponding to the $!$-pullback of $\xi$ along
$BH_{y,j}\to B\pi_0(H_{y,j})$, where we identify
$\cD(B\pi_0(H_{y,j}))\simeq\Rep(\pi_0(H_y))$.
Thus $\mathcal L^{\omega}_{\mathrm{triv},\cO}\simeq\omega_{H\backslash\cO}$. 
We set
$$
\mathcal L_{\xi,\cO}:=\mathcal L^{\omega}_{\xi,\cO}[-\delta(\cO)].
$$

On the finite-dimensional homogeneous space
$H_j/H_{y,j}$ coherent equivariant objects have
bounded cohomology consisting of finite-rank
equivariant local systems by
\cite[Theorem 11.6.1 and Remark 11.6.2]
{hotta2008dmodules}.
Transporting the finite-stage perverse $t$-structure along the
equivalences above and shifting it so that its heart contains
$\mathcal L_{\mathrm{triv},\cO}$ gives a bounded $t$-structure on
$\cD_c(H\backslash\cO)$, which we call the
\textit{$\delta$-perverse $t$-structure}. Its heart is
\begin{align*}
  \Rep(\pi_0(H_y))^{\heart}
  &\simeq \cD_c(H\backslash\cO)^{\heart},\\
  \xi &\mapsto \mathcal L_{\xi,\cO},
\end{align*}
where the heart on the left consists of
finite-dimensional representations.
Every object of this heart has finite length.
The bounded truncation filtration therefore gives
$$
  \cD_c(H\backslash\cO)
  =
  \Thick_{\cD_c(H\backslash\cO)}
   \{
    \mathcal L_{\xi,\cO}
    \mid
    \xi\in\Irr(\Rep(\pi_0(H_y)))
    \},
$$
which is generation in the sense of
\autoref{def:generation}(T).
Since thick spans are stable under shifts, the same holds with
$\mathcal L^{\omega}_{\xi,\cO}$
in place of $\mathcal L_{\xi,\cO}$.
In particular if $H_y$ is connected, then
$$
  \cD_c(H\backslash\cO)=
  \Thick_{\cD_c(H\backslash\cO)}
  \{\omega_{H\backslash\cO}\}.
$$
\end{cons}

\begin{cons}[(Co)standard objects]
\label{def:standard_object}
Let $X$ be a smooth affine spherical $G$-variety over
$\CC$ and fix a dimension theory $\delta$ on $LX$.
For an $L^+G$-orbit $LX_\lambda$ recall the embeddings
$$
  i_\lambda:
  LX_\lambda\xrightarrow{j_\lambda}
  \overline{LX}_\lambda
  \xrightarrow{\overline i_\lambda}LX
$$
of \autoref{def:orbits_iwahori_spherical}.
Assume that the reduced orbit closure
$\overline{LX}_\lambda$ is placid and that
$j_\lambda$ and $\overline i_\lambda$ are finitely
presented.

For each
$\xi\in\Irr(\Rep(\pi_0((L^+G)_\lambda)))$,
where $(L^+G)_\lambda$ is the stabilizer of a point of
$LX_\lambda$, let $\mathcal L_{\xi,LX_\lambda}$ be the
local system of
\autoref{cons:generator_sheaves_on_orbit}
with the normalization determined by $\delta$.
We define the \textit{standard} and \textit{costandard} objects
$$
  \Delta_{\lambda,\xi}
  :=i_{\lambda!}\mathcal L_{\xi,LX_\lambda},
  \quad
  \nabla_{\lambda,\xi}
  :=i_{\lambda*}\mathcal L_{\xi,LX_\lambda}
$$
in $\cD_c(L^+G\backslash LX)$.
For the constant local system we write
$\Delta_\lambda$ and $\nabla_\lambda$.

The same construction applies to $I$-orbits and to
orbits in an $L^+G$- or $I$-stable placid subscheme,
equipped with a dimension theory and satisfying
the same geometric conditions.

For a symmetric variety $X=G/G^\theta$ with connected
$G^\theta$ the required geometric conditions follow
from \autoref{thm:chen_Ti}, and a dimension theory is
provided by
\cite[Proposition 41]
{chenyi_2025singularitiesorbitclosuresloop}.
\end{cons}
\begin{remark}
    The standard and costandard objects are coherent.
The local system $\mathcal L_{\xi,LX_\lambda}$ has finite monodromy and is trivialized
by the finite \'etale cover
$L^+G/((L^+G)_\lambda)^\circ\to LX_\lambda$.
\'Etale locally on a placid presentation of
$\overline{LX}_\lambda$ the finitely presented open
embedding and this cover, together with its finite
monodromy data, descend to a common finite-type stage.
There the $!$- and $*$-extensions are bounded
regular holonomic complexes. Smooth base change and the
fixed dimension theory identify their pullbacks with
the extensions on $\overline{LX}_\lambda$.
They are therefore coherent, and finitely presented
closed extension along $\overline i_\lambda$ preserves
coherence by \cite[\S\S 4.2--4.3, 5.2]
{raskin_dmodules_infinite_var}.
\end{remark}

Let $J=L^+G$ or $I$ for the rest of this section.
\begin{lemma}\label{lem:finite_orbit_generation}
Let $X$ be an affine $G$-variety of finite type.
Let $Z\subset LX$ be a reduced, closed, $J$-stable placid
scheme that is a union of finitely many $J$-orbits.
Assume that the inclusions of the orbits into their
closures and of these closures into $Z$ are finitely
presented.
Then $\cD_c(J\backslash Z)$ is generated in the sense
of \autoref{def:generation} by the standard objects
of its orbits, extended to $Z$.
\end{lemma}
\begin{proof}
We argue by induction on the number of $J$-orbits.
We first note that a finite union of orbit closures is finitely presented closed placid subschemes of $Z$. 

Choose a maximal orbit $\cO$ in $Z$ and consider the diagram
$$
Z\setminus\cO\xhookrightarrow{i}Z
\xhookleftarrow{j}\cO,
$$
where $i$ and $j$ are finitely presented closed and open immersions respectively. Let $\cF\in\cD_c(J\backslash Z)$, then $j^!\cF\in\cD_c(J\backslash\cO)$.
By \autoref{cons:generator_sheaves_on_orbit}
$j^!\cF\in\cD_c(J\backslash\cO)$ belongs to the thick subcategory generated
by the irreducible equivariant local systems
$\mathcal L_{\xi,\cO}$ on $\cO$. Since the partially defined left adjoint $j_!$ is defined on
$\mathcal L_{\xi,\cO}$, with image being coherent $\cD$-modules by \autoref{def:standard_object}, c.f. \cite[Proposition 5.2.1]{raskin_dmodules_infinite_var}, $j_!j^!\cF$ is defined and belongs to the
thick span of the standard objects of $\cO$.

Let $\cC$ be the cone of the adjunction morphism
$$
j_!j^!\cF\longrightarrow\cF.
$$
Then $\cC$ is coherent and $j^!\cC=0$.
By Kashiwara's equivalence
$\cC\simeq i_*\cF_{Z\setminus\cO}$, where $\cF_{Z\setminus\cO}:=i^!\cC$.
Moreover $\cF_{Z\setminus\cO}$ is coherent since it is compact after forgetting the $J$-equivariance.
Thus $\cF_{Z\setminus\cO}\in\cD_c(J\backslash (Z\setminus\cO))$. The statement now follows from induction.
\end{proof}
\begin{p}
\label{prop:standardObjects_generate}
Let $X=G/G^\theta$ be a symmetric $G$-variety with
connected $G^\theta$.
Then $\cD_c(J\backslash LX)$ is thickly generated
by the standard objects of its $J$-orbits.
\end{p}

\begin{proof}
By \autoref{thm:chen_Ti} $(LX)_{\mathrm{red}}$ admits
a filtered presentation by finite disjoint unions
$Z_a$ of reduced $J$-orbit closures.
By \autoref{lem:coherent_finite_stage} every
$\cF\in\cD_c(J\backslash LX)$ is the closed extension
of an object of $\cD_c(J\backslash Z_a)$ for some $a$.
Each $Z_a$ is a finite union of $J$-orbits by
\cite[Theorem 18 and Corollary 27(i)]{chenyi_2025singularitiesorbitclosuresloop}.
The required finite-presentation properties follow
from \autoref{thm:chen_Ti}.
Applying \autoref{lem:finite_orbit_generation} on $Z_a$
and then extending to $LX$ proves the assertion.
\end{proof}

Next, we will show that $\cD_c(L^+G\backslash LX)$
is also thickly generated by the IC-extensions. 

\begin{definition}[IC-extension]
\label{def:IC_extension}
Under the notation and conditions of
\autoref{def:standard_object}, let
$\Perv_\delta(L^+G\backslash LX)$ denote the heart of
the $t$-structure on $\cD_c(L^+G\backslash LX)$
transported along the equivalence $\Phi_\delta$ of
\cite[Lemma 103(iii)]
{chenyi_2025singularitiesorbitclosuresloop}
from the $\delta$-perverse $t$-structure constructed in
\cite[Appendix A.7, equation (47)]
{chenyi_2025singularitiesorbitclosuresloop}.
Its restriction to each orbit has the normalization
of \autoref{cons:generator_sheaves_on_orbit}.
For an $L^+G$-orbit $LX_{\lambda}$ and
$\xi\in \Irr (\Rep(\pi_0(L^+G)_{\lambda}))$ the
\textit{IC-extension} is
$$
\IC_{\lambda,\xi}
:=(\overline i_\lambda)_*(j_\lambda)_{!*}\mathcal L_{\xi,LX_\lambda}
\in\Perv_\delta(L^+G\backslash LX),
$$
where $(j_\lambda)_{!*}$ is the intermediate extension in
$\Perv_\delta(L^+G\backslash\overline{LX}_\lambda)$.
The object is coherent by the same finite-type descent
as in \autoref{def:standard_object}, is supported on
$\overline{LX}_\lambda$ and satisfies
$$
i_\lambda^!\IC_{\lambda,\xi}
\simeq j_\lambda^!\,\overline i_\lambda^!\,\IC_{\lambda,\xi}
\simeq\mathcal L_{\xi,LX_\lambda}.
$$
\end{definition}

\begin{remark} 
\label{rem:sheaf_theory_basic_object}
When $X$ is a symmetric space, $LX$ is a 
$L^+G$-placid ind-scheme with dimension theory $\delta$
(\cite[Proposition 41]{chenyi_2025singularitiesorbitclosuresloop}, 
see also \autoref{thm:chen_Ti}). 
\begin{enumerate}
    \item We have $L^+X=LX_0$ is a closed orbit. Hence
$\IC_{0,\mathrm{triv}}\simeq\IC_0[-\delta(0)]$, where 
$\IC_0:=i_{0*}\omega_{L^+X}$
is the basic object defined in \autoref{def:basic_object}.
 \item By \cite[Lemma 103 (iii)]{chenyi_2025singularitiesorbitclosuresloop}, the dimension 
 theory $\delta$ induces an equivalence $\Phi_{\delta}: 
 \cD_c(L^+G\backslash LX) \cong \cD_c^*(L^+G\backslash LX)$,
 which sends $\IC_{\xi,\cO}$ to the corresponding IC-extension defined 
 in \cite[Example 108]{chenyi_2025singularitiesorbitclosuresloop}
 in the category of $*$-$\cD$-modules. 
 In particular, it sends the basic object 
 $\IC_0$ to $i_{0*}\underline{k}_{L^+X}[2\delta(0)]$. 
\end{enumerate}

\end{remark}

\begin{cor}\label{cor:spherical_IC_generation}
Let $X=G/G^\theta$ be a symmetric $G$-variety with
connected $G^\theta$.
Then $\cD_c(L^+G\backslash LX)$ is thickly generated by
the objects $\IC_{\lambda,\xi}$.
\end{cor}
\begin{proof}
By \autoref{prop:standardObjects_generate} it suffices
to show the standard objects are thickly generated
by the IC-extensions. 
The finite-type descent used in
\autoref{def:standard_object} shows that a standard
object has bounded regular holonomic perverse
cohomology with the normalization determined by $\delta$.
Its support is a finite union of $L^+G$-orbits.
On such a union equivariant perverse sheaves with
finite-rank restrictions have finite length, and their
simple objects are the intermediate extensions of
irreducible equivariant local systems on the orbits.
Thus the perverse truncation filtration and the
Jordan--H\"older filtrations of its terms express each
standard object as an object of the thick subcategory
generated by the $\IC_{\lambda,\xi}$.
\end{proof}

\subsection{Support of spherical convolution}
\label{subsec:spherical_convolution_support}

In this section we 
compute the support of the convolution $\Sat(V_{\lambda})*\IC_0$. For convenience, we work with $*$-$\cD$-modules till the end of \autoref{section:satake_category}.  

Let $e_0:=i_{0*}\underline{k}_{L^+X}
    \in\cD_c^*(L^+G\backslash LX)$ and write $\IC^*_{\mu,\xi}$ for the
$\delta$-perverse IC-extensions in this theory.
We denote its convolution action by $\star_*$. Let $\mathbb D'_{LX}$ be the  Verdier
duality in \cite[Lemma 103(i)]
{chenyi_2025singularitiesorbitclosuresloop}.
It satisfies
\[
  \mathbb D'_{LX}(e_0)\simeq\IC_0,
  \quad
  \mathbb D'_{LX}(\IC^*_{\mu,\xi})
    \simeq\IC_{\mu,\xi^\vee}.
\]
On the bounded proper convolution correspondences
considered below, there is an isomorphism
\[
  \mathbb D'_{LX}(\cF\star_*M)
    \simeq
    \mathbb D_{\Gr}(\cF)
      \star_!\mathbb D'_{LX}(M),
\]
where $\star_*$ denotes the convolution of $*$-$\cD$ modules and $\star_!$ denotes the convolution of $!$-$\cD$ modules used elsewhere in the paper.
Indeed, duality exchanges the pullbacks defining the
two twisted external products and commutes with
proper direct image. In particular,
\[
  \mathbb D'_{LX}
    \bigl(\Sat(V_\lambda)\star_*e_0\bigr)
    \simeq
    \Sat(V_\lambda)\star_!\IC_0,
\]
since $\Sat(V_\lambda)$ is Verdier self-dual.

For $\lambda\in \mathbb X_*(T)^+$, let $V_{\lambda}\in \mathrm{Rep}(\check{G})$ be the 
irreducible representation of $\check{G}$ of highest weight $\lambda$ and $\Sat(V_\lambda)
\in \mathrm{Perv}(K\backslash \Gr)$ be the corresponding
perverse sheaf under the geometric Satake equivalence. Let $X$ be a symmetric $G$-variety.
For $\mu\in \mathbb X_*(S)^+$, let 
$$
LX_{\le \mu}:= \bigsqcup_{\lambda\le 
\mu} LX_{\lambda} \subset LX
$$
and equip it with the reduced scheme structure.
By \cite[Theorem 18.]{chenyi_2025singularitiesorbitclosuresloop} $LX_{\le \mu}$
is the closure of $LX_{\mu}$. 
Let $\Gr_{\leq\lambda}^{(\infty)}$ be the inverse image of $\Gr_{\leq\lambda}$ in $LG$.
\begin{proposition}
\label{prop:symmetric_convolution}
The action map $m:LG\times^{L^+G}LX\to LX$ restricts to a morphism 
$$
\Gr_{\leq\lambda}^{(\infty)} \times^{L^+G} LX_{\le \mu}
\to LX_{\le \overline{\lambda}+\mu}. 
$$
\end{proposition}

\begin{proof}
We first identify
$LX$ with its image in $LG$ under $\tau$ via \autoref{prop:symmetric_var_spherical_orbit}. The $G$-action on $X$ is given by  $g\cdot_\theta x=gx\theta(g)^{-1}$ and induces an isomorphism
$$
LX_\mu\simeq L^+G\cdot_\theta t^\mu
\subset L^+Gt^\mu L^+G.
$$
We claim that 
$$ 
G(\cO)t^{\lambda}G(\cO)\cdot_{\theta} t^{\mu} 
\subset \bigsqcup_{\nu\le \overline{\lambda}+\mu} 
G(\cO)t^{\nu}G(\cO). 
$$
Write $L=Z_G(S)$ and let $w_L$ be the longest element of
the Weyl group with respect to $B\cap L$.
By \autoref{def:symmetric_relative_cocharacters}, the coweight $\lambda^\theta:=-w_L\theta(\lambda)$
is dominant. Moreover $w_L$ fixes $S$ pointwisely and commutes with
$\theta$. Hence $\overline{\lambda^\theta}=\overline{\lambda}$,
and $\lambda^\theta+\theta(\lambda)
=(1-w_L)\theta(\lambda)$ belongs to the coroot lattice.

Let $g\in G(\cO)t^\lambda G(\cO)$
and $x\in LX_\mu(k)$.
Then $x\in G(\cO)t^\mu G(\cO)$, and $\theta(g)^{-1}\in
G(\cO)t^{\lambda^\theta}G(\cO)$. Let $\nu\in \mathbb X_*(S)^+$  be 
so that $gx\theta(g)^{-1}\in L^+G \cdot_{\theta}t^{\nu} 
\subset L^+Gt^{\nu}L^+G$. By
\cite[(2.1.17)]{zhu2016introduction}, we have 
$$
\nu\le\lambda+\mu+\lambda^\theta.
$$
Applying $\overline{(-)}:=1-\theta$, we have 
$$
2\nu\le2(\overline{\lambda}+\mu)
$$
by noting that $\theta(\mu)=-\mu$ for $\mu\in \mathbb X_*(S)$ and 
$\overline{\lambda}=\overline{\lambda^{\theta}}$.
Furthermore $\overline{\lambda}+\mu-\nu
=
(\lambda+\mu+\lambda^\theta-\nu)
-(\lambda^\theta+\theta(\lambda))$
is a non-negative integral linear combination of coroots, which shows that $\nu\le\overline{\lambda}+\mu$.

To complete the proof, we show that $\lambda'\le\lambda$ and $\mu'\le\mu$ implies $\overline{\lambda'}+\mu'
\le\overline{\lambda}+\mu$. Write $\lambda-\lambda'
=\sum_{\alpha\in\Phi^+}n_\alpha\alpha^\vee,$ where $n_\alpha\in\ZZ_{\ge0}$. Since $\theta$ fixes $\Phi_L$ and sends
$\Phi^+\setminus\Phi_L$ to negative roots, we have
$$
\begin{aligned}
(\overline{\lambda}+\mu)-(\overline{\lambda'}+\mu')
&=(1-\theta)(\lambda-\lambda')+(\mu-\mu')\\
&=\sum_{\alpha\in\Phi^+\setminus\Phi_L}
n_\alpha (\alpha^\vee+(-\theta(\alpha))^\vee )
+(\mu-\mu').
\end{aligned}
$$
Note that each coroot appears in the sum is positive, and that $\mu-\mu'$ is
a nonnegative integral combination of positive coroots.
Thus $\overline{\lambda'}+\mu'\le\overline{\lambda}+\mu$.
\end{proof}

The following lemma follows easily from proper base change.
\begin{lem}\label{lem:support_convolution}
Let $A\subset\Gr$ be a closed $L^+G$-stable subscheme
contained in a finite union of Schubert varieties and
let $B\subset L^+X$ be a closed $L^+G$-stable subscheme.
Write $A^{(\infty)}$ for the inverse image of $A$ in $LG$.
If $\cF\in\cD_c(L^+G\backslash\Gr)$ is supported on $A$
and $\cG\in\cD_c(L^+G\backslash LX)$ is supported on $B$,
then $\cF\star_*\cG$ is supported on the image of
$A^{(\infty)}\times^{L^+G}B$ under the action map $m$.
\end{lem}

\begin{remark}
At the level of $k$-points the above result has been proved
by \cite[Lemma 5.2.2.]{SV17}. 
\end{remark}
\begin{proposition}[Support of convolution]
\label{prop:support_convolution}
Let $X$ be a symmetric $G$-variety.
Then $\Sat(V_\lambda)*\IC_0$ is supported on $LX_{\le\overline\lambda}$.
\end{proposition}
\begin{proof}
The assertion follows immediately from
\autoref{prop:symmetric_convolution} with $\mu=0$ and
\autoref{lem:support_convolution}.
\end{proof}

\subsection{Generation by the basic object} In this section, we first compute the stalks of
$\Sat(V_\lambda)\star_*e_0$ in the $*$-$\cD$ module theory, then
apply Verdier duality to obtain the desired
IC-summands in the $!$-$\cD$ module theory and prove
\autoref{thm:spherical_sat_full}.
\label{section:sat_full}
\begin{p}
\label{prop:convolution_restrict_open_orbit}
Let $X=G/G^\theta$ be a symmetric $G$-variety such that
$Z_{G^\theta}(S)$ is connected. For every
$\lambda\in\mathbb X_*(T)^+$, the restriction
\[
  i_{\overline\lambda}^*
    \bigl(\Sat(V_\lambda)\star_*e_0\bigr)
\]
is a nonzero finite direct sum of shifts of
$\underline{k}_{LX_{\overline\lambda}}$ in
$\cD_c^*(L^+G\backslash LX_{\overline\lambda})$.
Moreover, a shift of
$\IC^*_{\overline\lambda,\mathrm{triv}}$
is a direct summand of
$\Sat(V_\lambda)\star_*e_0$ in
$\cD_c^*(L^+G\backslash LX)$.
\end{p}
\begin{proof}
By
\autoref{prop:symmetric_convolution}, we have the bounded
convolution morphism
$$
m:\Gr_{\leq\lambda}^{(\infty)}\times^{L^+G}L^+X
\longrightarrow LX_{\le\overline\lambda}.
$$
As in the proof of \autoref{lem:support_convolution},
we can identify $\Gr_{\leq\lambda}^{(\infty)}\times^{L^+G}L^+X$ with 
$\{(gL^+G,z)\in \Gr_{\leq\lambda}\times LX_{\le\overline\lambda}
\mid g^{-1}z\in L^+X\}$. Thus $m$ is projective. 

We claim that $m$ is finitely presented. Choose a $G$-equivariant closed embedding $X\hookrightarrow V$
into a finite-dimensional representation and a finite
collection of quasi-compact \'etale charts on
$\Gr_{\leq\lambda}$ trivializing $LG\to\Gr$.
On each chart, we choose a representative $g\in LG$.
We can choose a uniform bound for the poles of $z\in LX_{\leq\overline\lambda}$ and $g^{-1}$.
Hence $g^{-1}z\in t^{-N}L^+V$ for some $N$, and the
condition $g^{-1}z\in L^+X$ is equivalent to the
vanishing of the finitely many negative Laurent
coefficients of $g^{-1}z$.
This proves the claim.

The equivariant decomposition theorem \cite[Proposition 70 and the proof of
Theorem 93]{chenyi_2025singularitiesorbitclosuresloop}
implies that
$$
\Sat(V_\lambda)\star_*e_0\cong
\bigoplus_{\nu\le\overline\lambda,\,n\in\mathbb Z}
\IC_{\nu,\mathrm{triv}}^*[n]^{\oplus m_{\nu,n}},
$$
where only finitely many multiplicities are nonzero.
Here the coefficient local systems are constant by
\autoref{lem:symmetric_stabilizers}.  Let $x_0=eG^\theta$
and $x=t^\lambda x_0$. Since
$\tau(x)=t^{\lambda-\theta(\lambda)}$, we have
$x\in LX_{\overline\lambda}$.
The fibre of $m$ over $x$ is
$$
F_x=\{gL^+G\in \Gr_{\leq\lambda}\mid g^{-1}x\in L^+X\}.
$$
In particular  $p=t^\lambda L^+G\in F_x$.
Consider the following Cartesian diagram,
$$
\begin{tikzcd}[column sep=huge, row sep=large]
F_x \arrow[r, hook, "j_x"] \arrow[d, "q_x"']
    \arrow[dr, phantom, "\lrcorner", very near start]
  & \Gr_{\leq\lambda}^{(\infty)}\times^{L^+G}L^+X
      \arrow[d, "m"] \\
\{x\} \arrow[r, hook, "i_x"']
  & LX_{\leq\overline\lambda}.
\end{tikzcd}
$$
Proper base change  identifies
the stalk  $(\Sat(V_\lambda)\star_* e_0)_x$ with $R\Gamma (F_x,\Sat(V_\lambda)|_{F_x}  )$.

Since $\Gr_{\leq\lambda}$ is normal \cite{BR23}, the intersection complex
$\Sat(V_\lambda)$ satisfies $\mathcal H^j (\Sat(V_\lambda)  )=0$
  for $j<-\dim\Gr_\lambda$ and 
  $\mathcal H^{-\dim\Gr_\lambda} (\Sat(V_\lambda)  )
  \cong\underline{k}_{\Gr_{\leq \lambda}}$.
In the hypercohomology spectral sequence $$E_2^{a,b}
=
\mathrm {H}^a (F_x,\mathcal H^b(\Sat(V_\lambda))|_{F_x} )
\Longrightarrow
\mathbb H^{a+b}(F_x,\Sat(V_\lambda)|_{F_x}),$$ the only nonzero
term that can contribute to total degree $-\dim\Gr_\lambda$ is 
$\mathrm{H}^0(F_x,\underline{k}_{F_x})$, and no differential enters
or leaves this term. 
Hence
$$
   \mathrm{H}^{-\dim\Gr_\lambda}R\Gamma (F_x,\Sat(V_\lambda)|_{F_x}  )
  \cong \mathrm{H}^0(F_x,\underline{k}_{F_x})\ne0
$$
since $F_x$ contains $p$.
Thus the convolution has a nonzero
stalk at $x$.

Every summand with $\nu<\overline\lambda$ in the preceding
decomposition restricts to zero on $LX_{\overline\lambda}$.
Consequently $m_{\overline\lambda,n}>0$ for some $n$.
The decomposition and the restriction formula for
IC-extensions prove both assertions.

\end{proof}

\begin{thm}
\label{thm:spherical_sat_full}
Let $X$ be a symmetric $G$-variety satisfying 
\autoref{assu:symmetric_condition_satfull}.
We have 
$$
\cD_{c}(L^+G\backslash LX)^{\Sat}\simeq
\cD_c(L^+G\backslash LX). 
$$
\end{thm}
\begin{proof}
By condition (1) of
\autoref{assu:symmetric_condition_satfull} and
\autoref{lem:symmetric_stabilizers}, the group $G^\theta$
and all stabilizers of $L^+G$-orbits in $LX$ are connected.
Thus the constant local system is the only irreducible
$L^+G$-equivariant local system on each orbit.
By \autoref{cor:spherical_IC_generation} it suffices
to show that $\IC_{\mu,\mathrm{triv}}$ belongs to
$\cD_c(L^+G\backslash LX)^{\Sat}$ for every
$\mu\in\mathbb X_*(S)^+$.

For any $\mu\in\mathbb X_*(S)^+$, Condition (2) of
\autoref{assu:symmetric_condition_satfull} shows that there exits
$\lambda\in\mathbb X_*(T)^+$ with $\overline\lambda=\mu$.
By \autoref{prop:convolution_restrict_open_orbit}
there is an integer $n$ such that
$\IC_{\mu,\mathrm{triv}}[n]$ is a direct summand of
$\Sat(V_\lambda)*\IC_0$ in $\cD_c(L^+G\backslash LX)$.
The latter belongs to
$\cD_c(L^+G\backslash LX)^{\Sat}$ by definition.
Since this subcategory is thick, it is closed under
direct summands and shifts and hence contains
$\IC_{\mu,\mathrm{triv}}$.
This proves the assertion.
\end{proof}

\begin{remark}
When $X=\GL_p\times \GL_q\backslash \GL_{p+q}$, the 
above result is proved in \cite[Proposition 2.4.6.]{chen2026relativelanglandsdualitylinear}. 
\end{remark}

\section{Relative Tensor  Products of Automorphic Categories}\label{section:construction_of_the_functor}

Let $X$ be as in \autoref{def:polarized_hyperspherical}. By \Cref{thm:CD_counit}, we get a fully faithful functor of $\cD^*(LG)$-modules
$$
\cD(\Gr)\otimes_{\cH_{K}}\cD(K\backslash LX)\rightarrow \cD(LX).
$$
Taking $I$-invariants, \Cref{lem:invrariance_lemma}  gives rise to a fully faithful functor $(\cD(\Gr)\otimes_{\cH_K}\cD(K\backslash LX))^I\rightarrow \cD(I\backslash LX)$.
Since $I$ has a prounipotent radical of finite
codimension, taking $I$-invariants identifies with
taking $I$-coinvariants and is therefore continuous
and $\mathrm{DGCat}_{\mathrm{cont}}$-linear
\cite[\S2.4.3 and the proof of Proposition 3.1.3]
{campbell2021affine}. We thus obtain the following fully faithful functor
\begin{equation}\label{equation:I-equiv-functor}
    F: \cD(I\backslash \Gr)\otimes_{\cH_K}\cD(K\backslash LX)\rightarrow \cD(I\backslash LX).
\end{equation}

\label{functor:renormalized}

A crucial step in the proof of our main theorem is to establish the following equivalence
\begin{thm}\label{thm:F_cstr_Sat}
The functor (\ref{equation:I-equiv-functor}) induces an equivalence
$$ 
F_{{c}}^\mathrm{Sat}:\cD_c(I\backslash \Gr)\otimes_{\cH_{{c}, K}}\cD_c(K\backslash LX)^\mathrm{Sat}\simeq \cD_c(I\backslash LX)^{\rm Sph, Sat}.
$$
\noi   
\end{thm} 

The rest of this section proves this theorem. 

\subsection{Relative tensor product of coherent $\cD$-modules.}
We first construct a fully faithful embedding
\begin{equation}\label{functor:embedding_of_small_tensor}
    \iota:\cD_{{c}}(I\backslash \Gr)\otimes_{\cH_{c,K}}\cD_{{c}}(K\backslash LX)\rightarrow\cD(I\backslash \Gr)\otimes_{\cH_{K}}\cD(K\backslash LX).
\end{equation}

Consider the following  {unrenormalization functors}
\begin{itemize}
    \item $\theta_1:\cD^\mathrm{ren}(I\backslash \Gr)\rightarrow \cD(I\backslash \Gr)$,
    \item $\theta_2:\cH^\mathrm{ren}_K\rightarrow \cH_K$ and
    \item  $\theta_3:\cD^\mathrm{ren}(K\backslash LX)\rightarrow \cD(K\backslash LX)$.
\end{itemize}
Recall the relative tensor product of small categories cf.\ \Cref{def:small_relative_tensor},
$$
\cD_{{c}}(I\backslash \Gr)\otimes_{\cH_{c,K}}\cD_{{c}}(L^+G\backslash LX):=(\cD^\mathrm{ren}(I\backslash \Gr)\otimes_{\cH^\mathrm{ren}_K}\cD^\mathrm{ren}(K\backslash LX))^\omega.
$$
\begin{lemma}\label{lemma:unrenormalization}
    Unrenormalization functors $\theta_i$'s induce a functor 
    $$
\Theta: \cD^\mathrm{ren}(I\backslash \Gr)\otimes_{\cH_{K}^\mathrm{ren}}\cD^\mathrm{ren}(K\backslash LX)\rightarrow \cD(I\backslash \Gr)\otimes_{\cH_{K}}\cD(L^+G\backslash LX) .
    $$
\end{lemma}

\begin{proof}
    Recall that with our convention $\cD(\bullet)=\cD^!(\bullet)$
convolution is given by $m_*p^!$ along the diagram
$$
\begin{tikzcd}
    K\backslash \Gr\times K\backslash \Gr& K\backslash LG\times^{K} \Gr\arrow [l,"p",swap] \arrow[r,"m"]& K\backslash \Gr,
\end{tikzcd}
$$   
where $p$ is the natural projection morphism and $m$ is the ind-proper convolution morphism. 
    Convolution preserves coherent objects by \autoref{lem:convolution_preserves_coherence}, and its renormalized
form is the Ind-extension of the coherent convolution functor.
The ordinary and renormalized convolution functors therefore commute
with unrenormalization since the two continuous functors coincide on
coherent pure tensors. The same argument applies to their
associativity and unit identifications. Thus $\theta_2$ is a
monoidal functor. In addition $\theta_1$, resp. $\theta_3$ is
a functor of right, resp. left $\cH_K^{\rm ren}$-modules.
By the above argument and the functoriality of Lurie relative tensor
product \cite[\S4.4.2]{HA}, $\theta_i$'s induce the desired functor $\Theta$.
\end{proof}
We have the following commutative diagram
$$
\begin{tikzcd}
    \cD^\mathrm{ren}(I\backslash \Gr)\otimes_{\cH_{K}^\mathrm{ren}}\cD^\mathrm{ren}(L^+G\backslash LX) \arrow[r,"\Theta"] & \cD(I\backslash \Gr)\otimes_{\cH_{K}}\cD(L^+G\backslash LX) \\
    \\
    \cD_c(I\backslash \Gr)\otimes_{\cH_{c,K}}\cD_c(K\backslash LX) \arrow[uu,"i"] \ar[uur, "\iota",swap],
\end{tikzcd}
$$
where functor $i$ is the natural embedding and $\iota:=\Theta\circ i$ is the desired functor (\ref{functor:embedding_of_small_tensor}). We will prove that $\iota$ is a fully faithful embedding.

\subsection{Compact generations}
\label{Compact generation} In this section, we find  generators of the  relative tensor product $\cD^{\rm ren}(I\backslash \Gr)\otimes_{\cH_K^{\rm ren}} \cD^{\rm ren}(K\backslash LX)$ whose mapping spaces are easier to compute.  
\begin{lem}\label{lemma:compact_generation}
    The category $\cD^\mathrm{ren}(I\backslash \Gr)\otimes_{\cH_{K}^\mathrm{ren}}\cD^\mathrm{ren}(K\backslash LX)$ is compactly generated by objects  $\ins^{\ren} (\cF\otimes \cG)$ for $\cF\in \cD_c(I\backslash \Gr)$ and $\cG\in\cD_c(K \backslash LX)$. 
\end{lem}
\begin{proof}
We know by \cite{campbell2024proof} that $\cH^\mathrm{ren}_K$ is a rigid  monoidal category. The statement follows directly from \autoref{coro:compact_objects_in_tensor_rigid}.
\end{proof}
By \autoref{def:small_relative_tensor} the source of
$\iota$ is the full subcategory of compact objects in
the category of \autoref{lemma:compact_generation}.
Thus \autoref{cor:thick_vs_generation} shows that it is
thickly generated by
$$
  \ins^{\ren}(\cF\otimes\cG),
  \quad
  \cF\in\cD_c(I\backslash\Gr),\quad
  \cG\in\cD_c(K\backslash LX).
$$
Since $\iota$ is exact and $k$-linear,
\autoref{rem:thick_properties}(3) reduces its full
faithfulness to proving that the induced maps on
mapping complexes between every pair of these
generators are equivalences.
\subsection{Mapping spaces}

To streamline notation 

we adopt the following conventions till the end of this section: 

\begin{itemize}
    \item $i:=\mathrm{ins}$, $i^{\rm r}:=\mathrm{ins}^{\rm ren}$, 
    \item $i^R:=\mathrm{ins}^R$, $i^{\mathrm{r},R}:=\mathrm{ins}^{\mathrm{ren},R}$ the right adjoint functors,
    \item   $\cS_I:=\cD(I\backslash \Gr), \cS_{c,I}:=\cD_c(I\backslash \Gr), \cS_I^{\rm ren}:=\cD^{\rm ren}(I\backslash \Gr)$,
  \item $\cT_H:=\cD(H\backslash LX), \cT_{c,H}:=\cD_c(H\backslash LX), \cT_H^{\rm ren}:=\cD^{\rm ren}(H\backslash LX)$ for $H=I,K$. 
\end{itemize}

Since the bar complex constructing Lurie's relative tensor product is functorial, $\Theta$ and $\theta_1\otimes \theta_3$ are intertwined by the insertion functors
 \begin{align*}
 & i:\cS_I\otimes\cT_K \rightarrow \cS_I\otimes_{\cH_K}\cT_K,\\
    &   i^{\rm r}:\cS_I^{\rm ren}\otimes\cT_K^{\rm ren} \rightarrow \cS_I^{\rm ren}\otimes_{\cH^\mathrm{ren}_K}\cT_K^{\rm ren}
 \end{align*}
 i.e. the following diagram commutes.
 \begin{equation}\label{commdiagran:2}
 \begin{tikzcd}
 {\cS_I^{\rm ren}\otimes_{\cH^\mathrm{ren}_K}\cT_K^{\rm ren}} \arrow[r,"\Theta"]    & {\cS_I\otimes_{\cH_K}\cT_K} \\
  \\
  {\cS_I^{\rm ren}\otimes\cT_K^{\rm ren} } \arrow[r,"\theta"] \arrow[uu," i^{\rm r}"]&     {\cS_I\otimes\cT_K}\arrow[uu,"i",swap]
  \end{tikzcd}
 \end{equation}
 where $\theta:=\theta_1\otimes \theta_3$. 

 The insertion functors both admit right adjoints by presentability,
which we denote by $i^R$ and $i^{\mathrm{r},R}$.
\begin{lemma}\label{lemma:fully_faithful}
    For any $\cF\otimes \cG, \cF'\otimes \cG'\in\cS_{c,I}\otimes\cT_{c,K}$   there exists a morphism 
    $$
    \vartheta: \underline{\mathrm{Map}}_{\cS_I^{\rm ren}\otimes\cT_K^{\rm ren}}(\cF\otimes\cG,i^{\mathrm{r},R}\circ i^{\rm r}(\cF'\otimes\cG'))\rightarrow \underline{\mathrm{Map}}_{\cS_I^{\rm }\otimes\cT_K^{\rm }}(\theta(\cF\otimes\cG),i^R\circ i\circ\theta(\cF'\otimes\cG'))
    $$
    making the following diagram commutes
 \begin{equation}\label{commdia:1}
 \begin{tikzcd}[column sep=small]
  \underline{\mathrm{Map}}_{\cS_I^{\rm ren}\otimes_{\cH^\mathrm{ren}_K}\cT_K^{\rm ren}}( {i}^{\rm r} (\cF\otimes\cG), {i} ^{\rm r}(\cF'\otimes\cG'))\arrow[dd,"\simeq",swap] \arrow[r,"\Theta"]    & \underline{\mathrm{Map}}_{\cS_I\otimes_{\cH_K}\cT_K} (\Theta( {i}^{\rm r} (\cF\otimes\cG)),\Theta( {i} ^{\rm r}(\cF'\otimes\cG')))\arrow[dd,"\simeq"]\\
  \\
  \underline{\mathrm{Map}}_{\cS_I^{\rm ren}\otimes\cT_K^{\rm ren} } (\cF\otimes\cG,i^{\mathrm{r},R}\circ {i}^{\rm r} (\cF'\otimes\cG'))\arrow[r,"\vartheta"] &     \underline{\mathrm{Map}}_{\cS_I\otimes\cT_K}(\theta(\cF\otimes\cG),i^R\circ {i} \circ\theta(\cF'\otimes\cG')),
 \end{tikzcd}
 \end{equation}
 where 
 \begin{enumerate}
     \item the top arrow is induced by $\Theta$,
     \item the left vertical equivalence are given by the adjunction $( {i}^{\rm r} ,i^{\mathrm{r},R})$,
     \item the right vertical equivalence is given by the commutative diagram (\ref{commdiagran:2}) and the adjoint pair $( {i} ,i^{R})$.
     \end{enumerate}
\end{lemma}

\begin{proof}
Define $\vartheta$ by composing the inverse of the
left vertical adjunction equivalence, the map induced
by $\Theta$ and the right vertical adjunction
equivalence. The diagram then commutes by construction.
\end{proof}

 Thus to prove the full faithfulness of $\iota$ it suffices to prove $\vartheta$ is an equivalence. We prove this geometrically in the next subsection.

\subsection{Comparison of mapping spaces}
\label{subsec:geometric_comparison}
Let $s:I\backslash\Gr\longrightarrow K\backslash\Gr$ and
$\pi:I\backslash\Fl\longrightarrow I\backslash\Gr$ be the natural
projections. Both are smooth and proper of relative dimension
$d:=\dim(G/B)$, and we write $s^*:=s^![-2d]$ and $\pi^*:=\pi^![-2d]$
for the left adjoints of $s_*$ and $\pi_*$
\cite[Proposition 3.9.3]{raskin_dmodules_infinite_var}.
Let $e:BK\simeq K\backslash K/K\hookrightarrow K\backslash\Gr$ be the
base point and let $\mathbf1_K:=e_*\omega_{BK}$ be the monoidal unit
of $\cH_K$. Let
$$
\cP_0:=s^!\mathbf1_K\simeq s^*\mathbf1_K[2d].
$$
By base change $\cP_0\simeq e'_*\omega_{BI}$ for the base point
$e':BI\simeq I\backslash K/K\hookrightarrow I\backslash\Gr$, that is,
$\cP_0$ identifies with $\mathbf1_K$ with its equivariance restricted from $K$ to $I$.
For $w\in W_{\mathrm{ext}}$ let
$\Fl_w:=I\dot wI/I$ be the affine Schubert cell and write
$j_w:I\backslash\Fl_w\hookrightarrow I\backslash\Fl$ for the locally
closed embedding. We equip $\Fl$ with the dimension theory
$\delta(\Fl_w)=\ell(w)$. The standard and costandard objects are defined as
$$
\Delta_w:=(j_w)_!\mathcal L_{\mathrm{triv},\Fl_w}
\simeq(j_w)_!\,\omega_{I\backslash\Fl_w}[-\ell(w)],
\quad
\nabla_w:=(j_w)_*\mathcal L_{\mathrm{triv},\Fl_w}
\simeq(j_w)_*\,\omega_{I\backslash\Fl_w}[-\ell(w)].
$$
With this normalization
$\Delta_w\ast\nabla_{w^{-1}}\simeq\mathbf1_I\simeq\nabla_{w^{-1}}\ast\Delta_w$,
where $\mathbf1_I$ is the monoidal unit of $\cH_I$.

\begin{lem}\label{lem:generators_convolution_comparison}
\begin{enumerate}
 \item There is an equivalence
\[
\cS_{c,I} \simeq\Thick_{\cS_{c,I}}
\left\{
  \Delta_w\ast\cP_0\ \middle|\
  w\in W_{\mathrm{ext}}
  \textnormal{ is of minimal length in }wW
\right\}. 
\]
\item For $w\in W_{\mathrm{ext}}$, convolution with the standard object gives
autoequivalences
$L_w:=\Delta_w\ast(-):\cS_I\xrightarrow{\sim}\cS_I$ and 
$L_w^{\rm r}:=\Delta_w\ast(-):
\cS_I^{\rm ren}\xrightarrow{\sim}\cS_I^{\rm ren}$.
They are compatible with the right $\cH_K$ and $\cH_K^{\rm ren}$ actions
respectively, and there are isomorphisms
$$
\begin{gathered}
\theta_1L_w^{\rm r}\simeq L_w\theta_1,\\
\Theta\circ
(L_w^{\rm r}\otimes_{\cH_K^{\rm ren}}\mathrm{id})
\simeq
(L_w\otimes_{\cH_K}\mathrm{id})\circ\Theta.
\end{gathered}
$$
\end{enumerate}
\end{lem}

\begin{proof}
For $w\in W_{\mathrm{ext}}$ of minimal length in $wW$,
proper base change gives a natural isomorphism
$$
\Delta_w\ast\cP_0\simeq\pi_*\Delta_w.
$$
Since $\pi$ is proper, $\Delta_w\ast\cP_0$ equals the standard object for the $I$-orbit through $w$ in $\Gr$. After passing to a finite congruence subgroup quotient,
the stabilizer $I_\mu$ of any orbit $\Gr_\mu$
is an extension of $T$ by a unipotent group.
Thus $I_\mu$ is connected, and
$$
\cD_c(I\backslash\Gr_\mu)
\simeq
\Thick_{\cD_c(I\backslash\Gr_\mu)}
\{\mathcal L_{\mathrm{triv},\Gr_\mu}\}
$$
by \autoref{cons:generator_sheaves_on_orbit}.
We conclude the proof of the first assertion by finite recollement and
\autoref{lem:coherent_finite_stage}.

The functor $L_w$ is an autoequivalence
of $\cS_I$ with inverse given by convolution by
$\nabla_{w^{-1}}$ by
\cite[\S3.2, proof of Lemma 4(c)]{bezrukavnikov2016two}.
These functors preserve coherent objects and commute
with the right spherical action. Taking Ind-completions
 yields the autoequivalence $L_w^{\rm r}$.
Compatibility with unrenormalization gives
$\theta_1L_w^{\rm r}\simeq L_w\theta_1$.
Taking geometric realizations of the two-sided bar
constructions of the relative tensor products gives
the asserted compatibility with $\Theta$.
\end{proof}

\begin{lem}\label{lem:spherical_module_adjunctions}
There are adjunctions
$s^*:\cH_K\rightleftarrows\cS_I:s_*$
 and $s^{*,\rm r}:\cH_K^{\rm ren}
\rightleftarrows\cS_I^{\rm ren}:s_*^{\rm r}$
of right $\cH_K$- and $\cH_K^{\rm ren}$-module functors
respectively. Moreover we have canonical isomorphisms:
$\theta_1s^{*,\rm r}\simeq s^*\theta_2$ and
$\theta_2s_*^{\rm r}\xrightarrow{\sim}s_*\theta_1$.
\end{lem}
\begin{proof}
The map $s$ is representable, smooth and proper with
geometric fibre $K/I\simeq G/B$.
Thus $s^*\dashv s_*$ is an adjunction of continuous
functors by
\cite[Propositions 3.9.2--3.9.3]{raskin_dmodules_infinite_var}. We claim that functors preserve coherent objects.
For $s^*$ this follows directly from
\autoref{defn:coherent_D_mod}.
For $s_*$
we may restrict to a finite union of
spherical Schubert varieties where coherent $\cD$ module supports by \autoref{lem:coherent_finite_stage}.
After passing to a sufficiently deep congruence
quotient the claim follows from proper pushforward
for finite-type algebraic stacks
\cite[Proposition~10.6.2]{drinfeld2013some}.

Consider the following commutative diagram:
$$
\begin{tikzcd}[row sep=huge]
(I\backslash\Gr)\times(K\backslash\Gr)
  \arrow[d,"s\times\mathrm{id}"']
&
 I\backslash(LG\times^K\Gr)
  \arrow[l,"p_I"']
  \arrow[r,"m_I"]
  \arrow[d,"a"]
&
I\backslash\Gr
  \arrow[d,"s"]
\\
(K\backslash\Gr)\times(K\backslash\Gr)
&
 K\backslash(LG\times^K\Gr)
  \arrow[l,"p_K"]
  \arrow[r,"m_K"']
&
K\backslash\Gr.
\end{tikzcd}
$$
Here $p_\bullet$ are natural projections, $m_\bullet$  the convolution morphisms, and squares
are $2$-Cartesian For $\cF\in\cS_I$ and
$\cG,\cG_1,\cG_2\in\cH_K$ proper base change gives
$$
\begin{aligned}
s_*(\cF\ast\cG)&\simeq(s_*\cF)\ast\cG,\\
s^*(\cG_1\ast\cG_2)&\simeq s^*(\cG_1)\ast\cG_2.
\end{aligned}
$$
These identifications are compatible with the adjunction
unit and counit. Hence $s^*\dashv s_*$ is an adjunction
of right $\cH_K$-module functors. In particular $\cP_0\ast\cG\simeq s^!(\cG)=s^*(\cG)[2d]$.
Restricting to coherent objects and taking Ind-completions
gives the adjunction
$$
s^{*,\rm r}:\cH_K^{\rm ren}
\rightleftarrows\cS_I^{\rm ren}:s_*^{\rm r}
$$
of right $\cH_K^{\rm ren}$-module functors.

The canonical comparison
$\theta_1s^{*,\rm r}\simeq s^*\theta_2$ has right mate
$$
\theta_2s_*^{\rm r}\longrightarrow s_*\theta_1.
$$
This is an equivalence on coherent objects and hence on
all objects since both functors are continuous.
The compatibility with the module structures
and the adjunction units and counits is clear from discussion.
\end{proof}

\begin{proposition}\label{prop:monad_map_space}
Let $\cF,\cF'\in\cS_{c,I}$ and
$\cG,\cG'\in\cT_{c,K}$. The morphism
$$
\begin{aligned}
\vartheta:\,
&\underline{\Map}_{\cS_I^{\rm ren}\otimes\cT_K^{\rm ren}}
 (\cF\otimes\cG,
i^{\mathrm{r},R}\circ i^{\rm r}(\cF'\otimes\cG')  )
\longrightarrow
\underline{\Map}_{\cS_I\otimes\cT_K}
 (\theta(\cF\otimes\cG),
i^R\circ i\circ\theta(\cF'\otimes\cG')  )
\end{aligned}
$$
induced by $\theta$ is an equivalence.
\end{proposition}

\begin{proof}
By \autoref{lemma:fully_faithful}, it suffices to prove
that $\Theta$ induces an equivalence on mapping complexes
between $i^{\rm r}(\cF\otimes\cG)$ and
$i^{\rm r}(\cF'\otimes\cG')$.  Via \autoref{lem:generators_convolution_comparison}, we  further reduce to prove the assertion for
$\cF=\cP_0$. Tensoring the two  adjunctions of
\autoref{lem:spherical_module_adjunctions} with
$\cT_K^{\rm ren}$ and $\cT_K$ respectively yields the following isomorphisms
$$
\begin{aligned}
&\underline{\Map}_{
\cS_I^{\rm ren}\otimes_{\cH_K^{\rm ren}}\cT_K^{\rm ren}}
 (
i^{\rm r}(\cP_0\otimes\cG),
i^{\rm r}(\cF'\otimes\cG')
  )\\
&\simeq
\underline{\Map}_{\cT_K^{\rm ren}}
 (\cG,(s_*^{\rm r}\cF')\ast\cG'  )[-2d]\\
&\xrightarrow{\sim}
\underline{\Map}_{\cT_K}
 (\theta_3\cG,
(s_*\theta_1\cF')\ast\theta_3\cG'  )[-2d]\\
&\simeq
\underline{\Map}_{\cS_I\otimes_{\cH_K}\cT_K}
 (
i\theta(\cP_0\otimes\cG),
i\theta(\cF'\otimes\cG')
  ),
\end{aligned}
$$
where the middle morphism is induced by $\theta_3$.
It is an equivalence because
$s_*^{\rm r}\cF'\in\cH_{c,K}$, and the spherical action
preserves coherent objects by \autoref{lem:convolution_preserves_coherence}. Compatibility with the adjunction units and counits
identifies this composite with the map induced by $\Theta$. We thus complete the proof.
\end{proof}
\subsection{Proof of \autoref{thm:F_cstr_Sat}}
\label{subsection:construction_Fc}
We complete the proof of \Cref{thm:F_cstr_Sat} in this subsection.
\begin{proof}
We have shown that  $\iota:    \cS_{c,I}\otimes_{\cH_{c,K}}\cT_{c,K} \rightarrow     \cS_I\otimes_{\cH_{K}}\cT_K$
is a fully faithful embedding. Composing with (\ref{equation:I-equiv-functor}) we obtain a fully faithful functor 
$$
F_{c}: \cS_{c,I}\otimes_{\cH_{c,K}}\cT_{c,K}\rightarrow \cT_I
$$
of $\cH_{c,I}$-modules. For any $\cF\in \cS_{c,I}$ and $\cG\in\cT_{c,K}$ we know that $F\circ i (\cF\otimes \cG)$ is given by pull-push along the convolution diagram
$$
\begin{tikzcd}
    I\backslash \Gr\times K\backslash LX& I\backslash LG\times^{K}LX\arrow [l,"p",swap] \arrow[r,"q"]& I\backslash LX,
\end{tikzcd}
$$
i.e. $F\circ i (\cF\otimes \cG)\simeq q_*p^!(\cF\boxtimes\cG)$.
By \autoref{lem:convolution_preserves_coherence}, convolution preserves coherent objects. It follows that $F_c$ takes values
in $\cT_{c,I}$. The same argument also applies with
$\cT_{c,K}$ replaced by $\cT_{c,K}^{\Sat}$ and gives  fully faithful functor 
 $$
 F_{c}':\cD_c(I\backslash \Gr)\otimes_{\cH_{c,K}}\cD_c(K\backslash LX)^\Sat \rightarrow \cD_c(I\backslash LX).
 $$
By \autoref{coro:compact_objects_in_tensor_rigid} and
\autoref{cor:thick_vs_generation} the source of $F_c'$
is thickly generated by the objects $\cF\otimes(\cG\ast\IC_0)$ for $\cF\in\cD_c(I\bs\Gr)$ and $\cG\in\cH_{c,K}$. Their images under $F_c'$ are
$$
  \cF\ast(\cG\ast\IC_0)
  \simeq
  (\cF\ast\cG)\ast\IC_0.
$$

Since $F_c'$ is fully faithful and exact,
\autoref{rem:thick_properties}(4) identifies the
essential image of $F_c'$ with the thick span of $(\cF*\cG)*\IC_0$.
The coherent right $\cH_{c,K}$-action implies that the essential
image of $F_c'$ is contained in
$\cD_c(I\backslash LX)^{\rm Sph, Sat}$.
The reverse containment follows immediately from taking $\cG=\mathbf1_K$.
Hence the essential image is precisely
$\cD_c(I\backslash LX)^{\rm Sph, Sat}$.
\end{proof}

\section{Comparison of two Satake subcategories}
\label{section:satake_comparison}
We compare the Satake subcategories $\cD_c(I\bs LX)^{\Sat}$ and $\cD_c(I\backslash LX)^{\rm Sph, Sat}$
 in
\autoref{def:satake_subcategory}. Throughout this section we write $k=\mathbb{C}$.
\subsection{Trace of multiplication.}
We first recall the following familiar result about the categorical dimension.
\begin{lemma}\cites[Remark 2.16]{sanders2022finite_etale}[Remark 3.10]{ben2013nonlinear}
\label{lem:trace}
Let $(\mathcal C,\otimes,\mathbf{1})$ be a $k$-linear symmetric monoidal category and let $A$ be a dualizable unital algebra object. Write $u:\mathbf{1}\to A$ and $\mu:A\otimes A\to A$ for its unit and multiplication respectively. Then there is a canonical morphism
$t_A:A\to\mathbf{1}$ satisfying
\begin{equation}\label{eq:trace-unit}
 t_Au=\dim_{\mathcal C}(A),
\end{equation}
where $\dim_{\mathcal C}(A)=\Tr(\id_A)$ is an endomorphism
of $\mathbf{1}$. If
$\dim_{\mathcal C}(A)=n\,\id_{\mathbf{1}}$ for
$n\in k^{\times}$, then $n^{-1}t_A$ retracts $u$.
\end{lemma}

\subsection{A retraction lemma}
Let $n:=|W|$ and recall $d:=\dim(G/B)$.
The natural projection $\pi':BB\longrightarrow BG$
is smooth and proper with geometric fibre $G/B$.
\begin{lemma}\label{lem:universal}
The object $\cA:=(\pi')_*\omega_{BB}[-2d]\in \cD^!(BG)$ has a canonical unital algebra structure with
unit $u:\omega_{BG}\longrightarrow \cA$
induced by the adjunction $(\pi')^*\dashv(\pi')_*$. The object $\cA$ is dualizable, and
\begin{equation}\label{eq:universal-splitting}
 t_{\cA}u=n\,\id_{\omega_{BG}}.
\end{equation}
\end{lemma}

\begin{proof}
    Since $\pi'$ is smooth of relative dimension  $d$, we have $(\pi')^!=(\pi')^*[2d]$ and $(\pi')^!\dashv  ((\pi')_*[-2d])$. By the symmetric monoidality of $(\pi')^!$ its right adjoint $(\pi')_*[-2d]$ has a canonical lax symmetric monoidal structure. The first assertion follows directly from  isomorphisms  $(\pi')_*(\pi')^!\omega_{BG}[-2d]
 \simeq(\pi')_*(\pi')^*\omega_{BG}
 \simeq\cA$.

Let $a:\mathrm{pt}\to BG$ be the natural projection. Proper base change gives a perfect complex
\begin{equation}\label{eq:fibre}
 a^!\cA\simeq \mathrm{R}\Gamma_{\dR}(G/B).
\end{equation}
By smooth descent we have 
$\cD^!(BG)\simeq
  \varprojlim_{[n]\in\Delta}\cD^!(G^n)$
as symmetric monoidal categories. Via this identification, every component of $\cA$
is a symmetric monoidal pullback of $a^!\cA$ and is therefore
dualizable. Thus $\cA$ is  dualizable.
The categorical dimension of $\cA$ belongs to
$$
 \End^0_{\cD^!(BG)}(\omega_{BG})
   =\mathrm{H}^0_{\dR}(BG)=k.
$$
Since symmetric
monoidal pullback preserves categorical traces,
$$
 \dim_{\cD^!(BG)}(\cA)
 =\chi_{\dR}(G/B)\,\id_{\omega_{BG}}
 =\vert W\vert\,\id_{\omega_{BG}},
$$
where the last equality follows from the Bruhat decomposition.
Then \eqref{eq:universal-splitting} follows from \autoref{lem:trace}.
\end{proof}

Recall the projection
$\pi:I\backslash\Fl\longrightarrow I\backslash\Gr$
from \autoref{subsec:geometric_comparison}, and write
$\eta^\pi:\mathrm{id}\longrightarrow\pi_*\pi^*$
for the unit of the adjunction $\pi^*\dashv\pi_*$.
\begin{proposition}\label{prop:retraction}
There is a natural transformation $\beta:\pi_*\pi^*\longrightarrow\id_{\cD(I\backslash\Gr)}$
such that $\beta\circ\eta^\pi=\id$.
In particular every object of $\cD(I\backslash\Gr)$ is a
retract of its image under $\pi_*\pi^*$.
\end{proposition}
\begin{proof}
Let $K_1:=\ker(\mathrm{ev}_0:K\to G)$ be the kernel of the evaluation at $0$ map. The left $I$-action on
the right $G$-torsor $LG/K_1\to\Gr$ commutes with $G$, so this
torsor descends to $I\bs\Gr$. Its classifying map
$\gamma:I\bs\Gr\to BG$ fits into a Cartesian square
$$
\begin{tikzcd}
 I\bs\Fl \ar[r] \ar[d,"\pi"'] & BB \ar[d,"\pi'"] \\
 I\bs\Gr \ar[r,"\gamma"'] & BG.
\end{tikzcd}
$$

Proper base change and the projection formula yield $\pi_*\pi^*\cF\simeq\gamma^!\cA\overset{!}{\otimes} \cF$ and $\eta^\pi_{\cF}=(\gamma^!u)\overset{!}{\otimes}\id_{\cF}$. Define
$$
  \beta_{\cF}:\pi_*\pi^*\cF
  \simeq\gamma^!\cA\overset{!}{\otimes} \cF
  \xrightarrow{(n^{-1}\gamma^!t_{\cA})\overset{!}{\otimes} \id_{\cF}}
  \cF.
$$
We thus complete the proof by \autoref{lem:universal}.
\end{proof}

\subsection{Comparison of generators}

Let $r:I\bs LX\longrightarrow K\bs LX$
be the natural projection. 

\begin{proposition}
    \label{lem:satake-generators}
There is an equivalence of full subcategories
\begin{equation}\label{eq:iwahori-mixed-generators}
  \cD_c(I\bs LX)^{\Sat}
  \simeq
  \cD_c(I\backslash LX)^{\rm Sph,Sat}.
\end{equation}
\end{proposition}

\begin{proof}
Let $\cG\in\cH_{c,I}$ and $\cM\in\cD_c(K\bs LX)$. Base change gives
canonical isomorphisms
\begin{equation}\label{eq:basepoint-convolution}
  \cG\ast \cP_0\simeq \pi_*\cG,
  \quad
  \cP_0\ast\cM\simeq r^!\cM.
\end{equation}
Hence $\cG\ast r^!\IC_0\simeq(\cG\ast \cP_0)\ast\IC_0$ for
$\cG\in\cH_{c,I}$.
Since $\cG\ast \cP_0\in\cD_c(I\bs\Gr)$, every Iwahori generator in
\autoref{def:satake_subcategory} belongs to the right-hand
side of \eqref{eq:iwahori-mixed-generators}.

Conversely let $\cF\in\cD_c(I\bs\Gr)$.
The object $\pi^*\cF$ is coherent because $\pi$ is smooth.
By \eqref{eq:basepoint-convolution} we have
$$
(\pi_*\pi^*\cF)\ast\IC_0
  \simeq
  ((\pi^*\cF)\ast \cP_0)\ast\IC_0
  \simeq
  (\pi^*\cF)\ast r^!\IC_0
  \in\cD_c(I\bs LX)^{\Sat}.
$$
Thus \autoref{prop:retraction} exhibits $\cF\ast\IC_0$ as a
retract of $(\pi_*\pi^*\cF)\ast\IC_0$. Since
$\cD_c(I\bs LX)^{\Sat}$ is thick, the reverse inclusion
follows.
\end{proof}

\section{Proof of the Main Theorem}\label{section:proof_of_the_main_theorem}

With previous preparations we prove the tamely ramified relative local conjecture in this section. 

The last piece of ingredient for the proof is the following sheared version of the tensor product formula in  \cite[Theorem 4.7]{ben2010integral}:
\begin{lemma}\label{lem:spectral_relative_tensor}
There is a natural equivalence
$$
\Perf(\tilde{\check{\mathfrak g}}^*[2]/\check G)
  \otimes_{\Perf(\check{\mathfrak g}^*[2]/\check G)}
  \Perf(\sh^{1/2}(\check M)/\check G)\simeq
  \Perf  ( 
    \sh^{1/2}  ( 
      \tilde{\check{\mathfrak g}}^*(2)
      \times_{\check{\mathfrak g}^*(2)}^{\mathbf R}
      \check M
      )/\check G
    ),
$$
which sends the image of 
$\cO_{\tilde{\check{\mathfrak g}}^*[2]/\check G}\otimes 
    \cO_{\sh^{1/2}(\check M)/\check G}$
under the insertion functor to 
$\cO_{\sh^{1/2}  ( 
    \tilde{\check{\mathfrak g}}^*(2)
    \times_{\check{\mathfrak g}^*(2)}^{\mathbf R}
    \check M
    )/\check G}$.
\end{lemma}

\begin{proof}
The space $\check M$ is affine by its
construction \cite[Lemma 3.4.9 and \S4.5]{BZSV}.
It allows us to apply the algebraic descriptions
in \autoref{def:sheared_spectral_categories}
and \autoref{def:shifted_vector_bundles}.

We first prove a presentable category analogue of the desired equivalence. By \autoref{lem:shearing_even_base},
$\sh^{1/2}_{\mathrm{ungr}}
  \mathcal O(\check{\mathfrak g}^*(2))$
and
$\sh^{1/2}_{\mathrm{ungr}}
  \mathcal O(\check{\mathfrak n}^{\perp}(2))$
are commutative, and
$\sh^{1/2}_{\mathrm{ungr}}\mathcal O(\check M)$
is an associative algebra with a central action of
$\sh^{1/2}_{\mathrm{ungr}}
  \mathcal O(\check{\mathfrak g}^*(2))$.
Write $\check\pi:B\check B\longrightarrow B\check G$ for the natural projection. We have  
\begin{equation}\label{eq:spectral-relative-tensor-presentable}
\begin{aligned}
  &\mathrm{LMod}_{
    \sh^{1/2}_{\mathrm{ungr}}
      \mathcal O(\check{\mathfrak n}^{\perp}(2))}
    (\QCoh(B\check B))
  \otimes_{
    \mathrm{Mod}_{
      \sh^{1/2}_{\mathrm{ungr}}
        \mathcal O(\check{\mathfrak g}^*(2))}
      (\QCoh(B\check G))}
  \mathrm{LMod}_{
    \sh^{1/2}_{\mathrm{ungr}}\mathcal O(\check M)}
    (\QCoh(B\check G))\\
  &\simeq
  \mathrm{LMod}_{
    \sh^{1/2}_{\mathrm{ungr}}
      \mathcal O(\check{\mathfrak n}^{\perp}(2))
    \otimes^{\mathbf L}_{
      \check\pi^*\sh^{1/2}_{\mathrm{ungr}}
        \mathcal O(\check{\mathfrak g}^*(2))}
    \check\pi^*\sh^{1/2}_{\mathrm{ungr}}
      \mathcal O(\check M)}
    (\QCoh(B\check B))\\
  &\simeq
  \mathrm{LMod}_{
    \sh^{1/2}_{\mathrm{ungr}}
    \mathcal O  ( 
      \check{\mathfrak n}^{\perp}(2)
      \times_{\check{\mathfrak g}^*(2)}^{\mathbf R}
      \check M
      )}
    (\QCoh(B\check B)),
\end{aligned}
\end{equation}
where the first equivalence following from \cite[Proposition 4.1(1)]{ben2010integral}, and the second equivalence is given by   \autoref{lem:shearing_even_base}. 

We are left with identifying compact objects on the first  and last term of \eqref{eq:spectral-relative-tensor-presentable}. The stacks $B\check G$ and $B\check B$ are perfect
by \cite[Corollary 3.22]{ben2010integral}.
Hence all categories appear in the first term are compactly generated
by free modules on perfect objects of their respective
base stacks. Moreover, compact objects in $\mathrm{Mod}_{
  \sh^{1/2}_{\mathrm{ungr}}
    \mathcal O(\check{\mathfrak g}^*(2))}
  (\QCoh(B\check G))$ are dualizable. To see this,  for $\cP\in\Perf(B\check G)$, we have
$$
    ( 
    \sh^{1/2}_{\mathrm{ungr}}
      \mathcal O(\check{\mathfrak g}^*(2))\otimes\cP
    )^\vee
  \simeq
  \sh^{1/2}_{\mathrm{ungr}}
    \mathcal O(\check{\mathfrak g}^*(2))\otimes\cP^\vee.
$$
Here the dual on the LHS is taken in
$\mathrm{Mod}_{
  \sh^{1/2}_{\mathrm{ungr}}
    \mathcal O(\check{\mathfrak g}^*(2))}
  (\QCoh(B\check G))$, 
whereas $\cP^\vee$ is taken in $\QCoh(B\check G)$.
Since dualizable objects are closed under finite colimits and retracts, every compact object of $\mathrm{Mod}_{
  \sh^{1/2}_{\mathrm{ungr}}
    \mathcal O(\check{\mathfrak g}^*(2))}
  (\QCoh(B\check G))$ is dualizable. Note that action by a dualizable object has a continuous right
adjoint given by action by its dual, the actions of $\mathrm{Mod}_{
  \sh^{1/2}_{\mathrm{ungr}}
    \mathcal O(\check{\mathfrak g}^*(2))}
  (\QCoh(B\check G))^\omega$ on $\mathrm{LMod}_{
    \sh^{1/2}_{\mathrm{ungr}}
      \mathcal O(\check{\mathfrak n}^{\perp}(2))}
    (\QCoh(B\check B))$ and  $\mathrm{LMod}_{
    \sh^{1/2}_{\mathrm{ungr}}\mathcal O(\check M)}
    (\QCoh(B\check G))$  preserve
compact objects.

Since $\check M$ is affine, the derived fibre product
$\check{\mathfrak n}^{\perp}(2)
\times_{\check{\mathfrak g}^*(2)}^{\mathbf R}\check M$
is affine. Thus 
$$(\mathrm{LMod}_{\sh^{1/2}_{\mathrm{ungr}}
  \mathcal O  ( 
    \check{\mathfrak n}^{\perp}(2)
    \times_{\check{\mathfrak g}^*(2)}^{\mathbf R}
    \check M
    )}(\QCoh(B\check B)))^\omega\simeq \Perf  ( 
    \sh^{1/2}  ( 
      \check{\mathfrak n}^{\perp}(2)
      \times_{\check{\mathfrak g}^*(2)}^{\mathbf R}
      \check M
      )/\check B
    ).$$
Taking compact objects in \eqref{eq:spectral-relative-tensor-presentable} therefore provides an equivalence
$$
\begin{aligned}
  &\Perf(\tilde{\check{\mathfrak g}}^*[2]/\check G)
  \otimes_{\Perf(\check{\mathfrak g}^*[2]/\check G)}
  \Perf(\sh^{1/2}(\check M)/\check G)\\
  &\simeq
  \Perf  ( 
    \sh^{1/2}  ( 
      \check{\mathfrak n}^{\perp}(2)
      \times_{\check{\mathfrak g}^*(2)}^{\mathbf R}
      \check M
      )/\check B
    ).
\end{aligned}
$$
The statement follows from the natural isomorphism
$  ( 
    \check{\mathfrak n}^{\perp}(2)
    \times_{\check{\mathfrak g}^*(2)}^{\mathbf R}
    \check M
    )/\check B
  \simeq
    ( 
    \tilde{\check{\mathfrak g}}^*(2)
    \times_{\check{\mathfrak g}^*(2)}^{\mathbf R}
    \check M
    )/\check G$
\end{proof}

Finally we prove our main theorem.
\begin{thm}\label{thm:iwahori_proof}
Let $X$ be a spherical variety as in
\autoref{def:polarized_hyperspherical} and assume that
\autoref{Conj:intro} holds. There is an equivalence
\begin{equation}\label{equation:last}
  \LL^{\rm Sat}:
  \cD_c(I\backslash LX)^\Sat
  \simeq
  \Perf  (
    \sh^{1/2}  (
      \tilde{\check{\mathfrak g}}^*(2)
      \times_{\check{\mathfrak g}^*(2)}^{\mathbf R}
      \check M
      )/\check G
    )
\end{equation}
satisfying the following properties.
\begin{enumerate}
\item (Pointing condition.)
For the natural projection
$r:I\backslash LX\rightarrow L^+G\backslash LX$
there is an isomorphism
\begin{equation}\label{eq:iwahori-pointing}
  \LL^{\rm Sat}(r^!\IC_0)
  \simeq
  \cO_{\sh^{1/2}  (
    \tilde{\check{\mathfrak g}}^*(2)
    \times_{\check{\mathfrak g}^*(2)}^{\mathbf R}
    \check M
    )}.
\end{equation}
The structure object on the right carries its natural
$\check G$-equivariant structure.

\item (Module structure.)
The equivalence $\LL^{\rm Sat}$ is compatible with
the actions coming from the monoidal equivalence
$$
\begin{aligned}
  \mathrm{End}_{\cH_{c,K}}
     (\cD_c(I\backslash\Gr)  )
  \simeq
  \mathrm{End}_{\Perf(\check{\mathfrak g}^*[2]/\check G)}
     (\Perf(\tilde{\check{\mathfrak g}}^*[2]/\check G)  ).
\end{aligned}
$$
\end{enumerate}
\end{thm}

\begin{proof}\label{proof:main_iwahori}
By \autoref{thm:F_cstr_Sat} and
\autoref{lem:satake-generators} convolution gives
an equivalence
$$
  F_c^\Sat:
  \cD_c(I\backslash\Gr)
  \otimes_{\cH_{c,K}}
  \cD_c(K\backslash LX)^\Sat
  \xrightarrow{\sim}
  \cD_c(I\backslash LX)^\Sat.
$$
To pass to the spectral side, recall the derived geometric Satake equivalence, Arkhipov-Bezrukavnikov-Ginzburg equivalence and the assumed unramified local equivalence \autoref{Conj:intro}:
\begin{equation}\label{eq:spectral-inlet-equivalences}
\begin{aligned}
  \Sat_K:\cH_{c,K}
  &\simeq
  \Perf(\check{\mathfrak g}^*[2]/\check G),\\
  \Sat_I:\cD_c(I\backslash\Gr)
  &\simeq
  \Perf(\tilde{\check{\mathfrak g}}^*[2]/\check G),\\
  \cD_c(K\backslash LX)^\Sat
  &\simeq
  \Perf(\sh^{1/2}(\check M)/\check G).
\end{aligned}
\end{equation}
By \cite[Theorem 6.12.3]{campbell_raskin_bd_grassmannian}
and \cite[Lemma 3.1.6]{beraldo_chen_automorphic_gluing} $\Sat_I$ is an equivalence of modules categories with repsect to the action of $\Sat_K$\footnote{The cited equivalences are established for presentable categories in the factorization setting. They restrict to the forms in \eqref{eq:spectral-inlet-equivalences} by \cite[\S\S6.9--6.11]{campbell_raskin_bd_grassmannian} and
\cite[Proposition 5.2.2 and \S5.4]
{gaitsgory-semiinfinite-intersection-coh}.}. On the level of objects, for $\cF\in\cD_c(I\backslash\Gr)$  and $\cG\in\cH_{c,K}$
$$
  \Sat_I(\cF\ast\cG)
  \simeq
  \Sat_I(\cF)
  \otimes^{\mathbf L}_{
    \check\pi^*\Sym(\check{\mathfrak g}[-2])}
  \check\pi^*\Sat_K(\cG).
$$
The unramified equivalence is compatible with the spherical
action by assumption.
Thus \eqref{eq:spectral-inlet-equivalences} together with
\autoref{lem:spectral_relative_tensor} establish the sought for equivalence
\eqref{equation:last}.

We now prove the pointing condition. By \cite[Lemma 3.1.7]{beraldo_chen_automorphic_gluing}
and \cite[Corollary 8.10.1.1 and \S9.4]
{campbell_raskin_bd_grassmannian} we have
$$
  \Sat_I(\cP_0)
  \simeq
  \cO_{\tilde{\check{\mathfrak g}}^*[2]}.
$$
Moreover \eqref{eq:basepoint-convolution} shows that
$$
  \cP_0\ast\IC_0
  \simeq r^!\IC_0.
$$
The pointing condition in \autoref{Conj:intro}
identifies $\IC_0$ with
$\cO_{\sh^{1/2}(\check M)}$. Then \eqref{eq:iwahori-pointing} follows from \autoref{lem:spectral_relative_tensor}.

Finally the compatibility of $\Sat_K$ and $\Sat_I$
induces the monoidal equivalence
\eqref{equation:algebra}.
The two endomorphism categories act on the corresponding
relative tensor products through the first factors.
These actions are identified by functoriality of the
relative tensor product. Transporting them along
$F_c^\Sat$ and the equivalence of
\autoref{lem:spectral_relative_tensor} proves
statement $(2)$.
\end{proof}

\section{Iwahori generation by basic object, a case study of $(G,X)=(\GL_n,\mathbb{A}^n)$}\label{section:affine_categorical_generation}
Let $G=\GL_n$ act on $X=\AA^n$ by left multiplication. 
In this section, we will show that
$\cD_c(I\bs LX)= \cD_c(I\bs LX)^{\Sat}$ and give the analogous statement 
at the functions level at the end.
We keep the field $k$ of \autoref{subsec:notations} and write $F=k(\!(t)\!)$ and $\cO=k[\![t]\!]$, so $LX(k)=F^n$ and $L^+X(k)=\cO^n$. Our main theorem in the first part is 
\begin{thm}
\label{prop:affcase_categorical_generation}
For $(G,X)=(\GL_n,\mathbb{A}^n)$, there is an equivalence 
$\cD_c(I\bs LX)\simeq \cD_c(I\bs LX)^{\Sat}$. 
\end{thm}

\subsection{Spherical and Iwahori orbits}

Let $T$ be the torus of diagonal matrices in $\GL_n$. The extended affine Weyl group can be written explicitly as
$W_{\mathrm{ext}}=\mathbb X_*(T)\rtimes W=\ZZ^n\rtimes S_n$.
Its elements have the form $t^\lambda w$ with
$\lambda\in\mathbb X_*(T)$ and $w\in W$. Recall that $\Omega\cong\ZZ$ in $W_{\mathrm{ext}}=W_{\mathrm{aff}}\rtimes\Omega$, where $\Omega$ is the subgroup of length $0$ element.
Denote $\tau$ to be the generator of
$\Omega\cong\ZZ$ which is represented by $\tau\in \GL_n(F)$ satisfying
$$
\tau e_1=te_n,\quad \tau e_i=e_{i-1}\ (2\le i\le n),\quad \tau^n=t1_n,
$$
where $e_1,\ldots,e_n$ is the standard basis. 

\begin{definition}[Iwahori orbit closures]
\label{def:iwahori_orbit_An}
For $a\in\ZZ$ we let $L_a:=\tau^aL^+X$.
When $a=nm+n-j$ for $m\in \mathbb{Z}$ and $1\le j\le n$, we also denote 
$\Lambda_{m,j}:= L_a$. 
\end{definition}

\begin{lemma}\label{lem:affcase_lattice_chain}
For $a=nm+n-j$, $m\in\ZZ$, and $1\le j\le n$,
$$
\Lambda_{m,j}(k)= 
L_a(k)=t^m (\cO^{j}\oplus(t\cO)^{n-j}  )= 
t^m \text{diag }(1_j,t\cdot 1_{n-j})\cO^n,
$$
and $\dim_k(L_a/L_{a+1})=1$. In particular, 
$t^mL^+X=\Lambda_{m,n}\supset\cdots\supset\Lambda_{m,0}=t^{m+1}L^+X$.
\end{lemma}

\begin{proof}
The formulas for $\tau$ give $L_r(k)=\cO^{n-r}\oplus(t\cO)^r$ for $0\le r\le n$ and $L_{a+n}=tL_a$.
\end{proof}

These $\Lambda_{m,j}$'s 
are in fact the Iwahori orbit closures $LX$, as proven
in the below proposition.

\begin{proposition}
\label{lem:affcase_orbits}
For $v\in F^n\smallsetminus\{0\}$, let $
\ord(v):=\min_i\val(v_i)$ and $\last(v):=\max\{i:\val(v_i)=\ord(v)\}.$
\begin{enumerate}
\item The $L^+G$-orbits of $LX$ 
are $\{0\}$ and $LX_m$, $m\in\ZZ$, where
$$
LX_m(k):=\{v\in F^n:\ord(v)=m\}=t^m(\cO^n\smallsetminus t\cO^n).
$$
\item 
The $I$-orbits of $LX$ are $\{0\}$ and $\cO_{m,j}$, $m\in\ZZ$, $1\le j\le n$, where
$$
\cO_{m,j}(k):=\{v\in F^n:\ord(v)=m,\ \last(v)=j\}=\Lambda_{m,j}\smallsetminus\Lambda_{m,j-1} 
= I\cdot t^me_j.
$$
\item We have 
$\overline{\cO}_{m,j}=\Lambda_{m,j}=L_a$ for $a=nm+n-j$.
\end{enumerate}
\end{proposition}
\begin{proof}
(1) For $h\in \GL_n(\cO)$, we have $\ord(hv)\ge \ord(v)$, 
and reapplying for $h^{-1}$ gives the equality. Since 
$LX_0(k)=\GL_n(\cO)e_1= 
\cO^n \setminus t\cO^n= \{ v:\ord(v)=0\}$, the assertion is clear. 

(2) Let $K_1=1+t\text{Mat}_n(\cO)\subset I(k)$.  
Observe that we have 
the following isomorphism
\begin{align*}
LX_m(k)/K_1 &\xrightarrow[]{\sim} k^n\setminus \{0\}\\
t^mu & \mapsto u\pmod t.
\end{align*}
Together with $I(k)/K_1\cong B(k)$, we deduce $LX_m(k)/I(k) 
\xrightarrow{\sim} (k^n\setminus \{0\})/B(k)$.
Note that $B$-orbits on $\mathbb{A}^n\setminus \{0\}$ are classified by 
position of the last non-zero coordinate, giving the desired 
description for $\cO_{m,j}(k)$.

(3) The orbit $\cO_{m,j}$ 
is the complement of the linear hyperplane 
$\Lambda_{m,j-1}$, hence is dense in $\Lambda_{m,j}$.
\end{proof}

\begin{example}
For $n=2$, $\cO_{m,1}$ is the vertical ray $\val(v_1)=m<\val(v_2)$, while $\cO_{m,2}$ is the horizontal ray $\val(v_2)=m\le\val(v_1)$.
The zero vector corresponds to $(\infty,\infty)$, and the shaded region is $L^+X(k)=\cO^2$.
\begin{center}
 \begin{tikzpicture}[scale=0.825,
   ray1/.style={very thick,orange!85!black},
   ray2/.style={very thick,blue!70!black}]
   \fill[gray!12] (0,0) rectangle (3.3,3.3);
   \draw[dashed,gray] (-1.7,-1.7) -- (3.3,3.3);
   \foreach \m in {-1,0,1}{
     \draw[ray1] (\m,\m+0.15) -- (\m,3);
     \draw[ray2] (\m,\m) -- (3,\m);
     \fill[blue!70!black] (\m,\m) circle (2.5pt);}
   \node at (3,3) {$\star$};
   \foreach \x/\l in {-1/-1,0/0,1/1,3/\infty}{
     \node[font=\small] at (\x,-2.2) {$\l$};
     \node[font=\small] at (-2.3,\x) {$\l$};}
   \node[font=\small] at (1,-3) {$\val(v_1)$};
   \node[font=\small,rotate=90] at (-3.1,1) {$\val(v_2)$};
   \node[font=\footnotesize] at (1.9,0.5) {$\cO^2$};
 \end{tikzpicture}
\end{center}
\end{example}

\subsection{Convolution computations}

We first record the lattice and convolution properties of $\tau$. We use $\cD=\cD^!$ as in \autoref{Section:sheaf theory}.

\begin{lemma}\label{lem:affcase_translation}
For every $a\in\ZZ$ and $\cF\in\cD_c(I\backslash LX)$, there is an isomorphism
$$
\IC_{\tau^a}\ast\cF\simeq(\tau^a)_*\cF.
$$
\end{lemma}
\begin{proof}Since $\tau^a$ has length zero, the associated affine Schubert variety $\overline{I \tau^a I/I}$ is  a single point, and $\IC_{\tau^a} = \delta_{\tau^aI}$. 
Consider the  correspondence
$$
\begin{tikzcd}[column sep=large,row sep=large]
&LG\times^I LX\ar[dl,"\mathrm{pr}"']\ar[dr,"m"]&\\
\Fl&&LX,
\end{tikzcd}
\quad
$$
\noi where $\rm p$ is the natural projection and $m$ is the convolution moprhism.
The fibre of $\tau^aI$ along $\rm pr$ is $\tau^aI\times^I LX$ with
$$
j_a:LX\xrightarrow{\sim}\mathrm{pr}^{-1}(\tau^aI),\quad
v\longmapsto(\tau^a,v),\quad m\circ j_a(v)=\tau^av.
$$
Thus
$$
\IC_{\tau^a} \ast \cF \simeq m_* (\delta_{\tau^a I} \widetilde\boxtimes \cF) \simeq m_* \big( j_{a,*} j_a^! (\delta_{\tau^a I} \widetilde\boxtimes \cF) \big) \simeq (m \circ j_a)_* \cF = (\tau^a)_* \cF.
$$ 
Here, the second equivalence follows since $\delta_{\tau^a I}$ is supported only at the point $\tau^a$  and the counit is for the adjunction $j_{a,*}\dashv j_a^!$ is an isomorphism.
\end{proof}

\subsection{Placidness and generations} 
\label{subsec:level_simple_objects}

In this section we will construct concretely a placid presentation for 
$LX= L\mathbb{A}^n$ as placid $I$-scheme c.f.\ \autoref{prop:placid_example_An}. From this we 
then construct a list of generators for $\mathcal{C}$,
called $S_a$, for $a\in \mathbb{Z}$ (see \autoref{prop:affcase_simples}), which come from 
Iwahori orbit closures $L_a$.

\begin{definition}[Levels]
\label{def:levels}
For $N,M\ge0$ and $N+M\ge1$, let
$$
\begin{gathered}
Y_N:=t^{-N}L^+X,\quad Y_{N,M}:=Y_N/t^ML^+X,\\
I_{N,M}:=\ker(I\to L^{N+M}G),\quad H_{N,M}:=I/I_{N,M}.
\end{gathered}
$$
In particular, $Y_{N,M}(k)=t^{-N}\cO^n/t^M\cO^n$ and
$I_{N,M}(k)=1+t^{N+M}\Mat_n(\cO)$.
The action of $I$ on $Y_{N,M}\cong\AA^{n(N+M)}$
factors through $H_{N,M}$.
\end{definition}

Write $i_N:Y_N\hookrightarrow LX$ and
$j_{N,M}:Y_{N,M}\hookrightarrow Y_{N+1,M}$ for the
natural inclusions. For $M'\ge M$, let
$p_{M',M}:Y_{N,M'}\to Y_{N,M}$ and
$\pi_{N,M}:Y_N\to Y_{N,M}$ be the truncation maps.
These maps are $I$-equivariant. Denote by
$Q_{N,M}:I\bs Y_N\to H_{N,M}\bs Y_{N,M}$
the induced projection, and use the same notation
$i_N$ for the inclusion $I\bs Y_N\hookrightarrow I\bs LX$.

\begin{proposition}[Placidness of $LX$]
\label{prop:placid_example_An}
\begin{enumerate}
\item Each $Y_N\simeq\varprojlim_M Y_{N,M}$ is a
placid $I$-scheme, and
$LX\simeq\varinjlim_N Y_N$ is a placid
$I$-presentation. The transition maps in $N$ are
finitely presented closed embeddings, and
$p_{M',M}$ is an affine-space bundle of relative
dimension $n(M'-M)$.

\item The functors
$$
\Phi_{N,M}:=i_{N,*}Q_{N,M}^![nM]:
\cD_c(H_{N,M}\bs Y_{N,M})
\longrightarrow\cD_c(I\bs LX)
$$
are fully faithful and induce an equivalence
$$
\varinjlim_{N,M}\cD_c(H_{N,M}\bs Y_{N,M})
\xrightarrow{\sim}\cD_c(I\bs LX),
$$
where the transition functors in the $N$ and $M$
directions are $(j_{N,M})_*$ and
$p_{M',M}^![-n(M'-M)]$, respectively.
Here the transition functors are understood through
the canonical equivalences
$\cD_c(H_{N,M}\bs Y_{N,M})\simeq
\cD_c(I\bs Y_{N,M})$
induced by $!$-pullback.
\end{enumerate}
\end{proposition}
\begin{proof}
By construction, the morphism $Y_N\hookrightarrow Y_{N+1}$ is a finitely presented closed embedding, and for $M'\ge M$, 
$p_{M',M}: Y_{N,M'}\twoheadrightarrow Y_{N,M}$
is an affine fibration with fiber $\mathbb{A}^{n(M'-M)}$.
It follows from definition that 
$LX=\colim_N \lim_M Y_{N,M}$
is a placid $I$-presentation of $LX$. 

The truncation maps and inclusions fit into the following
commutative diagram, whose two squares are Cartesian:
$$
\begin{tikzcd}[column sep=large,row sep=large]
Y_N
  \ar[r,hook]
  \ar[rr,hook,bend left=25,"i_N"]
  \ar[d,two heads,"\pi_{N,M'}"']
&
Y_{N+1}
  \ar[r,hook,"i_{N+1}"']
  \ar[d,two heads,"\pi_{N+1,M'}"]
&
LX
\\
Y_{N,M'}
  \ar[r,hook,"j_{N,M'}"]
  \ar[d,two heads,"p_{M',M}"']
&
Y_{N+1,M'}
  \ar[d,two heads,"p_{M',M}"]
&
\\
Y_{N,M}
  \ar[r,hook,"j_{N,M}"]
&
Y_{N+1,M}.
&
\end{tikzcd}
$$

We check $\Phi_{N,M}$ is well-defined. Indeed, we know 
$\Phi_{N,M'}p_{M',M}^![-n(M'-M)]\simeq\Phi_{N,M}$ from 
commutativity of the above diagram. In addition,  
$\Phi_{N+1,M}\circ (j_{N,M})_*\simeq \Phi_{N,M}$ by proper 
base change, where $\iota: Y_N\hookrightarrow Y_{N+1}$. 

Finally, we show that $\Phi_{N,M}$ induces an equivalence. 
By \autoref{lem:coherent_finite_stage}, we have 
$\cD_c(I\backslash LX) \simeq\colim_N \cD_c(I\backslash Y_N)$.
 Placidness of $Y_N$ and \cite[Proposition 3.6.1 and \S4.3]{raskin_dmodules_infinite_var} give an isomorphism $\cD(Y_N)\simeq\colim_M \cD(Y_{N,M})$. 
Since coherence is equivalent to compactness for placid scheme, 
we have $\cD_c(Y_N)\simeq\colim_M \cD_c(Y_{N,M})$. 
Coherence under quotient stack \autoref{defn:coherent_D_mod}
then implies $\cD_c(I\backslash Y_N)\simeq\colim_M \cD_c(I \backslash Y_{N,M})$. Because the action of $I$ on 
$Y_{N,M}$ factors through $H_{N,M}$ with unipotent
kernel $I_{N,M}$ (\autoref{def:levels}), there is an equivalence
$\cD(I\backslash Y_{N,M})\simeq\cD(H_{N,M}\backslash Y_{N,M})$
and hence restricts to coherent objects. 
The above argument complets the proof.
\end{proof}

\begin{definition}[Generators]
For $a\in \mathbb{Z}$,  define 
$S_a:= i_{a*}\omega_{L_a}[a] \in \cD_c
(I\backslash LX)$, where $i_a:L_a \hookrightarrow LX$ is the Iwahori orbit closures
defined in \autoref{def:iwahori_orbit_An}. 
Note that $S_0=r^!\IC_0=\omega_{L^+X}$, is the basic object. 
\end{definition}

We will show that the above elements $S_a$'s thickly generate $\cD_c(I\backslash LX)$.

\begin{proposition}[Generations]\label{prop:affcase_simples}
Under the  placid $I$-presentation in \autoref{prop:placid_example_An} of $LX$, we have
\begin{enumerate}
    \item The $H_{N,M}$-orbits for $Y_{N,M}$ are $O_a:=V_a\setminus 
    V_{a+1}$, where $V_a:=L_a/t^ML^+X$, $-nN\le a < nM$, 
    and $O_{nM}=\{0\}$. As a result,
    $\cD_c(H_{N,M}\backslash Y_{N,M})$ is thickly 
    generated
    by $\IC_{a,\triv}\simeq\underline{k}_{V_a}[nM-a]$ 
    for $-nN\le a \le nM$.
    \item There is an isomorphism 
    $\Phi_{N,M} (\underline{k}_{V_a}[nM-a])\simeq S_a$.
    \item 
    $\cD_c(I\backslash LX)$ is thickly 
    generated
    by $S_a$'s for $a\in \mathbb{Z}$.
\end{enumerate}
\end{proposition}

\begin{proof} (1) 
By \autoref{def:levels}, $H_{N,M}$-orbits for $Y_{N,M}$
correspond to $I$-orbits in $t^{-N}L^+X$ modulo $t^ML^+X$,
which by \autoref{lem:affcase_orbits} are 
$O_a= (L_a\setminus L_{a+1})/t^ML^+X$
for $-Nn \le a \le Mn$ (note that $O_{Mn}=\{0\}$, with stabilizer $H_{N,M}$). 

For such $-Nn\le a=nm+n-j<nM$, stabilizer of a point in
$O_a$ is 
$$
H_a:=\operatorname{Stab}_{H_{N,M}}(t^me_j)
=\{h\in H_{N,M}:h_{ij}\equiv\delta_{ij}\pmod{t^{M-m}},\ 1\le i\le n\}
\cong\GG_m^{n-1}\times\AA^{b_a}
$$
of algebraic vareities, 
for some $b_a\ge 0$, which is connected. Therefore 
by \autoref{cons:generator_sheaves_on_orbit} and 
similar strategy to \autoref{cor:spherical_IC_generation}, 
$\cD_c(H_{N,M}\backslash Y_{N,M})$ is thickly 
generated by $\IC_{a,\triv}=\IC(\overline{O_a},\underline{k}_{O_a})$
for $-Nn\le a\le nM$. Here we have $
\IC_{a,\triv}=\underline{k}_{V_a}[\dim V_a] =
\underline{k}_{V_a}[nM-a]$ since 
$\overline{O_a}=V_a$ is smooth, and 
$\dim V_a=\dim L_a/t^mL^+X=nM-a$ is 
due to $t^ML^+X= L_{nM}$ plus knowing 
$\dim_k (L_b/L_{b+1})=1$ from \autoref{lem:affcase_lattice_chain}.

(2) We will show 
$\Phi_{N,M} (j_{a*}\underline{k}_{V_a}[nM-a])=S_a$.
Write $\pi_a:L_a\longrightarrow V_a$ for the restriction of
$\pi_{N,M}$. We have the Cartesian square
\begin{center}
\begin{tikzcd}
L_a \arrow[r,"i_a'"]\arrow[d,"\pi_a",swap] & Y_{N} \arrow[d,"\pi_{N,M}"] \\
V_a \arrow[r,"j_a"] & Y_{N,M}
\end{tikzcd}
\end{center}
Passing to quotient stacks, with $I$ on the upper row and
$H_{N,M}$ on the lower row, gives a Cartesian square whose
right vertical map is $Q_{N,M}$. We use the same notation
$\pi_a$, $i_a'$ and $j_a$ for the other induced morphisms.
By base change, we have the following equivalences 
\begin{align*}
\Phi_{N,M}(j_{a*}\underline{k}_{V_a})&\simeq
i_{N*}Q_{N,M}^!(j_{a*}\underline{k}_{V_a}[nM])\\&
\simeq i_{N*}i_{a*}'\pi_a^! \underline{k}_{V_a}[nM]), \\
&\simeq  i_{a*}\pi_a^! \omega_{V_a}[nM-2\dim V_a], \\
& \simeq i_{a*}\omega_{L_a}[2a-nM]= S_a[a-nM].
\end{align*}

(3) By \autoref{prop:placid_example_An}, 
$\cD_c(I\backslash LX)$ is a filtered colimits of 
$\cD_c(H_{N,M}\backslash Y_{N,M})$. Hence every 
$\mathcal{F}\in \cD_c(I\backslash LX)$ is isomorphic to
$\Phi_{N,M}(\mathcal{G})$ for some $\mathcal{G}\in \cD_c(
H_{N,M}\backslash Y_{N,M})$. 
By (1), $\mathcal{G}$ is thickly generated by 
$\underline{k}_{V_a}$. The statement thus follows from \autoref{rem:thick_properties} and (2). 
\end{proof}

\subsection{Proof of Iwahori generation by basic object}
We are now ready to prove \autoref{prop:affcase_categorical_generation}. 
By \autoref{prop:affcase_simples}, it suffices to show $S_a\in\cD_c(I\backslash LX)^{\Sat}$.
As $L_a=\tau^a L^+X$ so by \autoref{lem:affcase_translation},
we have 
$$
  \IC_{\tau^a}\ast r^!\IC_0\simeq(\tau^a)_*i_{0*}\omega_{L_0}\simeq
i_{a*}\omega_{L_a}=S_a[-a]. 
$$
As thickly generated subcategory is closed under shift, 
we obtain the desired claim. 

\subsection{The classical function analogue}

We now prove the corresponding generation statement
for compactly supported locally constant functions.
Let $F_0=\FF_q(\!(t)\!)$, $\cO_0=\FF_q[\![t]\!]$,
and $K_0=\GL_n(\cO_0)$. Let $I_0\subset K_0$ be the
standard Iwahori subgroup.

For $m\in\ZZ$ and $0\le j\le n$, let
$$
g_{m,j}:=t^m\operatorname{diag}(1_j,t\cdot1_{n-j}),
\quad
\Lambda_{m,j,0}:=g_{m,j}\cO_0^n.
$$
As in \autoref{lem:affcase_orbits}, the nonzero
$K_0$-orbits and $I_0$-orbits in $F_0^n$ are,
respectively,
$$
LX_{m,0}:=t^m(\cO_0^n\smallsetminus t\cO_0^n),
\quad
\cO_{m,j,0}:=
\Lambda_{m,j,0}\smallsetminus\Lambda_{m,j-1,0},
$$
where $m\in\ZZ$ and $1\le j\le n$.
In both cases, the remaining orbit is $\{0\}$.

For $J\in\{K_0,I_0\}$, write
$$
C_J:=C_c^\infty(F_0^n)^J,
\quad
H_J:=C_c^\infty(J\bs\GL_n(F_0)/J),
$$
and let $\varphi_0:=\mathbf1_{\cO_0^n}$.
For each $J$, normalize Haar measure by $\vol(J)=1$.
The convolution action of $H_J$ on $C_J$ is given by
$$
(f\ast_J\varphi)(x)
=
\int_{\GL_n(F_0)}f(h)\varphi(h^{-1}x)\,dh.
$$

\begin{lemma}
\label{lem:affcase_translates}
For $m\in\ZZ$ and $0\le j\le n$, the functions
$\mathbf1_{t^mK_0}$ and $\mathbf1_{g_{m,j}K_0}$
belong to $H_{K_0}$ and $H_{I_0}$, respectively.
Moreover,
$$
\begin{aligned}
\mathbf1_{t^m\cO_0^n}
&=\mathbf1_{t^mK_0}\ast_{K_0}\varphi_0,\\
\mathbf1_{\Lambda_{m,j,0}}
&=[K_0:I_0]^{-1}
  \mathbf1_{g_{m,j}K_0}\ast_{I_0}\varphi_0.
\end{aligned}
$$
\end{lemma}

\begin{proof}
Since $t^m$ is central, $\mathbf1_{t^mK_0}$ is
bi-$K_0$-invariant. The lattice $\Lambda_{m,j,0}$
is $I_0$-stable, so
$I_0g_{m,j}K_0=g_{m,j}K_0$.
Thus $\mathbf1_{g_{m,j}K_0}$ is bi-$I_0$-invariant.

For $h\in g_{m,j}K_0$, we have
$h\cO_0^n=\Lambda_{m,j,0}$, and hence
$$
\varphi_0(h^{-1}x)=\mathbf1_{\Lambda_{m,j,0}}(x).
$$
With $\vol(I_0)=1$, the coset $g_{m,j}K_0$ has
volume $[K_0:I_0]$, which proves the second formula.
The first follows from the same calculation with
$\vol(K_0)=1$.
\end{proof}

\begin{proposition}
\label{prop:affcase_functions}
For $J\in\{K_0,I_0\}$, we have
$$
H_J\ast_J\varphi_0=C_J.
$$
\end{proposition}

\begin{proof}
Taking successive differences in
\autoref{lem:affcase_translates} gives
$$
\begin{aligned}
\mathbf1_{LX_{m,0}}
&=
(\mathbf1_{t^mK_0}-\mathbf1_{t^{m+1}K_0})
\ast_{K_0}\varphi_0,\\
\mathbf1_{\cO_{m,j,0}}
&=
[K_0:I_0]^{-1}
(\mathbf1_{g_{m,j}K_0}-\mathbf1_{g_{m,j-1}K_0})
\ast_{I_0}\varphi_0.
\end{aligned}
$$

Let $\varphi\in C_{I_0}$. Since $\varphi$ is compactly
supported and locally constant, we may choose $M\gg0$
such that it vanishes outside $t^{-M}\cO_0^n$
and is constant on $t^M\cO_0^n$. Then
$$
\varphi
=
\varphi(0)\mathbf1_{t^M\cO_0^n}
+
\sum_{m=-M}^{M-1}\sum_{j=1}^n
\varphi(\cO_{m,j,0})\mathbf1_{\cO_{m,j,0}},
$$
where $\varphi(\cO_{m,j,0})$ denotes the constant value
of $\varphi$ on that orbit.
The first summand belongs to
$H_{I_0}\ast_{I_0}\varphi_0$ by
\autoref{lem:affcase_translates}, and the remaining
summands do so by the displayed formulas.
This proves the assertion for $I_0$.
The same argument using the orbits $LX_{m,0}$
proves the assertion for $K_0$.
\end{proof}
\printbibliography

@article{BFGT21,
  author  = {Braverman, Alexander and Finkelberg, Michael
             and Ginzburg, Victor and Travkin, Roman},
  title   = {Mirabolic {Satake} equivalence and supergroups},
  journal = {Compositio Mathematica},
  volume  = {157},
  number  = {8},
  year    = {2021},
  pages   = {1724--1765},
  doi     = {10.1112/S0010437X21007387},
  eprint  = {1909.11492},
  archivePrefix = {arXiv},
  primaryClass  = {math.RT}
}

@article{BFT22,
  author  = {Braverman, Alexander and Finkelberg, Michael
             and Travkin, Roman},
  title   = {Orthosymplectic {Satake} equivalence},
  journal = {Communications in Number Theory and Physics},
  volume  = {16},
  number  = {4},
  year    = {2022},
  pages   = {695--732},
  doi     = {10.4310/CNTP.2022.v16.n4.a2},
  eprint  = {1912.01930},
  archivePrefix = {arXiv},
  primaryClass  = {math.RT}
}

@article{TY23,
  author  = {Travkin, Roman and Yang, Ruotao},
  title   = {Untwisted {Gaiotto} equivalence},
  journal = {Advances in Mathematics},
  volume  = {435},
  year    = {2023},
  pages   = {109359},
  doi     = {10.1016/j.aim.2023.109359},
  eprint  = {2201.10462},
  archivePrefix = {arXiv},
  primaryClass  = {math.RT}
}

@misc{campbell_raskin_bd_grassmannian,
  author        = {Campbell, Justin and Raskin, Sam},
  title         = {{Langlands} duality on the
                   {Beilinson--Drinfeld Grassmannian}},
  year          = {2023},
  eprint        = {2310.19734},
  archivePrefix = {arXiv},
  primaryClass  = {math.RT},
  url           = {https://arxiv.org/abs/2310.19734}
}

@misc{beraldo_chen_automorphic_gluing,
  author        = {Beraldo, Dario and Chen, Lin},
  title         = {Automorphic gluing},
  year          = {2022},
  eprint        = {2204.09141},
  archivePrefix = {arXiv},
  primaryClass  = {math.RT},
  url           = {https://arxiv.org/abs/2204.09141}
}

@misc{Devalapurkar_triple_product,
  author        = {Devalapurkar, Sanath K.},
  title         = {Derived geometric {Satake} for
                   $\mathrm{PGL}_2^{\times 3}/
                   \mathrm{PGL}_2^{\mathrm{diag}}$},
  year          = {2024},
  eprint        = {2404.09853},
  archivePrefix = {arXiv},
  primaryClass  = {math.RT},
  url           = {https://arxiv.org/abs/2404.09853}
}

@misc{Chen_relative_Satake,
  author        = {Chen, Tsao-Hsien},
  title         = {A relative {Langlands} dual realization of
                   {$T^*(G/K)$} and derived {Satake}},
  year          = {2026},
  eprint        = {2601.18022},
  archivePrefix = {arXiv},
  primaryClass  = {math.RT},
  url           = {https://arxiv.org/abs/2601.18022}
}

@misc{Chen_OBrien_Lorentzian,
  author        = {Chen, Tsao-Hsien and O'Brien, John},
  title         = {Lorentzian and Octonionic {Satake} equivalence},
  year          = {2024},
  eprint        = {2409.03969},
  archivePrefix = {arXiv},
  primaryClass  = {math.RT},
  url           = {https://arxiv.org/abs/2409.03969}
}

@article{TravkinYang_untwisted_Gaiotto,
  author  = {Travkin, Roman and Yang, Ruotao},
  title   = {Untwisted {Gaiotto} equivalence},
  journal = {Advances in Mathematics},
  volume  = {435},
  year    = {2023},
  pages   = {109359},
  eprint  = {2201.10462},
  archivePrefix = {arXiv},
  url     = {https://arxiv.org/abs/2201.10462}
}

@misc{Fu_derived_Weil,
  author        = {Fu, Haoshuo},
  title         = {Derived {Weil} Representation and
                   Relative {Langlands} Duality},
  year          = {2026},
  eprint        = {2603.26058},
  archivePrefix = {arXiv},
  primaryClass  = {math.RT},
  url           = {https://arxiv.org/abs/2603.26058}
}

@article{chen2022quaternionic,
  title={Quaternionic Satake equivalence},
  author={Chen, Tsao-Hsien and Macerato, Mark and Nadler, David and O'Brien, John},
  journal={arXiv preprint arXiv:2207.04078},
  year={2022}
}

@article{campbell2024proof,
  title={Proof of the geometric Langlands conjecture III: compatibility with parabolic induction},
  author={Campbell, Justin and Chen, Lin and Faergeman, Joakim and Gaitsgory, Dennis and Lin, Kevin and Raskin, Sam and Rozenblyum, Nick},
  journal={arXiv preprint arXiv:2409.07051},
  year={2024}
}

@article{BZSV,
    author = { David Ben-Zvi and Yiannis Sakellaridis and Akshay Venkatesh },
    title ={Relative Langlands duality} ,
    year = {2024}
}

@article{BravermanFinkelbergGinzburgTravkin_2021,
   title={Mirabolic Satake equivalence and supergroups},
   volume={157},
   ISSN={1570-5846},
   url={http://dx.doi.org/10.1112/S0010437X21007387},
   DOI={10.1112/s0010437x21007387},
   number={8},
   journal={Compositio Mathematica},
   publisher={Wiley},
   author={Braverman, Alexander and Finkelberg, Michael and Ginzburg, Victor and Travkin, Roman},
   year={2021},
   month=jul, pages={1724–1765} }

@misc{cohn2016differentialgradedcategoriesklinear,
      title={Differential Graded Categories are k-linear Stable Infinity Categories}, 
      author={Lee Cohn},
      year={2016},
      eprint={1308.2587},
      archivePrefix={arXiv},
      primaryClass={math.AT},
      url={https://arxiv.org/abs/1308.2587}, 
}

@misc{chenyi_2025singularitiesorbitclosuresloop,
      title={Singularities of orbit closures in loop spaces of symmetric varieties}, 
      author={Tsao-Hsien Chen and Lingfei Yi},
      year={2025},
      eprint={2310.20006},
      archivePrefix={arXiv},
      primaryClass={math.RT},
      url={https://arxiv.org/abs/2310.20006}, 
}

@misc{devalapurkar2024kutheoreticspectraldecompositionsspheres,
      title={ku-theoretic spectral decompositions for spheres and projective spaces}, 
      author={Sanath K. Devalapurkar},
      year={2024},
      eprint={2402.03995},
      archivePrefix={arXiv},
      primaryClass={math.AT},
      url={https://arxiv.org/abs/2402.03995}, 
}

@incollection {drinfeld_infinite_vectorbundle,
    AUTHOR = {Drinfeld, Vladimir},
     TITLE = {Infinite-dimensional vector bundles in algebraic geometry: an
              introduction},
 BOOKTITLE = {The unity of mathematics},
    SERIES = {Progr. Math.},
    VOLUME = {244},
     PAGES = {263--304},
 PUBLISHER = {Birkh\"auser Boston, Boston, MA},
      YEAR = {2006},
      ISBN = {978-0-8176-4076-7},
   MRCLASS = {14F05 (14C35 14F42)},
  MRNUMBER = {2181808},
       DOI = {10.1007/0-8176-4467-9\_7},
       URL = {https://doi.org/10.1007/0-8176-4467-9_7},
}

@article{GR17vol1,
    author = {Dennis Gaitsgory and Nick Rozenblyum},
    title = {A study in Derived Algebraic Geometry, volume 1},
    year = {2017}  
}

@article {gaitsgory-semiinfinite-intersection-coh,
    AUTHOR = {Gaitsgory, D.},
     TITLE = {The semi-infinite intersection cohomology sheaf},
   JOURNAL = {Adv. Math.},
  FJOURNAL = {Advances in Mathematics},
    VOLUME = {327},
      YEAR = {2018},
     PAGES = {789--868},
      ISSN = {0001-8708,1090-2082},
   MRCLASS = {14F05 (14D24 14M15 20J06)},
  MRNUMBER = {3762003},
MRREVIEWER = {Anthony\ Henderson},
       DOI = {10.1016/j.aim.2017.08.007},
       URL = {https://doi.org/10.1016/j.aim.2017.08.007},
}

@book {HTT,
    AUTHOR = {Lurie, Jacob},
     TITLE = {Higher topos theory},
    SERIES = {Annals of Mathematics Studies},
    VOLUME = {170},
 PUBLISHER = {Princeton University Press, Princeton, NJ},
      YEAR = {2009},
}

@article{HA,
    author = {Lurie, Jacob},
    title ={Higher Algebra} ,
    year = {2009}
}

@misc{raskin_dmodules_infinite_var,
title={D-modules on infinite dimentional varieties},
author={Sam Raskin},
url={https://gauss.math.yale.edu/~sr2532/dmod.pdf},
}

@article {SV17,
    AUTHOR = {Sakellaridis, Yiannis and Venkatesh, Akshay},
     TITLE = {Periods and harmonic analysis on spherical varieties},
   JOURNAL = {Ast\'erisque},
  FJOURNAL = {Ast\'erisque},
    NUMBER = {396},
      YEAR = {2017},
     PAGES = {viii+360},
      ISSN = {0303-1179,2492-5926},
      ISBN = {978-2-85629-871-8},
   MRCLASS = {22E50 (11F67)},
  MRNUMBER = {3764130},
MRREVIEWER = {Arnab\ Mitra},
}

@misc{travkin2026iwahorigaiottoconjecturecotangent,
      title={Iwahori Gaiotto conjecture for cotangent type}, 
      author={Roman Travkin and Ruotao Yang},
      year={2026},
      eprint={2610.08243},
      archivePrefix={arXiv},
      primaryClass={math.RT},
      url={https://arxiv.org/abs/2610.08243}, 
}

@misc{chen2026relativelanglandsdualitylinear,
      title={Relative Langlands duality for linear periods}, 
      author={Tsao-Hsien Chen and Sanath K. Devalapurkar and Mark Macerato and David Nadler},
      year={2026},
      eprint={2610.31771},
      archivePrefix={arXiv},
      primaryClass={math.RT},
      url={https://arxiv.org/abs/2610.31771}, 
}

@article{beraldo2017loop,
  title={Loop group actions on categories and Whittaker invariants},
  author={Beraldo, Dario},
  journal={Advances in Mathematics},
  volume={322},
  pages={565--636},
  year={2017},
  publisher={Elsevier}
}

@article{drinfeld2013some,
  title={On some finiteness questions for algebraic stacks},
  author={Drinfeld, Vladimir and Gaitsgory, Dennis},
  journal={Geometric and Functional Analysis},
  volume={23},
  number={1},
  pages={149--294},
  year={2013},
  publisher={Springer}
}

@book{milneAlgGrp,
    title = {Algebraic Groups: The Theory of Group Schemes of Finite Type over a Field},
    author = {J. S. Milne},
    url = {https://www.cambridge.org/us/universitypress/subjects/mathematics/geometry-and-topology/algebraic-groups-theory-group-schemes-finite-type-over-field?format=PB#contentsTabAnchor.},
    year = 2012   
}

@article{chen2023langlands,
  title={A Langlands dual realization of coherent sheaves on the nilpotent cone},
  author={Chen, Harrison and Dhillon, Gurbir},
  journal={arXiv preprint arXiv:2310.10539},
  year={2023}
}

@article{BR23,
title={Normality of {S}chubert varieties in affine {G}rassmannians},
author={Bieker, Patrick and Richarz, Timo},
journal={arXiv preprint arXiv:2312.11154},
year={2023}
}

@article{bezrukavnikov2008equivariant,
  title={Equivariant Satake category and Kostant--Whittaker reduction},
  author={Bezrukavnikov, Roman Vladimirovich and Finkel'berg, Mikhail Vladlenovich},
  journal={Moscow Mathematical Journal},
  volume={8},
  number={1},
  pages={39--72},
  year={2008},
  publisher={Независимый Московский университет--МЦНМО}
}

@article{ben2013nonlinear,
  title={Nonlinear traces},
  author={Ben-Zvi, David and Nadler, David},
  journal={arXiv preprint arXiv:1305.7175},
  year={2013}
}

@article{BFT_orthosymplectic_satake,
   title={Orthosymplectic Satake equivalence},
   volume={16},
   ISSN={1931-4531},
   url={http://dx.doi.org/10.4310/CNTP.2022.v16.n4.a2},
   DOI={10.4310/cntp.2022.v16.n4.a2},
   number={4},
   journal={Communications in Number Theory and Physics},
   publisher={International Press of Boston},
   author={Braverman, Alexander and Finkelberg, Michael and Travkin, Roman},
   year={2022},
   pages={695–732} }

@article{bezrukavnikov2016two,
  title={On two geometric realizations of an affine Hecke algebra},
  author={Bezrukavnikov, Roman},
  journal={Publications math{\'e}matiques de l'IH{\'E}S},
  volume={123},
  number={1},
  pages={1--67},
  year={2016},
  publisher={Springer}
}

@article{zhu2016introduction,
  title={An introduction to affine Grassmannians and the geometric Satake equivalence},
  author={Zhu, Xinwen},
  journal={arXiv preprint arXiv:1603.05593},
  year={2016}
}

@article{arkhipov2004quantum,
  title={Quantum groups, the loop Grassmannian, and the Springer resolution},
  author={Arkhipov, Sergey and Bezrukavnikov, Roman and Ginzburg, Victor},
  journal={Journal of the American Mathematical Society},
  volume={17},
  number={3},
  pages={595--678},
  year={2004}
}

@article{ben2010integral,
  title={Integral transforms and Drinfeld centers in derived algebraic geometry},
  author={Ben-Zvi, David and Francis, John and Nadler, David},
  journal={Journal of the American Mathematical Society},
  volume={23},
  number={4},
  pages={909--966},
  year={2010}
}

@article{campbell2021affine,
  title={Affine Harish-Chandra bimodules and Steinberg--Whittaker localization},
  author={Campbell, Justin and Dhillon, Gurbir},
  journal={arXiv preprint arXiv:2108.02806},
  year={2021}
}

@online{s2024building,
  author       = {S., Erik and Zhang, Barry},
  title        = {Building Effective Agents},
  year         = {2024},
  month        = dec,
  day          = {19},
  organization = {Anthropic},
  url          = {https://www.anthropic.com/research/building-effective-agents},
  urldate      = {2026-11-14}
}

@misc{feng2026autonomousmathematicsresearch,
      title={Towards Autonomous Mathematics Research}, 
      author={Tony Feng and Trieu H. Trinh and Garrett Bingham and Dawsen Hwang and Yuri Chervonyi and Junehyuk Jung and Joonkyung Lee and Carlo Pagano and Sang-hyun Kim and Federico Pasqualotto and Sergei Gukov and Jonathan N. Lee and Junsu Kim and Kaiying Hou and Golnaz Ghiasi and Yi Tay and YaGuang Li and Chenkai Kuang and Yuan Liu and Hanzhao Lin and Evan Zheran Liu and Nigamaa Nayakanti and Xiaomeng Yang and Heng-Tze Cheng and Demis Hassabis and Koray Kavukcuoglu and Quoc V. Le and Thang Luong},
      year={2026},
      eprint={2602.10177},
      archivePrefix={arXiv},
      primaryClass={cs.LG},
      url={https://arxiv.org/abs/2602.10177}, 
}

@book{hotta2008dmodules,
  title={D-modules, perverse sheaves, and representation theory},
  author={Hotta, Ryoshi and Takeuchi, Kiyoshi and Tanisaki, Toshiyuki},
  series={Progress in Mathematics},
  volume={236},
  year={2008},
  publisher={Birkh{\"a}user Boston}
}

@article{richardson1982orbits,
  title={Orbits, invariants, and representations associated to involutions of reductive groups},
  author={Richardson, Roger W.},
  journal={Inventiones Mathematicae},
  volume={66},
  number={2},
  pages={287--312},
  year={1982}
}

@misc{LinPhamYu_generation,
      title={Spherical and Iwahori generation on loop spaces of spherical varieties},
      author={Lin, Milton and Pham, Toan and Yu, Jize},
      note={In preparation},
      year={2026},
}

@article{sakellaridis2013spherical,
  author={Sakellaridis, Yiannis},
  title={Spherical functions on spherical varieties},
  journal={Amer. J. Math.},
  volume={135},
  number={5},
  pages={1291--1381},
  year={2013},
  eprint={0905.4244},
  archivePrefix={arXiv},
}

@article{sanders2022finite_etale,
  author={Sanders, Beren},
  title={A characterization of finite \'etale morphisms in tensor triangular geometry},
  journal={\'Epijournal de G\'eom\'etrie Alg\'ebrique},
  volume={6},
  year={2022},
  pages={Art. 18},
  doi={10.46298/epiga.2022.volume6.7641},
  eprint={2106.14066},
  eprinttype={arXiv},
}

\end{document}